\documentclass[12pt]{article}
\usepackage[utf8]{inputenc}
\usepackage[letterpaper,top=2cm,bottom=2cm,left=3cm,right=3cm,marginparwidth=1.75cm]{geometry}
\usepackage{graphicx}
\usepackage{graphics}
\usepackage{amsfonts}
\usepackage{amscd}
\usepackage{amsthm}
\usepackage{indentfirst}
\usepackage[all]{xy}
\usepackage{amssymb, amsmath, amsthm, amsgen, amstext, amsbsy, amsopn,dsfont}
\usepackage{epic,eepic}
\usepackage{enumitem} 
\usepackage{color}
\usepackage[dvipsnames]{xcolor}
\usepackage{ dsfont }
\usepackage[colorlinks=true, allcolors=blue]{hyperref}
\usepackage{comment}
\usepackage{mathtools}
\usepackage{transparent}
\usepackage{tikz-cd}

\theoremstyle{plain}
\newtheorem{thm}{Theorem}[section]
\newtheorem{question}[thm]{Question}
\newtheorem{prop}[thm]{Proposition}

\newtheorem{lem}[thm]{Lemma}
\newtheorem{claim}[thm]{Claim}

\theoremstyle{definition}
\newtheorem{definition}[thm]{Definition}
\newtheorem{rem}[thm]{Remark}

\newtheorem{example}[thm]{Example}

\numberwithin{equation}{section}
\numberwithin{figure}{section}

\newcommand{\R}{{\mathbb{R}}}

\newcommand{\T}{{\mathbb{T}}}
\newcommand{\Z}{{\mathbb{Z}}}
\newcommand{\N}{{\mathbb{N}}}

\newcommand{\D}{{\mathbb{D}}}

\newcommand{\Ham}{\operatorname{Ham}}

\title{A universal extension of helicity to topological flows}
\author{Oliver Edtmair and Sobhan Seyfaddini}
\date{\today}

\begin{document}


\maketitle

\begin{abstract}
Helicity is a fundamental conserved quantity in physical systems governed by vector fields whose evolution is described by volume-preserving transformations on a three-manifold. Notable examples include inviscid, incompressible fluid flows, modeled by the three-dimensional Euler equations, and conducting plasmas, described by the magnetohydrodynamics (MHD) equations.

A key property of helicity is its invariance under volume-preserving diffeomorphisms. In an influential article from 1973, Arnold---having provided an ergodic interpretation of helicity as the ``asymptotic Hopf invariant''---posed the question of whether this invariance persists under volume-preserving homeomorphisms. More generally, he asked whether helicity can be extended to topological volume-preserving flows. We answer both questions affirmatively for flows without rest points.

Our approach reformulates Arnold’s question in the framework of what we call $C^0$ Hamiltonian structures. This perspective enables us to leverage recent developments in $C^0$ symplectic geometry, particularly results concerning the algebraic structure of the group of area-preserving homeomorphisms.
\end{abstract}

\tableofcontents

\section{Introduction}

Helicity is a fundamental conserved quantity in physical systems where divergence-free vector fields evolve under volume-preserving transformations on a three-manifold. Notable examples include vorticity fields in inviscid, incompressible fluids, governed by the three-dimensional Euler equations, and magnetic fields in plasmas, described by the magnetohydrodynamics (MHD) equations. Although the concept of helicity was introduced by Woltjer in the study of magnetohydrodynamics \cite{Woltjer}, the term ``helicity'' was coined by Moffatt \cite{moffatt69} in his work on hydrodynamics.  
Building on Moreau~\cite{moreau61}, Moffatt derived the conservation of helicity from the Helmholtz–Kelvin law of vorticity transport~\cite{Helmholtz, hel58, Thomson_1868}. Owing to its physical significance, helicity remains an active topic of research in experimental physics; see, for example, \cite{Scheeler}. For historical overviews, we refer to \cite{Moffatt_1981, Moff-T92, Moffatt14}.

This article addresses questions raised by Arnold \cite{arnold74} concerning topological properties of helicity.

\subsection{Helicity and Arnold's questions}

Let $(Y^3, \mu)$ be a closed smooth $3$-manifold equipped with a volume form $\mu$, and let $X$ be a smooth vector field on $Y$ that preserves $\mu$. That is, the flow $\varphi^t_X$ of $X$ satisfies $(\varphi^t_X)^* \mu = \mu$ for all $t$. This condition is equivalent to the $2$-form
\begin{equation*}
\omega \coloneqq \iota_X \mu
\end{equation*}
being closed. The vector field $X$ is said to be \textit{exact} if $\omega$ is exact.

The \textit{helicity} of an exact volume-preserving vector field $X$ is defined as
\begin{equation*}
    \mathcal{H}(X) \coloneqq \int_Y \alpha \wedge d\alpha,
\end{equation*}
where $\alpha$ is any primitive $1$-form of $\omega$ and $Y$ is oriented via the volume form $\mu$. This integral turns out to be independent of the choice of primitive $1$-form $\alpha$.

A simple but fundamental property of helicity is that it is preserved under volume- and orientation-preserving diffeomorphisms; that is, 
\begin{equation*}
    \mathcal{H}(f_*X) = \mathcal{H}(X) \qquad \text{for every $f\in \operatorname{Diff}^+(Y,\mu)$.}
\end{equation*}
In fact, any other functional on the space of exact volume-preserving vector fields that is invariant under the action of $\operatorname{Diff}^+(Y,\mu)$ and satisfies certain natural regularity conditions must be a function of helicity \cite{Khesin22, Enciso16}.

Two volume-preserving smooth vector fields $X_1$ and $X_2$ on $(Y,\mu)$ are said to be \textit{topologically conjugate} if their flows are conjugate via a volume- and orientation-preserving homeomorphism, i.e.\ if there exists $f\in \operatorname{Homeo}^+(Y,\mu)$ such that
\begin{equation*}
    f \circ \varphi_{X_1}^t \circ f^{-1} = \varphi_{X_2}^t \qquad \text{for all $t \in \mathbb{R}$.}
\end{equation*}
For examples of volume-preserving smooth vector fields that are topologically conjugate but not smoothly (or even $C^1$) conjugate, see \cite[\S 10 \& 11]{Muller13}.

In his 1973 article \cite{arn86}, Arnold, having derived an ergodic interpretation of helicity, posed the following questions concerning the topological invariance of helicity.

\begin{question}
\label{ques:topological_invariance}
    Let $X_1$ and $X_2$ be two exact volume-preserving smooth vector fields which are topologically conjugate. Is it true that $\mathcal{H}(X_1) = \mathcal{H}(X_2)$?
\end{question}

\begin{question}
\label{ques:extension}
    Does helicity admit an extension to topological volume-preserving flows? Here, a topological flow refers to a continuous flow and is not necessarily generated by a vector field.
\end{question}

In this paper we address the above questions for fixed-point-free volume-preserving flows, which is the setting considered by Ghys in \cite[\S 5.4]{ghy23}.  

We provide a preliminary version of our results here, postponing the statement of our main result, Theorem \ref{thm:coherent_helicity_extension}, to Section \ref{sec:main-result}, as it requires some preparatory material.

\begin{thm}
\label{thm:topological_invariance_helicity_flows}
    Two nowhere vanishing, exact, volume-preserving, smooth vector fields which are topologically conjugate have the same helicity.

    Moreover, helicity admits an extension to fixed-point-free, exact, volume-preserving, topological flows whose flow lines have zero measure. This extension is invariant under conjugation by volume- and orientation-preserving homeomorphisms and compatible with the Calabi invariant in the sense of Eq. \eqref{eq:helicity_and_calabi_flows}, and it is uniquely determined by these properties.
\end{thm}

We briefly comment on the assumptions in the theorem. The notion of exactness for topological flows generalizes its smooth counterpart and is essential for defining helicity, even in the smooth setting; see Definitions \ref{def:flux_C0} \& \ref{def:exact_vol_pres}. The condition that flow lines have zero measure is quite natural; it rules out certain pathological situations (see Example \ref{ex:atomic_transverse_measure}). The same assumption appears also in other works \cite{gg97,Contreras-Iturriaga99}. The significance of the fixed-point-free condition is that it yields a \emph{$C^0$ Hamiltonian structure}, enabling us to build on recent advances in $C^0$ symplectic topology and the understanding of the algebraic structure of groups of area-preserving homeomorphisms \cite{chs24, chmss22, chmss}.

\medskip
 
\noindent \emph{Compatibility with Calabi.} We briefly explain this here, leaving further details to Sections  \ref{sec:plugs and Calabi} \& \ref{sec:main-result}.  Let $(\Sigma, \omega_\Sigma)$ denote an open surface equipped with an area form $\omega_\Sigma$. We denote its group of Hamiltonian diffeomorphisms by $\Ham(\Sigma)$. It was proven recently \cite{chmss, mt} that the Calabi homomorphism $\mathrm{Cal}_\Sigma: \Ham(\Sigma) \rightarrow \R$ admits infinitely many extensions to the group of Hamiltonian homeomorphisms $\overline{\Ham}(\Sigma)$.  Pick one such extension
\begin{equation*}
    \overline{\mathrm{Cal}}_\Sigma : \overline{\Ham}(\Sigma) \rightarrow \R.
\end{equation*}
The Calabi extensions for different surfaces $\Sigma$ can be picked such that they satisfy the naturality property \eqref{eqn:naturality_topological} below; we prove this fact, which is of independent interest and also crucial for our arguments, in Theorem \ref{thm:natural_isomorphisms_abelianizations_hamiltonian_homeomorphisms}.

Now, fix a topological volume-preserving flow $\psi^t$ on $(Y, \mu)$ and suppose that we have a topological volume-preserving embedding 
\begin{equation*}
    \alpha : ((0,1)\times \Sigma), dt \wedge \omega_\Sigma) \hookrightarrow (Y,\mu)
\end{equation*}
which intertwines the flow on $(0,1)\times \Sigma $ generated by the vector field $\partial_t$ and the flow $\psi^t$.  Consider a $C^0$ Hamiltonian isotopy $\varphi^t \in \overline{\Ham}(\Sigma)$ and note that its suspension to $(0,1) \times \Sigma$ is volume preserving.  We refer to the tuple $\mathcal{P} := (\Sigma,\omega_\Sigma,\alpha,\varphi^t)$  as a  \emph{plug}.  Given a plug $\mathcal{P}$, one can define a new volume-preserving flow $\psi^t \# \mathcal{P}$ on $Y$ by replacing the flow $\psi^t$ inside $\operatorname{im}(\alpha)$ with the suspension of $\varphi^t$. The compatibility condition between our helicity extension $\overline{\mathcal{H}}$ and the Calabi extension $\overline{\operatorname{Cal}}_\Sigma$ is given by
\begin{equation}
    \label{eq:helicity_and_calabi_flows}
    \overline{\mathcal{H}}(\psi^t\#\mathcal{P}) = \overline{\mathcal{H}}(\psi^t) + \overline{\mathrm{Cal}}_\Sigma(\varphi^1).
\end{equation}

\subsection{Context and Historical Background}

Our story begins in 1973, when Arnold \cite{arnold74} interpreted the helicity of a vector field as the average asymptotic linking number of flow lines, which motivated Questions \ref{ques:topological_invariance} \& \ref{ques:extension}; see Sections 2-4 of the English translation \cite{arn86} and Problem 1973-23 in \cite{Arnold-problems}. Two trajectories, beginning at two randomly chosen points in space, are followed for a long time and then closed into loops using a well-chosen system of short geodesic arcs.\footnote{The existence of the system of geodesic arcs is a subtle point which was proven rigorously in \cite{Vogel2003}.} The linking number of these loops---averaged over time and all pairs of initial points---converges to the helicity as the time tends to infinity. Since linking numbers are preserved under homeomorphisms, at first glance this seems to imply topological invariance of helicity. However, as remarked by Ghys in \cite{ghy07}, ``... one should be cautious that a homeomorphism might entangle the small geodesic arcs that were used to close the trajectories."  

Since their formulation, Questions \ref{ques:topological_invariance} and \ref{ques:extension} have reappeared in various works \cite{ak21, gg97, ghy07, tao_blog, Muller13}, including Arnold and Khesin’s textbook \cite[III, Problem 4.8]{ak21}, Ghys’ plenary ICM address \cite[Section 1.4]{ghy07}, and Tao’s blog \cite{tao_blog}. The case of flows without rest points is emphasized again in \cite[Sec.\ 5.4]{ghy23}.  One reason for this ongoing interest is that real-world flows often lack smoothness. Therefore, a positive answer to these questions would highlight the significance of helicity as a meaningful invariant, even in low-regularity settings. As Tao notes \cite{tao_blog}, “This would be of interest in fluid equations, as it would suggest that helicity remains invariant even after the development of singularities in the flow.” 

We should mention that Arnold's influential article has prompted substantial further work in various directions; see, for example, \cite{freedman_he_91, Contreras93, gg97, Contreras-Iturriaga99, Riv02, Khesin03, Kot-Vog03, Muller13, Khesin-Saldanha}. Among these,  Gambaudo \& Ghys \cite{gg97} and Müller \& Spaeth \cite{Muller13} contain results towards Questions \ref{ques:topological_invariance} \& \ref{ques:extension}. Both cases in \cite{gg97, Muller13} involve exact volume-preserving flows without fixed points, which are covered by Theorem \ref{thm:topological_invariance_helicity_flows}.  The Gambaudo–Ghys approach establishes the topological invariance of helicity for certain suspension flows by relating it to the Calabi invariant---a connection that is also central to our work.

The extension of helicity to low-regularity settings is the subject of ongoing  research in fluid dynamics: In \cite{Giri-Kwon-Novack}, Giri, Kwon and Novack extend helicity to a class of weak solutions of the Euler equations with regularity $H^{\frac{1}{2} - \varepsilon}$. Using convex integration, they further show that helicity need not be conserved for such solutions. This does not contradict our result, since the vector fields they construct are too irregular to generate a flow. It would be interesting to compare their extension of helicity with ours on the overlap of their respective domains.

\subsection{Hamiltonian structures}
\label{sec:intro-Ham-structures}

\noindent \textbf{Smooth Hamiltonian Structures.} 
Let $Y$ be an oriented smooth $3$-manifold.  A \textit{Hamiltonian structure} on $Y$ is a closed and maximally nondegenerate $2$-form $\omega$ on $Y$. 
The prototypical example is that of the \emph{standard Hamiltonian structure} on $\R^3$, which is given by the $2$-form $\omega_{\operatorname{std}}:= dx\wedge dy$, where $\R^3 = \R \times \R^2$ is equipped with the coordinates $(t,x,y)$.  Every smooth Hamiltonian structure is locally diffeomorphic to $\omega_{\operatorname{std}}$; see Lemma \ref{lem:Draboux}.

Maximal non-degeneracy of $\omega$ means that $\ker \omega$ defines a line field on $Y$, called the \textit{characteristic line field}, which integrates to a  $1$-dimensional foliation on $Y$, called the \textit{characteristic foliation}. The Hamiltonian structure $\omega$ naturally equips this foliation with a coorientation and a transverse measure locally diffeomorphic to the standard $2$-dimensional Lebesgue measure. Conversely, every such foliation uniquely determines a Hamiltonian structure.

The relevance of Hamiltonian structures to our discussion is as follows: Equip $Y$ with a volume form $\mu$. Then, a vector field $X$ is nowhere vanishing and volume-preserving if and only if \( \omega \coloneqq \iota_X \mu \) is a Hamiltonian structure. In this case, the flow $\varphi_X^t$ is tangent to the characteristic foliation of $\omega$.

We say a Hamiltonian structure $\omega$ is \textit{exact} if $\omega$ is an exact $2$-form. We define the \textit{helicity} of an exact Hamiltonian structure on an oriented closed smooth $3$-manifold $Y$ to be
\begin{equation*}
    \mathcal{H}(\omega) \coloneqq \int_Y \alpha \wedge d\alpha,
\end{equation*}
where $\alpha$ is any choice of primitive $1$-form of $\omega$. As already mentioned, this integral is independent of the choice of $\alpha$. By definition, a nowhere vanishing volume-preserving vector field $X$ is exact if and only if the induced Hamiltonian structure $\omega$ is exact. If this is the case, we have
\begin{equation*}
    \mathcal{H}(X) = \mathcal{H}(\omega).
\end{equation*}
This perspective allows us to reformulate Question \ref{ques:topological_invariance} concerning the topological invariance of helicity in terms of Hamiltonian structures. Specifically, the question becomes whether two exact Hamiltonian structures related by pullback under an orientation-preserving homeomorphism necessarily have the same helicity. To make sense of pullbacks by homeomorphisms, we rely on the interpretation of Hamiltonian structures as cooriented measured foliations.

\bigskip

\noindent \textbf{$C^0$ Hamiltonian Structures.} Let $Y$ be an oriented topological $3$-manifold. A \emph{$C^0$ Hamiltonian structure} on $Y$ is a cooriented, $1$-dimensional $C^0$ foliation equipped with a transverse measure, locally modeled on $\R^3$ with the characteristic foliation induced by $\omega_{\operatorname{std}}$. A more detailed definition, formalized in terms of $C^0$ Hamiltonian atlases, is given in Section~\ref{sec:C0_hamiltonian_structures}. We will see in Section~\ref{sec:c0_case} that $C^0$ Hamiltonian structures serve the same role for topological volume-preserving flows (whose flow lines have zero measure) as smooth Hamiltonian structures do for smooth volume-preserving flows. This allows us to recast Question~\ref{ques:extension} as a question about extending helicity to exact\footnote{We explain how to interpret exactness for $C^0$ Hamiltonian structures in Section~\ref{sec:flux_C0_Ham_structures}.} $C^0$ Hamiltonian structures.

\subsection{Plugs and the Calabi invariant in the \texorpdfstring{$C^0$}{C0} setting}
\label{sec:plugs and Calabi}

In the smooth setting, Gambaudo--Ghys \cite{gg97} discovered a natural connection between helicity and the Calabi invariant. We construct our extensions of helicity such that this connection continues to hold in the $C^0$ setting.

Let $(\Sigma,\omega_\Sigma)$ be an open surface with an area form, and let $\Ham(\Sigma)$ and $\overline{\Ham}(\Sigma)$ denote, respectively, the groups of compactly supported Hamiltonian diffeomorphisms and homeomorphisms of $(\Sigma,\omega_\Sigma)$; see Section \ref{sec:prelims-surfaces}. We denote by $\operatorname{Cal}_\Sigma : \Ham(\Sigma) \rightarrow \R$ the Calabi homomorphism \cite{cal69}, which is defined as follows: for $\varphi = \varphi_H^1 \in \operatorname{Ham}(\Sigma)$, generated by a compactly supported Hamiltonian $H \in C^\infty( [0,1]\times \Sigma )$, we have\footnote{Different conventions exist for defining the Calabi invariant; we include the factor 2 for convenience.} 
\begin{equation*}
\operatorname{Cal}_\Sigma(\varphi) = 2 \int_{[0,1]\times \Sigma} H dt \wedge \omega_\Sigma.
\end{equation*}

The $2$-form $\omega_\Sigma$ induces a natural Hamiltonian structure on $(0,1)\times \Sigma $, which we denote by the same symbol $\omega_\Sigma$. Given an oriented $3$-manifold $Y$ equipped with a $C^0$ Hamiltonian structure $\Omega$, we define a \emph{plug} to be a tuple $$\mathcal{P} \coloneqq (\Sigma,\omega_\Sigma,\alpha,(\varphi^t)_{t\in [0,1]})$$ 
where $\alpha : ((0,1)\times \Sigma),\omega_\Sigma) \hookrightarrow (Y,\Omega)$ is an embedding of $C^0$ Hamiltonian structures and $\varphi^t \in \overline{\Ham}(\Sigma)$ is a $C^0$ Hamiltonian isotopy. Given a plug $\mathcal{P}$, we define an operation, called \emph{plug insertion}, which creates a new $C^0$ Hamiltonian structure $\Omega \# \mathcal{P}$ on $Y$ by replacing the $C^0$ Hamiltonian structure $\Omega$ inside $\operatorname{im}(\alpha)$ with the pushforward of the $C^0$ Hamiltonian structure on $(0,1)\times \Sigma$ induced by $\omega_\Sigma$ via the embedding $\alpha \circ \Phi$, where $\Phi : (0,1)\times \Sigma \rightarrow (0,1)\times \Sigma$ is defined by $\Phi(t,p) \coloneqq (t,\varphi^t(p))$.\footnote{Plug insertion does not affect exactness of Hamiltonian structures, in smooth and $C^0$ settings.}  The effect of this operation on the characteristic foliations is as follows:  Before plug insertion, the characteristic leaves of $\Omega$ inside $\operatorname{im}(\alpha)$ are of the form $\alpha((0,1)\times \{p\})$ for $p\in \Sigma$. After plug insertion, the characteristic leaves of $\Omega\# \mathcal{P}$ are of the form
\begin{equation*}
    \{\alpha(t,\varphi^t(p))\mid t\in (0,1)\}
\end{equation*}
for $p \in \Sigma$.

In the smooth setting, the following important identity, which can be deduced from \cite{gg97}, relates the helicities of $\Omega$ and $\Omega \# \mathcal{P}$:
\begin{equation}
\label{eq:helicity_and_calabi_smooth}
    \mathcal{H}(\Omega\#\mathcal{P}) = \mathcal{H}(\Omega) + \operatorname{Cal}(\varphi^1)
\end{equation}

The Calabi homomorphism was recently extended to the group of Hamiltonian homeomorphisms \cite{chmss, mt}. In light of this, it is natural to require that identity~\eqref{eq:helicity_and_calabi_smooth} continues to hold for an extension of helicity to $C^0$ Hamiltonian structures. However, in the smooth setting, the Calabi homomorphism satisfies a certain naturality property---implicit in identity~\eqref{eq:helicity_and_calabi_smooth}---which was not known to hold in the $C^0$ category. Below, we establish this naturality in the $C^0$ setting via Theorem~\ref{thm:natural_isomorphisms_abelianizations_hamiltonian_homeomorphisms}—a result of independent interest that plays a key role in our main theorem.

\subsubsection*{A universal extension of Calabi.} 

In this subsection, $\Sigma, \Sigma_1, \Sigma_2$ denote non-empty, open, connected surfaces equipped with area forms. A smooth area- and orientation-preserving embedding $\iota: \Sigma_1 \hookrightarrow \Sigma_2$ induces an injective homomorphism $\operatorname{Ham}(\iota): \operatorname{Ham}(\Sigma_1) \rightarrow \operatorname{Ham}(\Sigma_2)$ via pushforward. The Calabi homomorphism satisfies the following naturality property:
\begin{equation}\label{eqn:naturality_smooth}
\operatorname{Cal}_{\Sigma_1} = \operatorname{Cal}_{\Sigma_2} \circ \operatorname{Ham}(\iota)
\end{equation}
Note, moreover, that the embedding $\iota$ induces a homomorphism $\operatorname{Ham}^{\operatorname{ab}}(\iota)$ at the level of abelianizations
\begin{equation}
\label{eqn:iso_abelianization_smooth}
\operatorname{Ham}^{\operatorname{ab}}(\iota) : \operatorname{Ham}^{\operatorname{ab}}(\Sigma_1) \rightarrow \operatorname{Ham}^{\operatorname{ab}}(\Sigma_2).
\end{equation}
Here $G^{\operatorname{ab}}\coloneqq G/[G,G]$ denotes the abelianization of a group $G$. It follows from Banyaga's work \cite{ban78} that the homomorphism \eqref{eqn:iso_abelianization_smooth} is an isomorphism and is independent of the embedding $\iota$. Hence the group $\operatorname{Ham}^{\operatorname{ab}}(\Sigma)$ does not depend on $\Sigma$ up to canonical isomorphism. In fact, the Calabi homomorphism induces a natural isomorphism
\begin{equation*}
    \operatorname{Cal}_\Sigma : \operatorname{Ham}^{\operatorname{ab}}(\Sigma) \overset{\cong}{\longrightarrow} \R.
\end{equation*}

\medskip 

For our arguments, we need
extensions of the Calabi homomorphism
\begin{equation*}
    \overline{\operatorname{Cal}}_\Sigma : \overline{\operatorname{Ham}}(\Sigma) \rightarrow \R
\end{equation*}
that satisfy an analogue of the naturality property \eqref{eqn:naturality_smooth} in the $C^0$ setting. Although \cite{chmss, mt} construct infinitely many $C^0$ extensions of Calabi, none are canonical,\footnote{The constructions in \cite{chmss, mt} rely on the axiom of choice in an essential way; see \cite[Remark 5.1]{chmss}.} and it is unclear whether they satisfy the desired properties.

We obtain suitable extensions of Calabi as a consequence of the theorem below. For every area- and orientation-preserving topological embedding of surfaces $\iota: \Sigma_1\hookrightarrow \Sigma_2$, let $\overline{\operatorname{Ham}}^{\operatorname{ab}}(\iota) : \overline{\operatorname{Ham}}^{\operatorname{ab}}(\Sigma_1) \rightarrow \overline{\operatorname{Ham}}^{\operatorname{ab}}(\Sigma_2)$ denote the induced homomorphism at the level of abelianizations.

\begin{thm}
\label{thm:natural_isomorphisms_abelianizations_hamiltonian_homeomorphisms}
    The homomorphism
    \begin{equation*}
        \overline{\operatorname{Ham}}^{\operatorname{ab}}(\iota) : \overline{\operatorname{Ham}}^{\operatorname{ab}}(\Sigma_1) \rightarrow \overline{\operatorname{Ham}}^{\operatorname{ab}}(\Sigma_2)
    \end{equation*}
    is an isomorphism. Moreover, it does not depend on the choice of area- and orientation-preserving embedding $\iota: \Sigma_1 \hookrightarrow \Sigma_2$, i.e.\ if $\iota':\Sigma_1\hookrightarrow \Sigma_2$ is another area- and orientation-preserving embedding, then
    \begin{equation*}
        \overline{\operatorname{Ham}}^{\operatorname{ab}}(\iota') = \overline{\operatorname{Ham}}^{\operatorname{ab}}(\iota).
    \end{equation*}
\end{thm}

We state and prove a generalization of Theorem \ref{thm:natural_isomorphisms_abelianizations_hamiltonian_homeomorphisms} in Section \ref{sec:hamiltonian_homeos_and_calabi}; see Theorem \ref{thm:natural_isomorphisms_abelianizations_hamiltonian_homeomorphisms-v2}.

Set $\mathcal{R} \coloneqq \overline{\operatorname{Ham}}^{\operatorname{ab}}(\D)$, where $\D\subset \R^2$ denotes the open unit disc. By Theorem \ref{thm:natural_isomorphisms_abelianizations_hamiltonian_homeomorphisms}, we may canonically identify 
\begin{equation}\label{eqn:universal-abelianization} \overline{\operatorname{Ham}}^{\operatorname{ab}}(\Sigma) \cong \mathcal{R}
\end{equation}
for every non-empty, connected, open surface $\Sigma$.

It was an open question---known as the \emph{simplicity conjecture}---whether the group $\overline{\operatorname{Ham}}(\D)$ is simple. This question was recently resolved in the negative in \cite{chs24}, which showed that $\overline{\operatorname{Ham}}(\D)$ is not simple. Moreover, it is shown in \cite[Cor.\ 1.3]{chs24} that $\overline{\operatorname{Ham}}(\D)$ is not perfect either. In other words, the group $\mathcal{R}$ is non-trivial.

It can be further deduced from \cite{chmss, mt} that the natural homomorphism  
\begin{equation}
\label{eq:natural_map_abelianization_smooth_to_C0}
    \operatorname{Ham}^{\operatorname{ab}}(\Sigma) \rightarrow \overline{\operatorname{Ham}}^{\operatorname{ab}}(\Sigma)
\end{equation}
is injective and not surjective; see Section~\ref{sec:hamiltonian_homeos_and_calabi}. The injectivity of the homomorphism~\eqref{eq:natural_map_abelianization_smooth_to_C0}, a key input from $C^0$ symplectic topology, is essential for our proof of the topological invariance of helicity. Recalling that the Calabi homomorphism induces an identification $\operatorname{Ham}^{\operatorname{ab}}(\Sigma) \cong \R$, the homomorphism~\eqref{eq:natural_map_abelianization_smooth_to_C0} allows us to naturally view $\R$ as a subgroup
\begin{equation}
\label{eqn:subgroup}
    \R \subset \mathcal{R}.
\end{equation}

In view of the natural commutative diagram
\begin{equation*}
\begin{tikzcd}
\overline{\operatorname{Ham}}(\Sigma) \arrow[r] 
  & \overline{\operatorname{Ham}}^{\operatorname{ab}}(\Sigma) \arrow[r, "\cong"] 
  & \mathcal{R} \\
\operatorname{Ham}(\Sigma) \arrow[u, hookrightarrow] \arrow[r] 
  \arrow[rr, bend right=20, "\operatorname{Cal}_\Sigma"']
  & \operatorname{Ham}^{\operatorname{ab}}(\Sigma) \arrow[r, "\cong"] \arrow[u, hookrightarrow] 
  & \R\arrow[u, hookrightarrow]
\end{tikzcd}
\end{equation*}
we call the natural homomorphism
\begin{equation}
    \label{eqn:calabi_universal}
    \overline{\operatorname{Cal}}_\Sigma:\overline{\operatorname{Ham}}(\Sigma) \rightarrow \overline{\operatorname{Ham}}^{\operatorname{ab}}(\Sigma) \cong \mathcal{R}
\end{equation}
the \emph{universal} (or \emph{$\mathcal{R}$-valued}) extension of the Calabi homomorphism.\footnote{Note that this amounts to a change of notation from \eqref{eq:helicity_and_calabi_flows}. From this point on-wards, $\overline{\mathrm{Cal}}_\Sigma$ will always denote the universal extension of Calabi, unless otherwise stated.}\footnote{Recall that the standing assumption of this subsection is that surfaces are connected. We will encounter non-connected surfaces $\Sigma$ later on, and in this case we define $\overline{\mathrm{Cal}}_\Sigma : \overline{\operatorname{Ham}}(\Sigma) \rightarrow \mathcal{R}$ to be the sum of all $\overline{\mathrm{Cal}}_S$ where $S$ ranges over the components of $\Sigma$.} Its restriction to $\operatorname{Ham}(\Sigma)$ agrees with the usual smooth Calabi homomorphism $\operatorname{Cal}_\Sigma$ via the natural inclusion $\R\subset \mathcal{R}$. Moreover, by Theorem \ref{thm:natural_isomorphisms_abelianizations_hamiltonian_homeomorphisms} it satisfies the naturality property
\begin{equation}\label{eqn:naturality_topological}
\overline{\mathrm{{Cal}}}_{\Sigma_1} =\overline{\mathrm{{Cal}}}_{\Sigma_2} \circ \overline{\operatorname{Ham}}(\iota)
\end{equation}
for any area- and orientation-preserving topological embedding $\iota: \Sigma_1 \hookrightarrow \Sigma_2$.

The extension $\overline{\operatorname{Cal}}_\Sigma$ is universal in the following sense: Consider an arbitrary extension of abelian groups $\R \subset A$. Then every system of extensions 
\begin{equation*}
    \overline{\operatorname{Cal}}_\Sigma^A:\overline{\operatorname{Ham}}(\Sigma) \rightarrow A
\end{equation*}
of Calabi, one for every surface $\Sigma$ and subject to the naturality condition \eqref{eqn:naturality_topological}, arises as
\begin{equation*}
    \overline{\operatorname{Cal}}^A_\Sigma = p \circ \overline{\operatorname{Cal}}_\Sigma
\end{equation*}
for a unique group homomorphism $p:\mathcal{R} \rightarrow A$ over $\R$. In particular, for every choice of projection $\operatorname{pr}:\mathcal{R}\rightarrow \R$, we obtain a system of $\R$-valued Calabi extensions $\operatorname{pr}\circ \overline{\operatorname{Cal}}_\Sigma$ subject to \eqref{eqn:naturality_topological}, and every such system uniquely arises this way. Note that a projection $\operatorname{pr}: \mathcal{R}\rightarrow \R$ is equivalent to a choice of a single Calabi extension $\overline{\operatorname{Cal}}_\D^{\R}:\overline{\operatorname{Ham}}(\D)\rightarrow \R$ for the open unit disc $\D$. Thanks to \cite{chmss}, such extensions are abundant, but unfortunately not very canonical since they are found using the axiom of choice. In fact, as pointed out in \cite[Remark 5.1]{chmss}, there are models of set theory in which the axiom of choice is false and every group homomorphism between Polish groups is automatically continuous. Since there does not exist a $C^0$ continuous extension of the Calabi homomorphism to $\overline{\operatorname{Ham}}(\D)$, in these models there is no extension as a group homomorphism at all.

For this reason, we prefer to work with the universal Calabi extension and construct a universal $\mathcal{R}$-valued extension of helicity to exact $C^0$ Hamiltonian structures. It is always possible to obtain an $\R$-valued helicity extension from this if one wishes to, but this requires a non-canonical choice.

\subsection{Main Result: A universal extension of helicity}
\label{sec:main-result}

We are now in position to state our main result, which defines  helicity  $\overline{\mathcal{H}}(\Omega)$
for an arbitrary exact $C^0$ Hamiltonian structure $\Omega$ on an oriented closed topological $3$-manifold $Y$. The notion of exactness for $C^0$ Hamiltonian structures is defined as the vanishing of a certain cohomology class $\overline{\mathrm{Flux}}(\Omega) \in H^2(Y; \R)$. In particular, if $Y$ is a rational homology three-sphere, then every $C^0$ Hamiltonian structure is exact. This is discussed in detail in Section \ref{sec:flux_C0_Ham_structures}.

Recall that $\Omega\#\mathcal{P}$ denotes the plug insertion operation introduced in Section \ref{sec:plugs and Calabi}. Moreover, $f^* \Omega$ stands for the pullback of $\Omega$ under a homeomorphism $f$.

\begin{thm}
\label{thm:coherent_helicity_extension}
    There is a unique way of assigning a $\mathcal{R}$-valued helicity $\overline{\mathcal{H}}(\Omega)\in \mathcal{R}$ to every exact $C^0$ Hamiltonian structure $\Omega$ on an oriented closed topological $3$-manifold $Y$ such that the following conditions are satisfied: 
    \begin{enumerate}
        \item \label{item:coherent_helicity_extensions_smooth}  
          Extension:  If $\omega$ and $Y$ are smooth, then 
         \begin{equation*} \label{eq:topological_helicity_extends_smooth_helicity}
             \overline{\mathcal{H}}(\omega) = \mathcal{H}(\omega) \in \R,
        \end{equation*}
         where we view $\R$ as a subgroup of $\mathcal{R}$ via the natural inclusion \eqref{eqn:subgroup}.
        
        \item \label{item:coherent_helicity_extensions_invariance} Invariance: We have 
        \begin{equation*}
        \label{eq:topological_helicity_homeomorphism_invariance}
         \overline{\mathcal{H}}(f^*\Omega) = \overline{\mathcal{H}}(\Omega),
        \end{equation*}
         for any orientation-preserving homeomorphism $f$. 
        
        \item \label{item:coherent_helicity_extensions_plug} Calabi Compatibility: For every plug $\mathcal{P} = (\Sigma,\omega_\Sigma,\alpha,\varphi^t)$,  have
        \begin{equation}
         \label{eq:helicity_and_calabi_topological}
          \overline{\mathcal{H}}(\Omega\#\mathcal{P}) = \overline{\mathcal{H}}(\Omega) + \overline{\mathrm{Cal}}_\Sigma(\varphi^1).
        \end{equation}
    \end{enumerate}
\end{thm}

As already mentioned, every choice of projection $\operatorname{pr}:\mathcal{R}\rightarrow \R$ gives rise to a real-valued extension of helicity.

We point out that it follows from the Extension and Invariance properties in Theorem \ref{thm:coherent_helicity_extension} that two smooth Hamiltonian structures which are conjugated by an orientation-preserving homeomorphism have the same helicity. The proof of this makes essential use of the injectivity of the natural map $\R \cong \operatorname{Ham}^{\operatorname{ab}}(\D) \rightarrow \overline{\operatorname{Ham}}^{\operatorname{ab}}(\D) \cong \mathcal{R}$.

Let us elaborate further on the universality of our helicity extension $\overline{\mathcal{H}}$ and the naturality of the properties (Extension, Invariance, and Calabi) that characterize it in Theorem~\ref{thm:coherent_helicity_extension}. The properties Extension and Invariance can be viewed as minimal requirements that any reasonable extension of helicity to $C^0$ Hamiltonian structures ought to satisfy. In light of identity~\eqref{eq:helicity_and_calabi_smooth}, which describes the helicity change under plug insertion in the smooth setting, the Calabi property---though less obvious---is likewise a natural condition to impose.

To further motivate this property, observe that identity~\eqref{eq:helicity_and_calabi_smooth} has the following direct consequence: the helicity change $\mathcal{H}(\omega \# \mathcal{P}) - \mathcal{H}(\omega)$ resulting from the insertion of a smooth plug $\mathcal{P} = (\Sigma, \omega_\Sigma, \alpha, \varphi^t)$ depends only on the time-one map $\varphi^1$, and is independent of both the ambient Hamiltonian structure $\omega$ and the embedding $\alpha$ of the plug. We refer to this fundamental property as \emph{plug homogeneity}. Remarkably, imposing plug homogeneity---together with Extension and Invariance---already forces any helicity extension to be a function of our universal extension $\overline{\mathcal{H}}$.

\begin{prop}
\label{prop:universality}
    Let $\R\subset A$ be an extension of abelian groups. Let $\overline{\mathcal{H}}^A(\Omega)\in A$ be an $A$-valued extension of helicity to exact $C^0$ Hamiltonian structures satisfying the Extension and Invariance properties in Theorem \ref{thm:coherent_helicity_extension}. Moreover, assume that $\overline{\mathcal{H}}^A$ satisfies plug homogeneity. That is, for any $C^0$ Hamiltonian structure $\Omega$ and for any $\Omega$-plug $\mathcal{P}=(\Sigma,\omega_\Sigma,\alpha,\varphi^t)$, the helicity difference
    \begin{equation*}
        \overline{\mathcal{H}}^A(\Omega\# \mathcal{P}) - \overline{\mathcal{H}}^A(\Omega) \in A
    \end{equation*}
    only depends on $\varphi^1$ and is independent of $\Omega$ and $\alpha$. Then there exists a unique homomorphism $p:\mathcal{R} \rightarrow A$ over $\R$ such that $\overline{\mathcal{H}}^A= p\circ \overline{\mathcal{H}}$.
\end{prop}

This means that, once plug homogeneity is accepted as a fundamental property of helicity, one is essentially forced to arrive at our universal $\mathcal{R}$-valued helicity extension $\overline{\mathcal{H}}$. Moreover, insisting on plug homogeneity makes the problem of defining an $\R$-valued extension of helicity equivalent to specifying a projection $\operatorname{pr}:\mathcal{R}\rightarrow \R$, which, as explained in Subsection \ref{sec:plugs and Calabi}, cannot be done without making non-canonical choices. This indicates that $\mathcal{R}$ is indeed the natural target group for a helicity extension to exact $C^0$ Hamiltonian structures.

\begin{rem}
    We end our introduction with the following remarks.
    \begin{enumerate}
        \item The $\mathcal{R}$-valued helicity extension from Theorem \ref{thm:coherent_helicity_extension} provides an obstruction to smoothability of $C^0$ Hamiltonian structures: if $\overline{\mathcal{H}}(\Omega)$ lies in $\mathcal{R}\setminus \mathbb{R}$, then $\Omega$ is not homeomorphic to any smooth Hamiltonian structure.
        \item In a sequel to this work, we will investigate helicity from the point of view of characteristic classes of foliations and Haefliger structures.
        \item While our results affirmatively answer Arnold’s Questions \ref{ques:topological_invariance} and \ref{ques:extension} for flows without fixed points, they remain open for flows with fixed points. In future work, we will address topological invariance of helicity in the presence of certain types of singularities, in particular the generic case of non-degenerate singularities. At present, it is unclear whether sufficiently complicated singular sets may destroy topological invariance.
    \end{enumerate}
\end{rem}

\noindent \textbf{Structure of the paper.}

In Section~\ref{sec:prelims-surfaces}, we review some preliminaries concerning area-preserving homeomorphisms, foliations, and transverse measures.

In Section~\ref{sec:hamiltonian_homeos_and_calabi}, we prove a generalization of Theorem~\ref{thm:natural_isomorphisms_abelianizations_hamiltonian_homeomorphisms}, which leads to the construction of our universal $\mathcal{R}$-valued extension of the Calabi homomorphism satisfying the naturality property stated in Section~\ref{sec:plugs and Calabi}.

In Sections~\ref{sec:smooth_hamiltonian_structures} and~\ref{sec:C0_hamiltonian_structures}, we provide precise definitions of smooth and $C^0$ Hamiltonian structures, along with the plug insertion operation.

Section~\ref{sec:smoothings-mod-plugs} presents two central results---Theorems~\ref{thm:from_C0_to_smooth_plus_plug} and~\ref{thm:smoothing_homeomorphisms}---which describe the structure of $C^0$ Hamiltonian structures and their homeomorphisms. These theorems roughly state that such structures and maps can be smoothed up to the insertion of a carefully constructed plug. 

In Section~\ref{sec:flux and helicity extension}, we extend the definitions of flux and helicity to $C^0$ Hamiltonian structures, thereby completing the proof of Theorem~\ref{thm:coherent_helicity_extension}. We also provide a proof of Proposition \ref{prop:universality}.

Finally, in Section~\ref{sec:relationship_hamiltonian_structures_vol_pres_flows}, we carefully explain the relationship between Hamiltonian structures and volume-preserving flows, both in the smooth and in the $C^0$ setting. This is essential for deducing Theorem~\ref{thm:topological_invariance_helicity_flows}, which is formulated for topological volume-preserving flows, from our main result, Theorem~\ref{thm:coherent_helicity_extension}, stated for $C^0$ Hamiltonian structures.

\bigskip

\noindent \textbf{Convention:} Throughout this paper, all manifolds will be oriented and all (local) homeomorphisms between manifolds will be orientation preserving, unless specified otherwise.

\subsection*{Acknowledgments}

We are grateful to Vikram Giri and Hyunju Kwon for helpful explanations about their work \cite{Giri-Kwon-Novack}.  We thank \'Etienne Ghys, Boris Khesin and Chi Cheuk Tsang for their interest and comments.  S.S. is indebted to Boris Khesin for introducing him to Arnold's questions and for enlightening discussions.

O.E. is supported by Dr.\ Max R\"{o}ssler, the Walter Haefner Foundation, and the ETH Z\"{u}rich Foundation. S.S. is partially supported by ERC Starting Grant number 851701 and a start up grant from ETH-Zürich.

\section{Preliminaries}
\label{sec:prelims-surfaces}

We recall some preliminaries concerning area-preserving homeomorphisms and foliations, which will be needed in the forthcoming sections.

\subsection{Area-preserving homeomorphisms}

Let $(\Sigma,\omega)$ be a smooth $2$-manifold equipped with an area form. We assume that $\Sigma$ does not have boundary, but we allow it to be open. 

Every compactly supported smooth Hamiltonian $H: [0,1]\times \Sigma \rightarrow \R$ induces a time-dependent Hamiltonian vector field $X_H$ generating a compactly supported Hamiltonian isotopy $(\varphi_H^t)_{t\in [0,1]}$. We adopt the sign convention that $X_H$ is characterized by the identity $\iota_{X_{H_t}} \omega = dH_t$. The compactly supported Hamiltonian diffeomorphism group of $(\Sigma,\omega)$ is denoted by $\operatorname{Ham}(\Sigma,\omega)$.

Let $\operatorname{Homeo}_c(\Sigma,\omega)$ denote the group of all compactly supported homeomorphisms of $\Sigma$ preserving the measure induced by $\omega$. We topologize $\operatorname{Homeo}_c(\Sigma,\omega)$ as the direct limit of all the subgroups $\operatorname{Homeo}_K(\Sigma,\omega)$ of homeomorphisms supported inside some compact subset $K\subset \Sigma$, equipped with the compact-open topology. The group of Hamiltonian homeomorphisms $\overline{\operatorname{Ham}}(\Sigma,\omega)$ is defined to be the closure of $\operatorname{Ham}(\Sigma,\omega)$ inside $\operatorname{Homeo}_c(\Sigma,\omega)$. It is topologized as a subspace of $\operatorname{Homeo}_c(\Sigma,\omega)$.

Let $\operatorname{Homeo}_0(\Sigma,\omega)$ denote the identity component of $\operatorname{Homeo}_c(\Sigma_,\omega)$. Then $\overline{\operatorname{Ham}}(\Sigma,\omega)$ agrees with the kernel of the mass flow homomorphism \cite{fat80, sch57}
\begin{equation*}
    \theta : \operatorname{Homeo}_0(\Sigma,\omega) \rightarrow H^1(\Sigma;\R)/\Gamma,
\end{equation*}
where $\Gamma \subset H^1(\Sigma;\R)$ is a discrete subgroup whose precise definition will not be relevant for us.

Whenever there is no risk of confusion, we will abbreviate $\operatorname{Ham}(\Sigma) \coloneqq \operatorname{Ham}(\Sigma,\omega)$ and $\overline{\operatorname{Ham}}(\Sigma) \coloneqq \overline{\operatorname{Ham}}(\Sigma,\omega)$.

It is known that $\Ham(\Sigma)$ and $\overline{\Ham}(\Sigma)$, equipped with the $C^0$ topology, are both simply connected, unless $\Sigma$ is a sphere. If $\Sigma$ is a sphere, the fundamental group is $\mathbb{Z}/2\mathbb{Z}$.  In the case of $\Ham(\Sigma)$, a proof of these facts is sketched in \cite[Sec.\ 7.2]{Polterovich-book01} for closed $\Sigma$.  The arguments therein can be adapted to the case of open $\Sigma$.  As for $\overline{\Ham}(\Sigma)$, it can be shown that every loop in $\overline{\Ham}(\Sigma)$ is homotopic a loop in $\Ham(\Sigma)$ using the following two facts: first, every homeomorphism in $\overline{\Ham}(\Sigma)$ can be written as a $C^0$ limit of elements in $\Ham(\Sigma)$; and second, $\operatorname{Ham}(\Sigma)$ and $\overline{\Ham}(\Sigma)$ are both locally path connected; see, for example, \cite[Cor.\ 2]{Serraille} or \cite[Lem.\ 3.2]{Sey13}.

The following two propositions will be used in the upcoming sections. Both are variants of known fragmentation and approximation results for area-preserving maps and can be established using classical techniques---specifically, fragmentation techniques and approximation of homeomorphisms by diffeomorphisms in the area-preserving setting; see, e.g., \cite{LeRoux, EPP}. We therefore omit the proofs.

\begin{prop}
\label{prop:fragmentation_hamiltonian_homeomorphisms}
    Let $(\Sigma,\omega)$ be a surface with area form.
    \begin{enumerate}
        \item Suppose $U_1, \dots, U_n$ is a finite open cover of $\Sigma$, and for each $i$, let $\mathcal{U}_i \subset \overline{\operatorname{Ham}}(U_i)$ be an open neighborhood of the identity. Then any Hamiltonian homeomorphism $\varphi \in \overline{\operatorname{Ham}}(\Sigma)$ that lies in a sufficiently small neighborhood of the identity can be written as a composition $\varphi = \varphi_1 \circ \cdots \circ \varphi_n$, where each $\varphi_i \in \mathcal{U}_i \subset \overline{\operatorname{Ham}}(U_i)$.
        
        \item Let $\mathcal{V}$ be an arbitrary open cover of $\Sigma$. Then, for every Hamiltonian homeomorphism $\varphi \in \overline{\operatorname{Ham}}(\Sigma)$, there exist finitely many open sets $U_1,\dots, U_n \in \mathcal{V}$ and Hamiltonian homeomorphisms $\varphi_i \in \overline{\operatorname{Ham}}(U_i)$ such that $\varphi = \varphi_1 \circ \cdots \circ \varphi_n$.
    \end{enumerate}
\end{prop}

\begin{prop}
\label{prop:smooth_approximation_area_preserving_homeos}
    Let $(\Sigma_i, \omega_i)$ be a surface equipped with an area form for each $i \in \{1, 2\}$. Suppose $\varphi : (\Sigma_1, \omega_1) \to (\Sigma_2, \omega_2)$ is an area-preserving homeomorphism. Let $U \subset \Sigma_1$ be an open subset on which $\varphi$ is smooth. Let $d$ be a distance function on $\Sigma_2$ that induces its topology, and let $\rho: \Sigma_1 \to \mathbb{R}_{\geq 0}$ be a continuous, non-negative function satisfying $\rho^{-1}(\{0\}) \subset U$. Then there exists an area-preserving diffeomorphism $\psi : (\Sigma_1, \omega_1) \to (\Sigma_2, \omega_2)$ such that
\begin{equation*}
    d(\varphi(x), \psi(x)) \leq \rho(x) \quad \text{for all } x \in \Sigma_1.
\end{equation*}
\end{prop}

\subsection{Foliations and transverse measures}
\label{sec:foliations}

Let $M^n$ be an oriented $C^r$ manifold for $r\in \Z_{\geq 0} \cup \{\infty\}$. For $k\geq 0$, let $B^k \subset \R^k$ denote the unit ball. We define a \textit{$d$-dimensional $C^r$ foliation} $\mathcal{F}$ on $M$ to be a decomposition of $M$ into connected subsets $L\subset M$, called \textit{leaves}, such that $M$ can be covered by $C^r$ coordinate charts
\begin{equation*}
    \phi : U \rightarrow B^d\times B^{n-d} \subset \R^n,
\end{equation*}
called \textit{foliation charts}, with the property that, for every leaf $L$, the connected components of $\phi(L\cap U)$ are of the form $B^d\times \{*\}$. The collection of the foliation charts $\{(U, \phi)\}$ defining $\mathcal{F}$ forms a \emph{foliated atlas}. Since the manifold $M$ is oriented, we may assume without loss of generality that the foliated atlas on $M$ is an oriented atlas of $M$, meaning that we require the transition maps between the foliation charts to be orientation preserving.  Note that this does not imply that $\mathcal{F}$ is oriented.

We will consider foliations $\mathcal{F}$ which are equipped with two additional structures:  a \textit{transverse measure} $\Lambda$ and a \emph{transverse orientation}, also referred to as a \emph{coorientation}.  A  transverse measure $\Lambda$ is specified by a foliated atlas of $M$ consisting of foliation charts $\phi: U \rightarrow  B^d \times B^{n-d}$ and a finite Borel measure $\lambda_\phi$ on $B^{n-d}$  for every chart $\phi$ such that the following is true: For $i\in \{1,2\}$, let $\phi_i : U_i \rightarrow B^d\times B^{n-d}$ be a chart and consider the transition map
\begin{equation*}
    \phi_{21} : \phi_1(U_1\cap U_2) \rightarrow \phi_2(U_1\cap U_2).
\end{equation*}
Then, for every point $p_0 = (x_0,y_0) \in \phi_1(U_1\cap U_2)$, there exist an open neighborhood $V$ of $x_0$ in $B^d$, an open neighborhood $W$ of $y_0$ in $B^{n-d}$, and a topological embedding $\iota : W \hookrightarrow B^{n-d}$ such that the restriction of $\phi_{21}$ to $V\times W$ is of the form $\phi_{21}(x,y) = (*, \iota(y))$ and moreover 
$$\iota_*\lambda_{\phi_1}|_W = \lambda_{\phi_2}|_{\iota(W)}.$$
Requiring that the topological embeddings $\iota$ are \emph{orientation preserving} determines a \emph{transverse orientation}, or \emph{coorientation}, of the foliation $\mathcal{F}$.

Given a foliation $\mathcal{F}$, a \textit{transversal} is a compact topological submanifold with boundary $T^{n-d}\subset M$ that is transverse to the foliation $\mathcal{F}$. This means that every point in $T$ is contained in a foliation chart $\phi: U \rightarrow B^d\times B^{n-d}$ such that $\phi(T\cap U)$ is contained in a transverse slice of the form $\{*\} \times B^{n-d}$. A transverse measure $\Lambda$ induces a finite Borel measure on every transversal $T$ and maps between subsets of transversals obtained by sliding along leaves of $\mathcal{F}$ are measure preserving. In fact, one can equivalently define a transverse measure $\Lambda$ on $\mathcal{F}$ to be an assignment of a finite Borel measure to every transversal $T$ subject to the condition that sliding along leaves of $\mathcal{F}$ preserves measure.  A transverse orientation admits a similar description: it is an assignment of an orientation to every transversal $T$ such that sliding along leaves of $\mathcal{F}$ preserves orientation.

Finally, note that since the ambient manifold $M$ is oriented, a coorientation of $\mathcal{F}$ induces a natural orientation on the foliation $\mathcal{F}$.  Vice-versa, the orientation on $M$ together with an orientation of $\mathcal{F}$ determine a coorientation of $\mathcal{F}$.

\section{Universal extension of the Calabi homomorphism}
\label{sec:hamiltonian_homeos_and_calabi}

In this section, we prove Theorem \ref{thm:natural_isomorphisms_abelianizations_hamiltonian_homeomorphisms-v2}, which generalizes Theorem \ref{thm:natural_isomorphisms_abelianizations_hamiltonian_homeomorphisms} of the introduction.  As a consequence, we deduce the existence of a $\mathcal{R}$-valued universal extension of the Calabi homomorphism which satisfies the naturality property stated in Section \ref{sec:plugs and Calabi}. 

\medskip

Let $\mathcal{S}$ be the category whose objects are pairs  $(\Sigma, \omega)$, where $\Sigma$ is a non-empty, connected, smooth surface without boundary, and $\omega$ is an area form on $\Sigma$. Morphisms in $\mathcal{S}$ are smooth embeddings \(\iota : (\Sigma_1, \omega_1) \hookrightarrow (\Sigma_2, \omega_2)\). We include both open and closed surfaces, without requiring finite area in the open case. Let $\mathcal{S}_{\operatorname{op}}$ and $\mathcal{S}_{\operatorname{cl}}$ denote the full subcategories of open and closed surfaces, respectively. When there is no risk of confusion, we simply write $\Sigma$ for $(\Sigma, \omega)$.

Recall that $\operatorname{Ham}(\Sigma)$ denotes the group of all compactly supported Hamiltonian diffeomorphisms of $\Sigma$. Every area-preserving embedding $\iota : \Sigma_1 \hookrightarrow \Sigma_2$ induces a group homomorphism 
\begin{equation*}
    \operatorname{Ham}(\iota) : \operatorname{Ham}(\Sigma_1) \rightarrow \operatorname{Ham}(\Sigma_2) \qquad \operatorname{Ham}(\iota)(\varphi) \coloneqq \iota_*\varphi
\end{equation*}
given by pushforward of compactly supported Hamiltonian diffeomorphisms via $\iota$. We can therefore regard $\operatorname{Ham}$ as a functor from $\mathcal{S}$ to the category of all groups. We define
\begin{equation*}
    \operatorname{Ham}^{\operatorname{ab}}(\Sigma) \coloneqq \operatorname{Ham}(\Sigma) / [\operatorname{Ham}(\Sigma), \operatorname{Ham}(\Sigma)]
\end{equation*}
to be the abelianization of $\Ham(\Sigma)$. Since abelianization forms a functor from the category of groups to the category of abelian groups,  $\operatorname{Ham}^{\operatorname{ab}}$ is a functor from $\mathcal{S}$ to the category of abelian groups. 

It was proved by Banyaga \cite{ban78} that $\Ham(\Sigma)$ is perfect for closed $\Sigma$. In other words, $$\operatorname{Ham}^{\operatorname{ab}}(\Sigma) = 1$$ for every closed surface $\Sigma \in \mathcal{S}_{\operatorname{cl}}$. For open surfaces $\Sigma \in \mathcal{S}_{\operatorname{op}}$ (and more generally for exact connected symplectic manifolds of arbitrary dimension), Banyaga showed that $\operatorname{Ham}^{\operatorname{ab}}(\Sigma) \cong \R$. Moreover, an explicit isomorphism between $\operatorname{Ham}^{\operatorname{ab}}(\Sigma)$ and $\R$ is induced by the Calabi homomorphism
\begin{equation*}
    \operatorname{Cal}_\Sigma : \operatorname{Ham}(\Sigma) \rightarrow \R,
\end{equation*}
whose definition we now recall. For a given $\varphi \in \operatorname{Ham}(\Sigma)$,  pick $H : [0,1]\times \Sigma \rightarrow \R$  such that $\varphi_H^1 = \varphi$. Then,
\begin{equation*}
    \operatorname{Cal}_\Sigma(\varphi) = 2\int_{[0,1]\times \Sigma} H dt \wedge \omega.
\end{equation*}
This turns out to be independent of the choice of the Hamiltonian $H$. Banyaga proved that $\operatorname{Cal}_\Sigma$ is a surjective group homomorphism whose kernel is given by the commutator subgroup of $\operatorname{Ham}(\Sigma)$. In other words, the Calabi homomorphism descends to an isomorphism $\operatorname{Cal}_\Sigma : \operatorname{Ham}^{\operatorname{ab}}(\Sigma) \rightarrow \R$.

The Calabi homomorphism is natural in the following sense: If $\iota : \Sigma_1\hookrightarrow \Sigma_2$ is an area-preserving embedding between open surfaces $\Sigma_1,\Sigma_2 \in \mathcal{S}_{\operatorname{op}}$, then
\begin{equation*}
    \operatorname{Cal}_{\Sigma_2} \circ \operatorname{Ham}(\iota) = \operatorname{Cal}_{\Sigma_1}.
\end{equation*}
This follows from the observation that the integral of a compactly supported Hamiltonian does not change if we push it forward via an area-preserving embedding. Naturality of the Calabi homomorphism and the fact that $\operatorname{Cal}_\Sigma:\operatorname{Ham}^{\operatorname{ab}}(\Sigma)\rightarrow \R$ is an isomorphism for every open surface $\Sigma \in \mathcal{S}_{\operatorname{op}}$ imply that $\operatorname{Ham}^{\operatorname{ab}}(\iota)$ is an isomorphism for all embeddings $\iota$ between open surfaces. Moreover, one can deduce that $\operatorname{Ham}^{\operatorname{ab}}(\iota)$ is independent of the choice of embedding $\iota$.

\medskip 

Recall from Section \ref{sec:prelims-surfaces} that $\overline{\operatorname{Ham}}(\Sigma)$ denotes the group of compactly supported Hamiltonian homeomorphisms of $\Sigma$. As before, this assignment defines a functor from the category $\mathcal{S}$ to the category of groups. Moreover, since Hamiltonian homeomorphisms can be pushed forward via area-preserving topological embeddings, we can extend this functor to a larger category $\overline{\mathcal{S}}$, which has the same objects as $\mathcal{S}$ but allows all area-preserving topological embeddings as morphisms.

There is a natural inclusion $\operatorname{Ham}(\Sigma) \subset \overline{\operatorname{Ham}}(\Sigma)$, which corresponds to a natural transformation from the functor $\operatorname{Ham}$ to the composition of the inclusion $\mathcal{S} \hookrightarrow \overline{\mathcal{S}}$ with the functor $\overline{\operatorname{Ham}}$. 

As mentioned in the introduction, for every area-preserving topological embedding of surfaces $\iota: \Sigma_1\hookrightarrow \Sigma_2$, we let $\overline{\operatorname{Ham}}^{\operatorname{ab}}(\iota) : \overline{\operatorname{Ham}}^{\operatorname{ab}}(\Sigma_1) \rightarrow \overline{\operatorname{Ham}}^{\operatorname{ab}}(\Sigma_2)$ denote the induced homomorphism at the level of abelianizations, where
\begin{equation*}
    \overline{\operatorname{Ham}}^{\operatorname{ab}}(\Sigma)\coloneqq \overline{\operatorname{Ham}}(\Sigma)/[\overline{\operatorname{Ham}}(\Sigma), \overline{\operatorname{Ham}}(\Sigma)].
\end{equation*}
As in the smooth setting, $\overline{\operatorname{Ham}}^{\operatorname{ab}}$ constitutes a functor from $\overline{\mathcal{S}}$ to the category of abelian groups. 

We prove the following generalization of Theorem \ref{thm:natural_isomorphisms_abelianizations_hamiltonian_homeomorphisms}.

\begin{thm}
\label{thm:natural_isomorphisms_abelianizations_hamiltonian_homeomorphisms-v2}
    For $i\in \{1,2\}$, let $(\Sigma_i,\omega_i)$ be non-empty connected surfaces without boundary and equipped with area forms. Let $\iota : \Sigma_1 \hookrightarrow \Sigma_2$ be an area-preserving topological embedding. Then, the induced map
    \begin{equation*}
        \overline{\operatorname{Ham}}^{\operatorname{ab}}(\iota) : \overline{\operatorname{Ham}}^{\operatorname{ab}}(\Sigma_1) \rightarrow \overline{\operatorname{Ham}}^{\operatorname{ab}}(\Sigma_2)
    \end{equation*}
    does not depend on the choice of embedding $\iota: \Sigma_1 \hookrightarrow \Sigma_2$, i.e.\ if $\iota':\Sigma_1\hookrightarrow \Sigma_2$ is any other embedding, then $$\overline{\operatorname{Ham}}^{\operatorname{ab}}(\iota') = \overline{\operatorname{Ham}}^{\operatorname{ab}}(\iota).$$
    Moreover, 
    \begin{itemize}
        \item[i.] If $\Sigma_1$ and $\Sigma_2$ are both open or both closed, then $\overline{\operatorname{Ham}}^{\operatorname{ab}}(\iota)$ is an isomorphism. 
        \item[ii.]  If $\Sigma_1$ is open and $\Sigma_2$ is closed, then $\overline{\operatorname{Ham}}^{\operatorname{ab}}(\iota)$ is surjective and its kernel is given by the image of $\operatorname{Ham}^{\operatorname{ab}}(\Sigma_1)$ in $\overline{\operatorname{Ham}}^{\operatorname{ab}}(\Sigma_1)$.
    \end{itemize}
\end{thm}

Before presenting the proof of the above, we remark on some of its consequences.

\medskip

For every surface $\Sigma \in \mathcal{S}$, the abelian group $\overline{\operatorname{Ham}}^{\operatorname{ab}}(\Sigma)$ is non-trivial. This was proven in the case of the disc in \cite{chs24} and for closed surfaces and open surfaces which are the interiors of compact surfaces with boundary in \cite{chs24b, chmss22}. For general open surfaces, possibly of infinite area, non-triviality of $\overline{\operatorname{Ham}}^{\operatorname{ab}}(\Sigma)$ can be deduced from Theorem \ref{thm:natural_isomorphisms_abelianizations_hamiltonian_homeomorphisms-v2}.

It is interesting to point out that non-perfectness of $\overline{\operatorname{Ham}}(\Sigma)$ for one single open surface $\Sigma\in \mathcal{S}_{\operatorname{op}}$, for example $\Sigma = \D$, implies non-perfectness for all open surfaces via Theorem \ref{thm:natural_isomorphisms_abelianizations_hamiltonian_homeomorphisms-v2}. Similarly, non-perfectness of $\overline{\operatorname{Ham}}(\Sigma)$ for one single closed surface $\Sigma \in \mathcal{S}_{\operatorname{cl}}$ implies non-perfectness for all closed surfaces.

As already explained in the introduction, we define
\begin{equation*}
    \mathcal{R} \coloneqq \overline{\operatorname{Ham}}^{\operatorname{ab}}(\mathbb{D}).
\end{equation*}
Theorem \ref{thm:natural_isomorphisms_abelianizations_hamiltonian_homeomorphisms-v2} allows us to canonically identify
\begin{equation*}
    \mathcal{R}\cong \overline{\operatorname{Ham}}^{\operatorname{ab}}(\Sigma)
\end{equation*}
for every open surface $\Sigma\in \mathcal{S}_{\operatorname{op}}$.

Despite significant interest, the algebraic structure of $\mathcal{R}$ remains rather mysterious. Most of what is known follows from the existence of a certain sequence of quasimorphisms
\begin{equation*}
    c_d : \overline{\operatorname{Ham}}(\D) \rightarrow \R
\end{equation*}
called \emph{link spectral invariants}; see \cite{chmss22, chmss}. Consider the space $\R^\N$ of real valued sequences and let $N\subset \R^\N$ denote the subspace of sequences converging to zero. Then the link spectral invariants induce a map
\begin{equation*}
    c : \overline{\operatorname{Ham}}(\D) \rightarrow \R^\N/N \qquad \varphi \mapsto [c_1(\varphi),c_2(\varphi), c_3(\varphi), \dots].
\end{equation*}
Since the sequence of defects of the quasimorphisms $c_d$ behaves like $O(d^{-1})$ as $d$ goes to infinity, this map $c$ is a group homomorphism. Moreover, it fits into the following commutative diagram:
\begin{equation*}
\begin{tikzcd}
\overline{\operatorname{Ham}}(\D) \arrow[r] \arrow[rr, bend left=30, "c"] & \overline{\operatorname{Ham}}^{\operatorname{ab}}(\D) \arrow[r] & \R^\N/N \\
\operatorname{Ham}(\D) \arrow[u, hook] \arrow[r] \arrow[rr, bend right=30, "\operatorname{Cal}_\D"] & \operatorname{Ham}^{\operatorname{ab}}(\D) \arrow[u] \arrow[r, "\cong"] & \R \arrow[u, hook, "\Delta"]
\end{tikzcd}
\end{equation*}
Here the homomorphism $\Delta$ maps $x\in \R$ to the equivalence class of the constant sequence $(x)_{n\in \N}$. Commutativity of this diagram is equivalent to the important \emph{Weyl law} satisfied by the link spectral invariants, which says that they asymptotically recover the Calabi invariant.

It follows from this commutative diagram that the natural map
\begin{equation*}
    \operatorname{Ham}^{\operatorname{ab}}(\D) \rightarrow \overline{\operatorname{Ham}}^{\operatorname{ab}}(\D)
\end{equation*}
is injective. By Theorem \ref{thm:natural_isomorphisms_abelianizations_hamiltonian_homeomorphisms-v2} the same is then true for any other open surface $\Sigma$ as well. As discussed in the introduction, after identifying $\operatorname{Ham}^{\operatorname{ab}}(\D) \cong \R$ via the Calabi homomorphism, we can therefore regard
\begin{equation*}
    \R \subset \mathcal{R}
\end{equation*}
as a subgroup. It is proven in \cite[Proposition 5.3]{chmss} that the homomorphism $c$ is surjective. In particular, this implies that $\R$ is a proper subgroup of $\mathcal{R}$.

Let $\Sigma \in \mathcal{S}_{\operatorname{op}}$ be an open surface. The restriction of the natural map
\begin{equation*}
    \overline{\mathrm{Cal}}_\Sigma  : \overline{\operatorname{Ham}}(\Sigma) \rightarrow \overline{\operatorname{Ham}}^{\operatorname{ab}}(\Sigma) \cong \mathcal{R}.
\end{equation*}
to $\operatorname{Ham}(\Sigma)$ agrees with the usual Calabi homomorphism via the natural inclusion $\R \subset \mathcal{R}$. We refer to this map as the \emph{universal $\mathcal{R}$-valued extension of Calabi}. If $\Sigma$ is not connected, we define $\overline{\mathrm{Cal}}_\Sigma : \overline{\operatorname{Ham}}(\Sigma) \rightarrow \mathcal{R}$ to be the sum of all $\overline{\mathrm{Cal}}_S$ where $S$ ranges over the components of $\Sigma$. In view of Theorem \ref{thm:natural_isomorphisms_abelianizations_hamiltonian_homeomorphisms-v2}, it is clear that  we have  the naturality property
\begin{equation}
\label{eq:naturality_universal_calabi_section_hamiltonian_homeoms_and_calabi}
\overline{\mathrm{{Cal}}}_{\Sigma_1} =\overline{\mathrm{{Cal}}}_{\Sigma_2} \circ \overline{\operatorname{Ham}}(\iota)
\end{equation}
for any area-preserving topological embedding $\iota: \Sigma_1 \hookrightarrow \Sigma_2$.

Every projection $p : \mathcal{R} \rightarrow \R$ gives rise to a real-valued extension $p\circ \overline{\operatorname{Cal}}_\Sigma$ of the Calabi homomorphism $\operatorname{Cal}_\Sigma$ for any open surface $\Sigma$. Conversely, every real-valued Calabi extension arises this way. Clearly, the real-valued Calabi extensions arising from a single choice of projection $p: \mathcal{R}\rightarrow \R$ satisfy the naturality condition \eqref{eq:naturality_universal_calabi_section_hamiltonian_homeoms_and_calabi}.

Projections $p: \mathcal{R}\rightarrow \R$ can be obtained as follows: The homomorphism $c$ descends to a surjective homomorphism $c : \mathcal{R} \rightarrow \R^\N/N$. Viewing $\R$ as a subgroup of $\R^\N/N$ via the inclusion $\Delta$, this is actually a homomorphism over $\R$. Composing $c$ with any projection $\R^\N/N\rightarrow \R$ yields a projection $p:\mathcal{R}\rightarrow \R$. There exist infinitely many projections $\R^\N/N\rightarrow \R$, but the choice of any such projection requires the axiom of choice.

Finally, note that one interesting consequence of the discussion here is that having a $\R$-valued extension of the Calabi homomorphism to $\overline{\Ham}(\Sigma)$ for one single open surface $\Sigma$, for example $\Sigma = \D$, implies the existence of such extensions for all open surfaces. The fact that Calabi admits such extensions was proven for surfaces obtained as the interior of a compact surface with non-empty boundary in \cite{chmss22} in the genus zero case and by Mak--Trifa in \cite{mt} in the case of arbitrary genus. Here we obtain extensions in full generality, even for surfaces of infinite area.

\begin{proof}[Proof of Theorem \ref{thm:natural_isomorphisms_abelianizations_hamiltonian_homeomorphisms-v2}]
    The proof proceeds in several steps.
    
    \emph{Step 1 (independence of $\iota$):} We show that $\overline{\operatorname{Ham}}^{\operatorname{ab}}(\iota)$ is independent of $\iota$. Consider two area-preserving topological embeddings $\iota,\iota' : \Sigma_1 \hookrightarrow \Sigma_2$. Let $\varphi \in \overline{\operatorname{Ham}}(\Sigma_1)$. We need to show that $[\iota_*\varphi] = [\iota_*'\varphi]$ in $\overline{\operatorname{Ham}}^{\operatorname{ab}}(\Sigma_2)$. By Proposition \ref{prop:fragmentation_hamiltonian_homeomorphisms}, we may pick a finite collection of relatively compact open discs $D_1,\dots,D_n\Subset \Sigma_1$ and Hamiltonian homeomorphisms $\varphi_i \in \overline{\operatorname{Ham}}(D_i)$ such that $\varphi = \varphi_1\circ \cdots \circ \varphi_n$. Since $\Sigma_2$ is connected, we can also find Hamiltonian homeomorphisms $\alpha_i \in \overline{\operatorname{Ham}}(\Sigma_2)$ such that $\alpha_i \circ \iota|_{D_i} = \iota'|_{D_i}$. In $\overline{\operatorname{Ham}}^{\operatorname{ab}}(\Sigma_2)$ we can then compute
    \begin{equation*}
        [\iota_*'\varphi] = [\iota_*'\varphi_1  \circ\cdots\circ  \iota_*'\varphi_n] = [\alpha_1(\iota_*\varphi_1)\alpha_1^{-1} \circ \cdots \circ \alpha_n(\iota_*\varphi_n)\alpha_n^{-1}] = [\iota_*\varphi_1 \circ \cdots \circ \iota_*\varphi_n] = [\iota_*\varphi],
    \end{equation*}
    as desired.

    \emph{Step 2 (surjectivity):} We show that $\overline{\operatorname{Ham}}^{\operatorname{ab}}(\iota)$ is surjective. By Step 1, we may replace $\iota$ by a smooth area-preserving embedding. After identifying $\Sigma_1$ with its image under $\iota$, we can regard $\Sigma_1$ as an open subset of $\Sigma_2$. Given an arbitrary element $\varphi \in \overline{\operatorname{Ham}}(\Sigma_2)$, we need to show that the class $[\varphi] \in \overline{\operatorname{Ham}}^{\operatorname{ab}}(\Sigma_2)$ has a representative in $\overline{\operatorname{Ham}}(\Sigma_1)$. By Proposition \ref{prop:fragmentation_hamiltonian_homeomorphisms}, there exist finitely many relatively compact open discs $D_1,\dots,D_n \Subset \Sigma_2$, each sufficiently small such that there exists $\alpha_i \in \operatorname{Ham}(\Sigma_2)$ such that $\alpha_i(D_i)\Subset \Sigma_1$, and Hamiltonian homeomorphisms $\varphi_i \in \overline{\operatorname{Ham}}(D_i)$ such that $\varphi = \varphi_1\circ \cdots \varphi_n$. Set $\psi_i \coloneqq \alpha_i \varphi_i\alpha_i^{-1} \in \overline{\operatorname{Ham}}(\Sigma_1)$. Observe that in $\overline{\operatorname{Ham}}^{\operatorname{ab}}(\Sigma_2)$ we have the identity
    \begin{equation*}
        [\varphi] = [\varphi_1 \circ \cdots \circ \varphi_n] = [\psi_1 \circ \cdots \circ \psi_n].
    \end{equation*}
    Therefore, $\psi\coloneqq \psi_1 \circ \cdots \circ \psi_n \in \overline{\operatorname{Ham}}(\Sigma_1)$ is the desired representative of $[\varphi]$.

    \emph{Step 3 (injectivity in the case of open surfaces):} Assume that both $\Sigma_1$ and $\Sigma_2$ are open. We show that $\overline{\operatorname{Ham}}^{\operatorname{ab}}(\iota)$ is injective. As in Step 2, we may assume that $\iota$ is smooth. We identify $\Sigma_1$ with its image under $\iota$ and view it as an open subset of $\Sigma_2$. Our goal is to construct an inverse
    \begin{equation*}
        F : \overline{\operatorname{Ham}}^{\operatorname{ab}}(\Sigma_2) \rightarrow \overline{\operatorname{Ham}}^{\operatorname{ab}}(\Sigma_1)
    \end{equation*}
    of $\overline{\operatorname{Ham}}^{\operatorname{ab}}(\iota)$. By the universal property of abelianization, defining a group homomorphism $F: \overline{\operatorname{Ham}}^{\operatorname{ab}}(\Sigma_2) \rightarrow \overline{\operatorname{Ham}}^{\operatorname{ab}}(\Sigma_1)$ is equivalent to defining a group homomorphism $\tilde{F} : \overline{\operatorname{Ham}}(\Sigma_2) \rightarrow \overline{\operatorname{Ham}}^{\operatorname{ab}}(\Sigma_1)$. Given $\varphi \in \overline{\operatorname{Ham}}(\Sigma_2)$, we define $\tilde{F}(\varphi)$ as follows. By Proposition  \ref{prop:fragmentation_hamiltonian_homeomorphisms}, we may pick a fragmentation $\varphi = \varphi_n \circ \cdots \circ \varphi_1$ such that $\varphi_i \in \overline{\operatorname{Ham}}(D_i)$ for some open disc $D_i \Subset \Sigma_2$ which is sufficiently small such that there exists an area-preserving embedding $\iota_i : D_i \hookrightarrow \Sigma_1$. We set
    \begin{equation*}
        \tilde{F}(\varphi) \coloneqq \prod\limits_{i=1}^n \overline{\operatorname{Ham}}^{\operatorname{ab}}(\iota_i)([\varphi_i]) = \prod\limits_{i=1}^n [(\iota_i)_*\varphi_i] \in \overline{\operatorname{Ham}}^{\operatorname{ab}}(\Sigma_1).
    \end{equation*}
    Since $\overline{\operatorname{Ham}}^{\operatorname{ab}}(\Sigma_1)$ is abelian, the order of the product of course does not matter. We need to check that this definition of $\tilde{F}$ is independent of the choice of fragmentation of $\varphi$. Postponing the proof of this claim, let us first argue that $\tilde{F}$ indeed induces a homomorphism $F$ which is an inverse of $\overline{\operatorname{Ham}}^{\operatorname{ab}}(\iota)$. Observe that $\tilde{F}$ is a group homomorphism because given fragmentations of $\varphi, \psi \in \overline{\operatorname{Ham}}(\Sigma_2)$, one can obtain a fragmentation of $\varphi\circ \psi$ by concatenation. Consider $\varphi \in \overline{\operatorname{Ham}}(\Sigma_1)$. We may pick a fragmentation $\varphi = \varphi_n \circ \cdots \circ \varphi_1$ such that all the discs $D_i$ are contained in $\Sigma_1$. We can then take the embeddings $\iota_i: D_i\hookrightarrow \Sigma_1$ to be inclusions. With these choices, we see that $F\circ \overline{\operatorname{Ham}}^{\operatorname{ab}}(\iota)([\varphi]) = [\varphi]$. Since we already know that $\overline{\operatorname{Ham}}^{\operatorname{ab}}(\iota)$ is surjective, this implies that $\overline{\operatorname{Ham}}^{\operatorname{ab}}(\iota)$ is an isomorphism and that $F$ is its inverse.
    
    Showing that the definition of $\tilde{F}$ is independent of choices reduces to verifying the following claim:

    \textit{Let $D_1,\dots,D_n \Subset \Sigma_2$ be open discs which are small enough such that they admit area-preserving embeddings $\iota_i : D_i \hookrightarrow \Sigma_1$ into $\Sigma_1$. Let $\varphi_i \in \overline{\operatorname{Ham}}(D_i)$ be Hamiltonian homeomorphisms. Then
    \begin{equation*}
        \varphi_n \circ \cdots \circ \varphi_1 = \operatorname{id} \qquad \Longrightarrow \qquad \prod_i \overline{\operatorname{Ham}}^{\operatorname{ab}}(\iota_i)([\varphi_i]) = \operatorname{id}.
    \end{equation*}}

    Fix discs $D_i$ and Hamiltonian homeomorphisms $\varphi_i \in \overline{\operatorname{Ham}}(D_i)$ as in the claim and assume that $\varphi_n\circ \cdots\circ \varphi_1 = \operatorname{id}$. Set $\varphi_i^{(0)} \coloneqq \varphi_i$ for $1\leq i \leq n$. Our strategy is to construct, for $1\leq i,j \leq n$, Hamiltonian homeomorphisms $\varphi_i^{(j)} \in \overline{\operatorname{Ham}}(D_i)$ such that the following properties are satisfied:
     \begin{enumerate}
         \item \label{item:coherent_system_of_calabi_extensions_proof_composition}$\varphi_n^{(j)} \circ \cdots \circ\varphi_1^{(j)} = \operatorname{id}$ for all $j$.
         \item \label{item:coherent_system_of_calabi_extensions_proof_sum}$\prod_i \overline{\operatorname{Ham}}^{\operatorname{ab}}(\iota_i)([\varphi_i^{(j)}]) = \prod_i \overline{\operatorname{Ham}}^{\operatorname{ab}}(\iota_i)([\varphi_i^{(j-1)}])$ for all $j$.
         \item \label{item:coherent_system_of_calabi_extensions_proof_smooth} $\varphi_i^{(j)}\in \operatorname{Ham}(D_i)$ is smooth for all $i \leq j$.
     \end{enumerate}
     Once we have such homeomorphisms $\varphi_i^{(j)}$, the desired claim follows from the injectivity of $\operatorname{Ham}^{\operatorname{ab}}(\iota):\operatorname{Ham}^{\operatorname{ab}}(\Sigma_1) \rightarrow \operatorname{Ham}^{\operatorname{ab}}(\Sigma_2)$. In order to see this, note that by Step 1 we may assume that the embeddings $\iota_i : D_i \hookrightarrow \Sigma_1$ are smooth. By property \ref{item:coherent_system_of_calabi_extensions_proof_smooth}, all $\varphi_i^{(n)}$ are smooth as well. We can therefore compute
     \begin{equation*}
         \operatorname{Ham}^{\operatorname{ab}}(\iota) \left(\prod_i\operatorname{Ham}^{\operatorname{ab}}(\iota_i)([\varphi_i^{(n)}])\right) = \prod_i[\varphi_i^{(n)}] = \operatorname{id},
     \end{equation*}
     where we use that $\operatorname{Ham}^{\operatorname{ab}}(\kappa)$ is independent of the embedding $\kappa$. By injectivity of $\operatorname{Ham}^{\operatorname{ab}}(\iota)$, this means that $\prod_i\operatorname{Ham}^{\operatorname{ab}}(\iota_i)([\varphi_i^{(n)}])$ is equal to the identity in $\operatorname{Ham}^{\operatorname{ab}}(\Sigma_1)$. But since $\prod_i \overline{\operatorname{Ham}}^{\operatorname{ab}}(\iota_i)([\varphi_i^{(n)}])$ is the image of $\prod_i\operatorname{Ham}^{\operatorname{ab}}(\iota_i)([\varphi_i^{(n)}])$ in $\overline{\operatorname{Ham}}^{\operatorname{ab}}(\Sigma_1)$, we see that $\prod_i \overline{\operatorname{Ham}}^{\operatorname{ab}}(\iota_i)([\varphi_i^{(n)}])$ is equal to the identity as well. It is then immediate from property \ref{item:coherent_system_of_calabi_extensions_proof_sum} that $\prod_i\overline{\operatorname{Ham}}^{\operatorname{ab}}(\iota_i)([\varphi_i])$ is equal to the identity, as desired.
     
     It remains to construct the homeomorphisms $\varphi_i^{(j)}$ satisfying properties \ref{item:coherent_system_of_calabi_extensions_proof_composition}, \ref{item:coherent_system_of_calabi_extensions_proof_sum}, and \ref{item:coherent_system_of_calabi_extensions_proof_smooth}. Fix $1\leq j \leq n$ and assume that the homeomorphisms $\varphi_i^{(k)}$ have already been constructed for $k < j$. We explain how to construct $\varphi_i^{(j)}$. For $1\leq i \leq j-1$, we simply set $\varphi_i^{(j)} \coloneqq \varphi_i^{(j-1)}$. Note that all of these homeomorphisms are smooth. In order to simplify notation, we set
     \begin{equation*}
         \psi_i \coloneqq \varphi_i^{(j-1)} \quad \text{for $j\leq i \leq n$ and} \quad \psi \coloneqq \varphi_{j-1}^{(j-1)} \circ \cdots \circ \varphi_{1}^{(j-1)}.
     \end{equation*}
     Note that $\psi$ is smooth and that
     \begin{equation}
     \label{eq:coherent_system_of_calabi_extensions_proof_c}
         \psi_n \circ \cdots \circ \psi_j \circ \psi = \operatorname{id}.
     \end{equation}
     Set
     \begin{equation*}
         K_i \coloneqq \operatorname{supp} (\psi_i \circ \cdots \circ \psi_{j+1}) \qquad \text{for $j+1 \leq i \leq n$}.
     \end{equation*}
     It follows from smoothness of $\psi$ and identity \eqref{eq:coherent_system_of_calabi_extensions_proof_c} that $\psi_n \circ \cdots \circ \psi_{j+1} \circ \psi_j$ is smooth. As a consequence, any point at which $\psi_j$ is not smooth must be contained in $\operatorname{supp}(\psi_j) \cap \psi_j^{-1}(K_n)$. By Proposition \ref{prop:smooth_approximation_area_preserving_homeos}, we can find $\psi_j^{\operatorname{sm}} \in \operatorname{Ham}(D_j)$ which agrees with $\psi_j$ outside an arbitrarily small neighborhood of $\operatorname{supp}(\psi_j) \cap \psi_j^{-1}(K_n)$ and is arbitrarily $C^0$ close to $\psi_j$. Here we use that because $D_j$ is a disc, every compactly supported area-preserving diffeomorphism of $D_j$ is automatically contained in $\operatorname{Ham}(D_j)$. We can then write $\psi_j = \alpha \circ \psi_j^{\operatorname{sm}}$, where $\alpha$ is an area-preserving homeomorphism which is supported in an arbitrarily small neighborhood of $\operatorname{supp}(\psi_j) \cap K_n$ and is arbitrarily $C^0$ close to the identity. Since $\operatorname{supp}(\psi_j)\cap K_n$ is contained in $D_j\cap (D_{j+1} \cup \cdots \cup D_n)$, we can assume that $\alpha \in \operatorname{Homeo}_0(D_j\cap(D_{j+1} \cup \cdots \cup D_n),\omega_2)$. We claim that we can in addition assume that $\alpha \in \overline{\operatorname{Ham}}(D_j \cap (D_{j+1} \cup \cdots \cup D_n))$. Indeed, if this is not the case, simply pick $\beta \in \operatorname{Diff}_0(D_j\cap(D_{j+1} \cup \cdots \cup D_n),\omega_2)$ close to the identity such that the mass flow homomorphism agrees on $\alpha$ and $\beta$. Then replace $\alpha$ and $\psi_j^{\operatorname{sm}}$ by $\alpha\circ \beta^{-1}$ and $\beta\circ \psi_j^{\operatorname{sm}}$, respectively. To summarize, we have constructed a factorization
     \begin{equation}\label{eq:coherent_system_of_calabi_extensions_proof_d}
         \psi_j = \alpha \circ \psi_j^{\operatorname{sm}} \quad \text{with $\psi_j^{\operatorname{sm}} \in \operatorname{Ham}(D_j)$ and $\alpha \in \overline{\operatorname{Ham}}(D_j \cap (D_{j+1} \cup \dots \cup D_n))$}.
     \end{equation}
     Moreover, we can take $\alpha$ to be arbitrarily close to the identity. Now set
     \begin{equation*}
         U_i \coloneqq D_j\cap (D_i \setminus K_{i-1}) \quad \text{for $j+1 < i \leq n$ and} \quad U_{j+1} \coloneqq D_j \cap D_{j+1}.
     \end{equation*}
     Observe that
     \begin{equation*}
         U_{j+1} \cup \cdots \cup U_n = D_j\cap (D_{j+1} \cup \cdots \cup D_n).
     \end{equation*}
     Since $\alpha \in \overline{\operatorname{Ham}}(D_j \cap (D_{j+1} \cup \dots \cup D_n))$ is close to the identity, Proposition \ref{prop:fragmentation_hamiltonian_homeomorphisms} allows us to fragment
     \begin{equation}
     \label{eq:coherent_system_of_calabi_extensions_proof_g}
         \alpha = \alpha_n \circ \cdots \circ \alpha_{j+1} \qquad \text{with $\alpha_i \in \overline{\operatorname{Ham}}(U_i)$.}
     \end{equation}
     Using that $U_i$ is disjoint from $K_{i-1}$ by definition, we obtain
     \begin{equation}
     \label{eq:coherent_system_of_calabi_extensions_proof_e}
         \psi_n \circ \cdots \circ \psi_{j+1} \circ \alpha = (\psi_n \circ \alpha_n) \circ \cdots \circ (\psi_{j+1} \circ \alpha_{j+1}).
     \end{equation}
     We set
     \begin{equation*}
         \varphi_i^{(j)} \coloneqq \psi_i \circ \alpha_i \quad \text{for $j+1 \leq i \leq n$ and} \quad \varphi_j^{(j)} \coloneqq \psi_j^{\operatorname{sm}}.
     \end{equation*}
     This concludes the construction of the Hamiltonian homeomorphisms $\varphi_i^{(j)}$ and we need to check that they satisfy properties \ref{item:coherent_system_of_calabi_extensions_proof_composition}, \ref{item:coherent_system_of_calabi_extensions_proof_sum}, and \ref{item:coherent_system_of_calabi_extensions_proof_smooth}. First, note that $\varphi_i^{(j)} \in \overline{\operatorname{Ham}}(D_i)$ for all $i$ and $\varphi_i^{(j)} \in \operatorname{Ham}(D_i)$ for all $i\leq j$ by construction. In particular, property \ref{item:coherent_system_of_calabi_extensions_proof_smooth} is satisfied. It is immediate from \eqref{eq:coherent_system_of_calabi_extensions_proof_c}, \eqref{eq:coherent_system_of_calabi_extensions_proof_d}, and \eqref{eq:coherent_system_of_calabi_extensions_proof_e} that property \ref{item:coherent_system_of_calabi_extensions_proof_composition} holds. It remains to show property \ref{item:coherent_system_of_calabi_extensions_proof_sum}. Since $\varphi_i^{(j)} = \varphi_i^{(j-1)}$ for $i<j$, we need to check that
     \begin{equation*}
         \prod\limits_{i=j}^n \overline{\operatorname{Ham}}^{\operatorname{ab}}(\iota_i)([\psi_i]) = \overline{\operatorname{Ham}}^{\operatorname{ab}}(\iota_j)([\psi_j^{\operatorname{sm}}]) \cdot \prod\limits_{i=j+1}^n \overline{\operatorname{Ham}}^{\operatorname{ab}}(\iota_i)([\psi_i\circ \alpha_i]).
     \end{equation*}
     Substituting $\overline{\operatorname{Ham}}^{\operatorname{ab}}(\iota_i)([\psi_i\circ \alpha_i]) = \overline{\operatorname{Ham}}^{\operatorname{ab}}(\iota_i)([\psi_i]) \cdot \overline{\operatorname{Ham}}^{\operatorname{ab}}(\iota_i)([\alpha_i])$, we see that this is equivalent to showing thatv
     \begin{equation*}
         \overline{\operatorname{Ham}}^{\operatorname{ab}}(\iota_j)([\psi_j]) = \overline{\operatorname{Ham}}^{\operatorname{ab}}(\iota_j)([\psi_j^{\operatorname{sm}}]) \cdot \prod\limits_{i=j+1}^n \overline{\operatorname{Ham}}^{\operatorname{ab}}(\iota_i)([\alpha_i]).
     \end{equation*}
     It follows from identities \eqref{eq:coherent_system_of_calabi_extensions_proof_d} and \eqref{eq:coherent_system_of_calabi_extensions_proof_g} that
     \begin{equation*}
         \overline{\operatorname{Ham}}^{\operatorname{ab}}(\iota_j)([\psi_j]) = \overline{\operatorname{Ham}}^{\operatorname{ab}}(\iota_j)([\psi_j^{\operatorname{sm}}]) \cdot \prod\limits_{i=j+1}^n \overline{\operatorname{Ham}}^{\operatorname{ab}}(\iota_j)([\alpha_i]).
     \end{equation*}
     Thus it suffices to prove
     \begin{equation*}
         \overline{\operatorname{Ham}}^{\operatorname{ab}}(\iota_i)([\alpha_i]) = \overline{\operatorname{Ham}}^{\operatorname{ab}}(\iota_j)([\alpha_i]) \qquad \text{for all $j+1\leq i \leq n$.}
     \end{equation*}
     In order to see this, simply observe that, for every $i$, there exists $\beta \in \overline{\operatorname{Ham}}(\Sigma_1)$ such that $(\iota_i)_*(\alpha_i) = \beta\circ (\iota_j)_*(\alpha_i)\circ \beta^{-1}$. This concludes the proof that the homeomorphisms $\varphi_i^{(j)}$ satisfy property \ref{item:coherent_system_of_calabi_extensions_proof_sum} and thus the proof that the definition of the map $\tilde{F}$ is independent of choices.

     \emph{Step 4 (embeddings of open into closed surfaces):} Now consider an embedding $\iota: \Sigma_1 \hookrightarrow \Sigma_2$ of an open surface $\Sigma_1$ into a closed surface $\Sigma_2$. Again, we assume that $\iota$ is smooth and identify $\Sigma_1$ with its image under $\iota$. Since the group of Hamiltonian diffeomorphisms $\operatorname{Ham}(\Sigma_2)$ of the closed surface $\Sigma_2$ is perfect, the image of $\operatorname{Ham}^{\operatorname{ab}}(\Sigma_1)$ in $\overline{\operatorname{Ham}}^{\operatorname{ab}}(\Sigma_1)$ is contained in the kernel of $\overline{\operatorname{Ham}}^{\operatorname{ab}}(\iota)$. Therefore, $\overline{\operatorname{Ham}}^{\operatorname{ab}}(\iota)$ induces a homomorphism
     \begin{equation*}
         G:\overline{\operatorname{Ham}}^{\operatorname{ab}}(\Sigma_1) / \operatorname{Ham}^{\operatorname{ab}}(\Sigma_1) \rightarrow \overline{\operatorname{Ham}}^{\operatorname{ab}}(\Sigma_2).
     \end{equation*}
     Our goal is to show that this is an isomorphism. Again, the strategy is to construct an inverse. Our construction is very similar to the one in Step 3. Given $\varphi\in \overline{\operatorname{Ham}}(\Sigma_2)$, we pick a fragmentation $\varphi = \varphi_n\circ \cdots \circ \varphi_1$ with $\varphi_i \in \overline{\operatorname{Ham}}(D_i)$ for open discs $D_i\Subset \Sigma_2$ which are sufficiently small such that they admit embeddings $\iota_i : D_i \hookrightarrow \Sigma_1$. Given these choices, we set
     \begin{equation*}
         \tilde{F}(\varphi) \coloneqq \left[\prod_i \overline{\operatorname{Ham}}^{\operatorname{ab}}(\iota_i)([\varphi_i])\right] \in\overline{\operatorname{Ham}}^{\operatorname{ab}}(\Sigma_1) / \operatorname{Ham}^{\operatorname{ab}}(\Sigma_1).
     \end{equation*}
     Again, the main difficulty is to show that this definition does not depend on the choice of fragmentation. Once we know this, we can argue as in Step 3 that $\tilde{F}$ descends to a homomorphism
     \begin{equation*}
         F : \overline{\operatorname{Ham}}^{\operatorname{ab}}(\Sigma_2) \rightarrow \overline{\operatorname{Ham}}^{\operatorname{ab}}(\Sigma_1) / \operatorname{Ham}^{\operatorname{ab}}(\Sigma_1)
     \end{equation*}
     which is an inverse to $G$. In order to show that $\tilde{F}$ is independent of the choice of fragmentation, we have to show that if $D_i \Subset \Sigma_2$ are open discs admitting embeddings $\iota_i: D_i \hookrightarrow \Sigma_1$, which we can assume to be smooth, and if $\varphi_i \in \overline{\operatorname{Ham}}(D_i)$ are Hamiltonian homeomorphisms such that $\varphi_n \circ \cdots \circ \varphi_1 = \operatorname{id}$, then $\left[\prod_i \overline{\operatorname{Ham}}^{\operatorname{ab}}(\iota_i)([\varphi_i])\right] = \operatorname{id}$ in the quotient $\overline{\operatorname{Ham}}^{\operatorname{ab}}(\Sigma_1) / \operatorname{Ham}^{\operatorname{ab}}(\Sigma_1)$. Going through the argument in Step 3 word by word, we see that there exist Hamiltonian diffeomorphisms $\varphi_i^{(n)} \in \operatorname{Ham}(D_i)$ such that $\varphi_n^{(n)} \circ \cdots \circ \varphi_1^{(n)} = \operatorname{id}$ and
     \begin{equation*}
         \prod_i \overline{\operatorname{Ham}}^{\operatorname{ab}}(\iota_i)([\varphi_i^{(n)}]) = \prod_i \overline{\operatorname{Ham}}^{\operatorname{ab}}(\iota_i)([\varphi_i])
     \end{equation*}
     in $\overline{\operatorname{Ham}}^{\operatorname{ab}}(\Sigma_1)$. Since the $\varphi_i^{(n)}$ are Hamiltonian diffeomorphisms and the embeddings $\iota_i$ can assumed to be smooth, the left hand side of this identity is contained in $\operatorname{Ham}^{\operatorname{ab}}(\Sigma_1)$. Thus the right hand side represents the identity in the quotient $\overline{\operatorname{Ham}}^{\operatorname{ab}}(\Sigma_1)/\operatorname{Ham}^{\operatorname{ab}}(\Sigma_1)$, which is exactly what we need to show for concluding that $\tilde{F}$ is well-defined and independent of choices. This finishes the proof that the kernel of $\overline{\operatorname{Ham}}^{\operatorname{ab}}(\iota)$ is given by the image of $\operatorname{Ham}^{\operatorname{ab}}(\Sigma_1)$ in the case of an open surface $\Sigma_1$ embedded into a closed surface $\Sigma_2$.

     \emph{Step 5 (two closed surfaces):} Any area-preserving embedding $\iota$ of a closed connected surface into another is necessarily an area-preserving homeomorphism, which clearly implies that $\overline{\operatorname{Ham}}^{\operatorname{ab}}(\iota)$ is an isomorphism. This concludes the proof of the theorem.
\end{proof}

\section{Smooth Hamiltonian structures}
\label{sec:smooth_hamiltonian_structures}

In this section, we introduce smooth Hamiltonian structures and discuss basic notions such as flux, helicity, and plugs.

\subsection{Basic definitions}

Let $W$ be an oriented smooth $3$-manifold without boundary and possibly open. A \textit{Hamiltonian structure} on $W$ is a closed, maximally nondegenerate $2$-form $\omega$ on $W$. 
    
The line bundle $\operatorname{ker}\omega \subset TW$ is called the \textit{characteristic line bundle} of $\omega$. The foliation generated by this line bundle is called the \textit{characteristic foliation} of $\omega$. It is naturally cooriented by $\omega$. Together with the ambient orientation of $W$, this coorientation induces a natural orientation of the characteristic foliation.
    
The $2$-form $\omega$ also induces a transverse measure on the characteristic foliation. Locally on every smooth transversal, the induced measure is diffeomorphic to the standard $2$-dimensional Lebesgue measure. The (cooriented and measured) characteristic foliation uniquely determines the $2$-form $\omega$. We can therefore equivalently think of a Hamiltonian structure as a $1$-dimensional cooriented and measured smooth foliation with the property that the induced measure on transversals is locally diffeomorphic to the $2$-dimensional Lebesgue measure.

A diffeomorphism of Hamiltonian structures $(W_0, \omega_0)$ and $(W_1, \omega_1)$ is a diffeomorphism $f:W_0 \rightarrow W_1$ which satisfies $f^* \omega_1 = \omega_0$.

\begin{example}
\label{ex:standard-structure}
    Consider $\R^3 = \R\times \R^2$ equipped with coordinates $(t,x,y)$ and oriented with respect to these coordinates.
     
    We will refer to the closed, maximally nondegenerate $2$-form $\omega_{\operatorname{std}} = dx\wedge dy$ as the {\it standard Hamiltonian structure} on $\R^3$.

    For $a>0$, let $B(a)\subset \R^2$ denote the ball of area $a$ and consider the product $(-1,1)\times B(a) \subset \R \times \R^2$.  The restriction of the standard Hamiltonian structure on $\R^3$ induces a Hamiltonian structure on $(-1,1)\times B(a) \subset \R \times \R^2$.  We will abbreviate it by the same symbol $\omega_{\operatorname{std}}$ and refer to it as the standard Hamiltonian structure on $(-1,1)\times B(a)$. 
\end{example}

\begin{example}
\label{ex:surafcexintervarl}
    Consider a surface  $\Sigma$, without boundary, equipped with an area form $\omega_\Sigma$.  Let $\operatorname{pr}_2: (0,1) \times \Sigma \rightarrow \Sigma$ be the projection onto the second factor.  We equip $(0,1) \times \Sigma$ with the orientation induced by the volume form $dt \wedge \operatorname{pr}_2^* \omega_\Sigma$, where $t$ denotes the coordinate on $(0,1)$. Then, $\operatorname{pr}_2^* \omega_\Sigma$ is a Hamiltonian structure on $(0,1)\times \Sigma$ for which the oriented leaves of the characteristic foliation are given by the flow lines of the vector field $\partial_t$. This Hamiltonian structure will be denoted by $( (0,1) \times \Sigma, \omega_\Sigma )$.

    Now, let $(\varphi^t)_{t\in [0,1]}$ be a smooth isotopy in $\Ham(\Sigma, \omega_\Sigma)$ such that $\varphi^t = \operatorname{id}_\Sigma$ for $t$ close to $0$ and $\varphi^t = \varphi^1$ for $t$ close to $1$. Let $H:[0,1]\times \Sigma\rightarrow \R$ be the unique compactly supported Hamiltonian generating $\varphi^t$. Then, the $2$-form
    \begin{equation*}
        dH \wedge dt + \omega_\Sigma
    \end{equation*}
    is another Hamiltonian structure on $(0,1) \times \Sigma$. It agrees with $\omega_\Sigma$ outside a compact set and the leaves of its characteristic foliation are of the form $\{(t,\varphi^t(p)) \mid t\in (0,1)\}$ for $p \in \Sigma$.
\end{example}

Hamiltonian structures satisfy a version of the Darboux neighborhood theorem which we state here as a lemma.

\begin{lem}
\label{lem:Draboux}
    Let $\omega$ be a Hamiltonian structure on $W$. For every point $p$ in $(W,\omega)$, there exists a neighborhood $U$ such that $(U, \omega)$ is diffeomorphic to $((-1,1)\times B(a),\omega_{\operatorname{std}})$, for some $a>0$.
\end{lem}

\begin{rem}
    For every Hamiltonian structure $\omega$ on $W$, there exist a symplectic $4$-manifold $(M, \omega')$ and an embedding $\iota: (W, \omega) \hookrightarrow (M, \omega')$ such that $\iota^* \omega' = \omega$.  Moreover, by  Gotay's theorem \cite{Gotay},  any two such embeddings $\iota_i : (W, \omega) \rightarrow (M_i, \omega_i), \; i\in\{1,2\}$, are neighborhood equivalent in the following sense: There exist open neighborhoods $\iota_i(W_i) \subset U_i \subset  M_i$ and a symplectomorphism $\phi : (U_1, \omega_1) \rightarrow (U_2, \omega_2)$ such that $\phi \circ \iota_1 = \iota_2$. 
\end{rem}

\subsection{Flux and Helicity}

Let $(Y, \omega)$ be a closed oriented $3$-manifold $Y$ equipped with a Hamiltonian structure $\omega$.

\begin{definition}
\label{def:smooth_flux}
    We define the \textit{flux} of the Hamiltonian structure $\omega$ to be the cohomology class
    \begin{equation*}
        \operatorname{Flux}(\omega) \coloneqq [\omega] \in H^2(Y;\R).
    \end{equation*}
    We say that a Hamiltonian structure $\omega$ is \textit{exact} if $\operatorname{Flux}(\omega) = 0$.
\end{definition}

Recall that the helicity of an exact Hamiltonian structure $\omega$ is defined to be the quantity
\begin{equation*}
    \mathcal{H}(\omega) \coloneqq \int_Y \alpha \wedge \omega,
\end{equation*}
where $\alpha$ is any primitive of $\omega$. We check in the lemma below that helicity is well-defined.  

\begin{lem}
    Let $\omega$ be an exact Hamiltonian structure on $Y$ and let $\alpha$ and $\beta$ be two $1$-forms such that $d\alpha = d\beta = \omega$. Then,
    \begin{equation*}
        \int_Y \alpha \wedge \omega = \int_Y \beta \wedge \omega.
    \end{equation*}
\end{lem}

\begin{proof}
    The $1$-form $\alpha - \beta$ is closed and the $2$-form $\omega$ is exact. Hence, the $3$-form $(\alpha - \beta)\wedge \omega$ is exact and consequently, by Stokes' theorem, we have
    \begin{equation*}
        \int_Y (\alpha -\beta) \wedge \omega = 0.
    \end{equation*}
\end{proof}

\subsection{Plugs}
\label{sec:smooth_plugs}

We introduce the concept of plugs, which can be \emph{inserted} into a given Hamiltonian structure to generate a new one.

\begin{definition}
\label{def:plug}
    Let $\omega$ be a Hamiltonian structure on $W$. A \textit{plug} is a tuple $\mathcal{P} = (\Sigma, \omega_\Sigma, \alpha, (\varphi^t)_{t\in [0,1]})$ consisting of
    \begin{enumerate}
        \item an open surface  $\Sigma$, not necessarily connected; 
        \item an area form $\omega_\Sigma$ on $\Sigma$;
        \item a smooth embedding of Hamiltonian structures $\alpha : ((0,1) \times \Sigma,\omega_\Sigma)\hookrightarrow (W,\omega)$;
        \item a smooth isotopy $(\varphi^t)_{t\in [0,1]}$ in $\Ham(\Sigma)$ such that $\varphi^t = \operatorname{id}_\Sigma$ for $t$ close to $0$ and $\varphi^t = \varphi^1$ for $t$ close to $1$.
    \end{enumerate}
    We define the \textit{Calabi invariant} of the plug $\mathcal{P} = (\Sigma, \omega_\Sigma, \alpha, (\varphi^t)_{t\in [0,1]})$ to be 
    \begin{equation}
    \label{eqn:plug-Calabi-defn}
        \mathrm{Cal}(\mathcal{P}):= \mathrm{Cal}_\Sigma(\varphi^1).
    \end{equation}
\end{definition}

Given a plug $\mathcal{P}$, we define a new  Hamiltonian structure $\omega\# \mathcal{P}$ which coincides with $\omega$ outside of $\operatorname{im}(\alpha)$ and with $\alpha_* (dH \wedge dt + \omega_\Sigma)$ inside $\operatorname{im}(\alpha)$, where $H$ is the unique compactly supported Hamiltonian generating $\varphi^t$. Note that $\omega\# \mathcal{P}$ is well-defined because $dH \wedge dt$ vanishes near the boundary of  $\operatorname{im}(\alpha)$. 

The Hamiltonian structure $\omega\# \mathcal{P}$ has the following equivalent definition which is better suited for the $C^0$ setting which we consider below. Define $\Phi : (0,1)\times \Sigma \rightarrow (0,1)\times \Sigma$ by $\Phi(t,p)\coloneqq (t,\varphi^t(p))$. Then $\omega\#{\mathcal{{P}}}$ is defined to agree with $\omega$ outside of $\operatorname{im}(\alpha)$ and with $\alpha_*\Phi_*\omega_\Sigma$ inside $\operatorname{im}(\alpha)$.  Note that before plug insertion, the characteristic leaves of $\omega$ inside $\operatorname{im}(\alpha)$ are of the form $\alpha((0,1)\times \{p\})$ for $p\in \Sigma$. After plug insertion, the characteristic leaves of $\omega\# \mathcal{P}$ are of the form
\begin{equation*}
    \{\alpha(t,\varphi^t(p))\mid t\in (0,1)\}
\end{equation*}
for $p \in \Sigma$.

Although the Hamiltonian structure $\omega \# \mathcal{P}$ depends on the entire isotopy $\varphi^t$, its diffeomorphism type depends only on the time-1 map $\varphi^1$.

\begin{lem}
\label{lem:plug_independence_of_hamiltonian_isotopy}
    Consider two plugs $ \mathcal{P}_1 = (\Sigma, \omega_\Sigma, \alpha, (\varphi_1^t)_{t\in [0,1]})$ and  $ \mathcal{P}_2 = (\Sigma, \omega_\Sigma, \alpha, (\varphi_2^t)_{t\in [0,1]})$.  
    If $\varphi_1^1 = \varphi_2^1$, then there exists a diffeomorphism of Hamiltonian structures $\psi: (Y,\omega\#\mathcal{P}_1) \rightarrow (Y,\omega\#\mathcal{P}_2)$ which is supported inside $\mathrm{im}(\alpha)$ and is isotopic to the identity through diffeomorphisms supported inside $\mathrm{im}(\alpha)$.
\end{lem}

We omit the proof of the above lemma, as its generalization in the $C^0$ setting is proven below in Lemma \ref{lem:C0_plugs_with_same_time_1_maps}.

The next two lemmas describe the effect of plugs on flux and helicity. Lemma \ref{lem:helicity_after_plug_insertion} may be viewed as a special case of \cite[Théorème 3.1]{gg97}.

\begin{lem}
\label{lem:flux_invariant_under_plug_insertion_smooth_case}
    Let $\omega$ be a Hamiltonian structure on a closed $3$-manifold $Y$ and let $\mathcal{P}$ be an $\omega$-plug. Then,
    \begin{equation*}
        \operatorname{Flux}(\omega) = \operatorname{Flux}(\omega\# \mathcal{P}).
    \end{equation*}
\end{lem}

\begin{proof}
    Write $\mathcal{P} = (\Sigma,\omega_\Sigma,\alpha,(\varphi^t)_{t\in [0,1]})$. Let $H$ be the unique compactly supported Hamiltonian generating $\varphi^t$. Then the $1$-form $\alpha_*(Hdt)$ is compactly supported in $\operatorname{im}(\alpha)$ and can hence be extended to a $1$-form on $Y$ by setting it equal to zero outside of $\operatorname{im}(\alpha)$. By definition, the Hamiltonian structures $\omega$ and $\omega \#\mathcal{P}$ differ by the exact $2$-form $d\alpha_*(Hdt)$. This clearly implies that $\operatorname{Flux}(\omega) = \operatorname{Flux}(\omega \# \mathcal{P})$.
\end{proof}

\begin{lem}
\label{lem:Helicity-Plug}
\label{lem:helicity_after_plug_insertion}
    Let $Y$ be a closed $3$-manifold and let $\omega$ be an exact Hamiltonian structure on $Y$. Let $\mathcal{P}$ be an $\omega$-plug. Then
    \begin{equation*}
        \mathcal{H}(\omega\#\mathcal{P}) = \mathcal{H}(\omega) + \mathrm{Cal}(\mathcal{P}).
    \end{equation*}
\end{lem}

\begin{proof}
    As in the proof of Lemma \ref{lem:flux_invariant_under_plug_insertion_smooth_case}, note that $\omega$ and $\omega\#\mathcal{P}$ differ by the exact $2$-form $d\alpha_*(Hdt)$. Let $\lambda$ be an arbitrary primitive $1$-form of $\omega$. We compute
    \begin{align}
        \mathcal{H}(\omega \# \mathcal{P}) &= \int_Y (\lambda + \alpha_*(Hdt)) \wedge (\omega + d\alpha_*(Hdt)) \nonumber \\
        &= \mathcal{H}(\omega) + \int_{(0,1)\times \Sigma} H dt \wedge \omega_\Sigma + \alpha^*\lambda \wedge d(Hdt) + Hdt \wedge d(Hdt). \label{eq:helicity_after_plug_insertion_proof}
    \end{align}
    Note that $Hdt \wedge d(Hdt) = H dt\wedge dH \wedge dt$ vanishes identically. Moreover,
    \begin{equation*}
        d (\alpha^*\lambda \wedge Hdt) = \omega_\Sigma \wedge Hdt - \alpha^*\lambda \wedge d(Hdt). 
    \end{equation*}
    By Stokes' theorem, the integral $\int_{(0,1)\times \Sigma} d(\alpha^*\lambda \wedge Hdt)$ vanishes, which implies that
    \begin{equation*}
        \int_{(0,1)\times \Sigma} \alpha^*\lambda \wedge d(Hdt) = \int_{(0,1)\times \Sigma} \omega_\Sigma \wedge Hdt.
    \end{equation*}
    Combining this with \eqref{eq:helicity_after_plug_insertion_proof}, we see that
    \begin{equation*}
        \mathcal{H}(\omega\#\mathcal{P}) = \mathcal{H}(\omega) + 2 \int_{(0,1)\times \Sigma} Hdt \wedge \omega_\Sigma = \mathcal{H}(\omega) + \operatorname{Cal}(\varphi_H^1) = \mathcal{H}(\omega) + \operatorname{Cal}(\mathcal{P}).
    \end{equation*}
\end{proof}

\section{\texorpdfstring{$C^0$}{C0} Hamiltonian structures}
\label{sec:C0_hamiltonian_structures}

In this section, we introduce $C^0$ Hamiltonian structures and $C^0$ plugs.

\subsection{Basic definitions}

Recall the definition of the standard Hamiltonian structure on $\R^3$ from Example \ref{ex:standard-structure}. A $C^0$ Hamiltonian structure on a topological $3$-manifold $W$ is a $1$-dimensional cooriented $C^0$ foliation, equipped with a transverse measure, which is locally modeled on $\R^3$ with the characteristic foliation induced by $\omega_{\operatorname{std}}$. Explicitly, this notion can be formalized in terms of $C^0$ Hamiltonian atlases.

\begin{definition}
\label{def:C0_Hamiltonian_atlas}
    Consider $\R^3 = \R\times \R^2$ equipped with the smooth Hamiltonian structure $\omega_{\operatorname{std}}$. Let $\psi: U \rightarrow V$ be a homeomorphism between two open subsets of $\R^3$. We say that \textit{$\psi$ preserves the standard Hamiltonian structure $\omega_{\operatorname{std}}$ on $\R^3$} if, for every point $p\in U$, there exist an open neighborhood of $p$ of the form $I\times B$, where $I\subset \R$ is an open interval and $B\subset \R^2$ is a ball, and a continuous area- and orientation-preserving embedding $a:B\hookrightarrow \R^2$ such that
    \begin{equation*}
        \psi(t,z) = (*, a(z)) \qquad \text{for all $(t,z)\in I\times B$.}
    \end{equation*}
    A \textit{$C^0$ Hamiltonian atlas} on a topological $3$-manifold is an atlas whose transition maps are homeomorphisms preserving the standard Hamiltonian structure on $\R^3$. Two such atlases are considered equivalent if the transition maps between the two atlases preserve the standard Hamiltonian structure on $\R^3$. A \textit{$C^0$ Hamiltonian structure $\Omega$} on $W$ is an equivalence class of $C^0$ Hamiltonian atlases.
\end{definition}

\begin{definition}
\label{def:homeo_C0_Ham_Structure}
    A \textit{homeomorphism of $C^0$ Hamiltonian structures} $\psi : (W_0,\Omega_0) \rightarrow (W_1,\Omega_1)$ is a homeomorphism $\psi: W_0  \rightarrow W_1$ which, in local charts, preserves the standard Hamiltonian structure on $\R^3$.
    A \textit{topological embedding of $C^0$ Hamiltonian structures} $\alpha:(W_0,\Omega_0)\hookrightarrow (W_1,\Omega_1)$ is a topological embedding $\alpha:W_0\hookrightarrow W_1$ which is a homeomorphism of the $C^0$ Hamiltonian structures  $(W_0,\Omega_0)$  and $(\alpha(W_0), \Omega_1)$.
\end{definition}

Let $\psi: W_0 \rightarrow W_1$ be a homeomorphism. Given a $C^0$ Hamiltonian structure $\Omega_1$ on $W_1$, we can pull back its defining atlas, via $\psi$, to obtain a $C^0$ Hamiltonian structure on $W_0$ which we abbreviate by $\psi^* \Omega_1$ and refer to as  the \emph{pullback} of $\Omega_1$ via $\psi$.  We define the \emph{pushforward} of a $C^0$ Hamiltonian structure $\Omega_0$ on $W_0$ to be the pullback of $\Omega_0$ via $\psi^{-1}$; we abbreviate this by $\psi_* \Omega_0$.  

\begin{rem}
    A smooth Hamiltonian structure $\omega$ on a smooth $3$-manifold $W$ gives rise to a $C^0$ Hamiltonian structure. Indeed, consider the atlas of $W$ consisting of all Darboux charts. The transition maps of this atlas are diffeomorphisms $\psi:U\rightarrow V$ between open subsets of $\R^3$ which satisfy $\psi^*\omega_{\operatorname{std}} = \omega_{\operatorname{std}}$ as smooth $2$-forms. Note that for a diffeomorphism $\psi$, this condition is equivalent to preserving the standard Hamiltonian structure $\omega_{\operatorname{std}}$ on $\R^3$ in the sense of Definition \ref{def:C0_Hamiltonian_atlas}. Hence our atlas consisting of smooth Darboux charts is a $C^0$ Hamiltonian atlas.
\end{rem}

\begin{rem}
\label{rem:smooth-structure-on-topological-mfld}
    Let $W$ be a topological $3$-manifold. In what follows, by a smooth Hamiltonian structure $\omega$ on $W$ we mean a choice of smooth structure on $W$ and a maximally nondegenerate closed $2$-form. Alternatively, a smooth Hamiltonian structure on a topological $3$-manifold can be specified by an atlas on $W$ whose transition maps are diffeomorphisms preserving the standard Hamiltonian structure on $\R^3$.
\end{rem}

\subsection{\texorpdfstring{$C^0$}{C0} plugs}
\label{subsec:C0_plugs}

Analogous to the smooth plugs from Section \ref{sec:smooth_plugs}, we introduce the concept of  $C^0$  plugs, which can be \emph{inserted} into  $C^0$  Hamiltonian structures to generate new $C^0$  Hamiltonian structures.  We also define the \emph{Calabi invariant} of $C^0$ plugs.

\medskip

Throughout the rest of this section, let $W$ denote a topological $3$-manifold without boundary and possibly open.

\begin{definition}
\label{def:C0_plugs}
    Let $\Omega$ be a $C^0$ Hamiltonian structure on $W$. An \textit{$\Omega$-plug} (or simply a \textit{plug}, if $\Omega$ is clear from the context) is a tuple $\mathcal{P} = (\Sigma, \omega_\Sigma, \alpha, (\varphi^t)_{t\in [0,1]})$ consisting of
    \begin{enumerate}
        \item a smooth, open surface $\Sigma$, not necessarily connected;
        \item a smooth area form $\omega_\Sigma$ on $\Sigma$;
        \item a topological embedding of $C^0$ Hamiltonian structures $\alpha : ((0,1) \times \Sigma,\omega_\Sigma)\hookrightarrow (W,\Omega)$;
        \item a continuous isotopy $(\varphi^t)_{t\in [0,1]}$ in $\overline{\operatorname{Ham}}(\Sigma, \omega_\Sigma)$ such that $\varphi^t = \operatorname{id}_\Sigma$ for $t$ close to $0$ and $\varphi^t = \varphi^1$ for $t$ close to $1$.
    \end{enumerate}

    We define the \textit{Calabi invariant} of $\mathcal{P}$ to be
    \begin{equation*}
        \overline{\operatorname{Cal}}(\mathcal{P}) \coloneqq \overline{\operatorname{Cal}}_{\Sigma}(\varphi^1) \in \mathcal{R}.
    \end{equation*}
    Here, we recall that $\overline{\operatorname{Cal}}_\Sigma : \overline{\operatorname{Ham}}(\Sigma) \rightarrow \mathcal{R}$ denotes the universal $\mathcal{R}$-valued extension of the Calabi homomorphism; Section \ref{sec:hamiltonian_homeos_and_calabi}.
\end{definition}

We now describe how an $\Omega$-plug $\mathcal{P}$ can be \emph{inserted} into $\Omega$ yielding a new $C^0$ Hamiltonian structure $\Omega\#{\mathcal{P}}$.  Define $\Phi : (0,1)\times \Sigma \rightarrow (0,1)\times \Sigma$ by $$\Phi(t,p)\coloneqq (t,\varphi^t(p)).$$ Inside $\operatorname{im}(\alpha)$, we define $\Omega\#\mathcal{{P}}$ as the pushforward $(\alpha\circ\Phi)_*\omega_\Sigma$, where $\omega_\Sigma$ is regarded as a $C^0$ Hamiltonian structure on $(0,1)\times \Sigma$. Note that $(\alpha\circ\Phi)_*\omega_\Sigma$ agrees with $\Omega$ near the boundary of $\operatorname{im}(\alpha)$. We define $\Omega\# \mathcal{P}$ to agree with $\Omega$ outside $\operatorname{im}(\alpha)$. The effect of plug insertion on the characteristic foliation is as follows:  Before plug insertion, the characteristic leaves of $\Omega$ inside $\operatorname{im}(\alpha)$ are of the form $\alpha((0,1)\times \{p\})$ for $p\in \Sigma$. After plug insertion, the characteristic leaves of $\omega\# \mathcal{P}$ are of the form
\begin{equation*}
    \{\alpha(t,\varphi^t(p))\mid t\in (0,1)\}
\end{equation*}
for $p \in \Sigma$.

The \textit{inverse} of an $\Omega$-plug $\mathcal{P}$ is defined to be the $\Omega\#\mathcal{P}$-plug
\begin{equation}\label{eqn:inverse-plug}
    \overline{\mathcal{P}} \coloneqq (\Sigma,\omega_\Sigma, \alpha\circ \Phi, ((\varphi^t)^{-1})_{t\in [0,1]}).
\end{equation}
Note that if $\mathcal{P}$ is an $\Omega$-plug, then $\Omega\#\mathcal{P}\#\overline{\mathcal{P}} = \Omega$ and $\overline{\overline{\mathcal{P}}} = \mathcal{P}$.  A plug $\mathcal{P}$ is called \textit{trivial} if $\varphi^t = \operatorname{id}$ for all $t\in [0,1]$. If $\mathcal{P}$ is trivial, then $
\Omega\# \mathcal{P} = \Omega$. If $\psi: (W',\Omega') \rightarrow (W,\Omega)$ is a homeomorphism of $C^0$ Hamiltonian structures, then the \textit{pull back} $\psi^*\mathcal{P}$ of the $
\Omega$-plug $\mathcal{P}$ via $\psi$ is the $\Omega'$-plug defined by
\begin{equation}\label{eqn:pullback-plug}
    \psi^*\mathcal{P} \coloneqq (\Sigma,\omega_\Sigma,\psi^{-1} \circ \alpha,(\varphi^t)_{t\in [0,1]}).
\end{equation}
We observe that $\psi$ is a homeomorphism of $C^0$ Hamiltonian structures
\begin{equation}\label{eqn:pullback-plug-homeos}
    \psi: (W',\Omega'\#\psi^*\mathcal{P}) \rightarrow (W,\Omega\#\mathcal{P}).
\end{equation}

\begin{rem}
\label{rem:shrinking_plugs}
   Consider an $\Omega$-plug $\mathcal{P} = (\Sigma, \omega_\Sigma, \alpha, (\varphi^t)_{t\in [0,1]})$. It is always possible to slightly shrink the image of $\mathcal{P}$ without affecting $\Omega \# \mathcal{P}$ or the Calabi invariant $\overline{\operatorname{Cal}}(\mathcal{P})$ as follows: Pick a relatively compact open subset $\Sigma' \Subset \Sigma$ such that the entire isotopy $\varphi^t$ is supported in $\Sigma'$. Let $\varepsilon>0$ such that $\varphi^t$ is equal to the identity for all $t\in (0,2\varepsilon]$ and equal to $\varphi^1$ for all $t \in [1-2\varepsilon,1)$. Let $\tau: (0,1)\rightarrow (\varepsilon,1-\varepsilon)$ be an orientation-preserving diffeomorphism which agrees with the identity on the interval $[2\varepsilon,1-2\varepsilon]$. Define an embedding of $C^0$ Hamiltonian structures $\alpha' : ((0,1)\times \Sigma',\omega_\Sigma) \hookrightarrow (W,\Omega)$ by setting $\alpha'(t,p) \coloneqq \alpha(\tau(t),p)$. Set $\mathcal{P}' \coloneqq (\Sigma',\omega_\Sigma,\alpha',(\varphi^t)_{t\in [0,1]})$. The closure of the image of $\mathcal{P}'$ is clearly contained in the image of $\mathcal{P}$. Moreover, $\Omega \# \mathcal{P}' = \Omega \# \mathcal{P}$. Finally, by the naturality of the $\mathcal{R}$-valued extension of Calabi, we have $\overline{\operatorname{Cal}}(\mathcal{P'}) = \overline{\operatorname{Cal}}(\mathcal{P})$.
\end{rem}

Analogously to Lemma \ref{lem:plug_independence_of_hamiltonian_isotopy} in the smooth setting, although the $C^0$ Hamiltonian structure $\omega \# \mathcal{P}$ depends on the entire isotopy $\varphi^t$, its homeomorphism type depends only on the time-1 map $\varphi^1$.

\begin{lem}
\label{lem:C0_plugs_with_same_time_1_maps}
    Let $\Omega$ be a $C^0$ Hamiltonian structure on $W$. For $i \in \{0,1\}$, let $\mathcal{P}_i = (\Sigma,\omega_\Sigma,\alpha,(\varphi_i^t)_{t\in [0,1]})$ be two $\Omega$-plugs. If $\varphi_0^1 = \varphi_1^1$, then there exists a homeomorphism of $C^0$ Hamiltonian structures $\psi:(W,\Omega\#\mathcal{P}_0) \rightarrow (W,\Omega\#\mathcal{P}_1)$ which is supported inside $\operatorname{im}(\alpha)$ and isotopic to the identity through homeomorphisms supported in $\operatorname{im}(\alpha)$.
\end{lem}

\begin{proof}
    Define $\Phi_i : (0,1)\times \Sigma \rightarrow (0,1)\times \Sigma$ by $\Phi_i(t,p) \coloneqq (t,\varphi_i^t(p))$. Then
    \begin{equation*}
        \Phi_1 \circ \Phi_0^{-1} : ((0,1)\times \Sigma, (\Phi_0)_*\omega_\Sigma) \rightarrow ((0,1)\times \Sigma, (\Phi_1)_*\omega_\Sigma)
    \end{equation*}
    is a homeomorphism of $C^0$ Hamiltonian structures. Since $\varphi_0^1 = \varphi_1^1$, it is compactly supported. We can therefore define the desired homeomorphism $\psi: (W,\Omega\#\mathcal{P}_0) \rightarrow (W,\Omega\#\mathcal{P}_1)$ by setting $\psi\coloneqq \alpha\circ \Phi_1 \circ \Phi_0^{-1} \circ \alpha^{-1}$ inside $\operatorname{im}(\alpha)$ and extending it to be the identity on the complement of this set.

    Now, the surface $\Sigma$ is open and therefore not  $S^2$. Hence $\overline{\operatorname{Ham}}(\Sigma)$ is simply connected, as explained in Section \ref{sec:prelims-surfaces}. It follows from simply connectedness of  $\overline{\operatorname{Ham}}(\Sigma)$ that the homeomorphism $\psi$ is isotopic to the identity via homeomorphisms supported in $\operatorname{im}(\alpha)$.
\end{proof}

The next lemma states that sliding a plug along the characteristic foliation has no effect on the homeomorphism type of the $C^0$ Hamiltonian structure.

\begin{lem}
\label{lem:C0_plugs_with_isotopic_embeddings}
    Let $\Omega$ be a $C^0$ Hamiltonian structure on $W$. For $i\in \{0,1\}$, let $\mathcal{P}_i = (\Sigma, \omega_\Sigma, \alpha_i,(\varphi^t)_{t\in [0,1]})$ be two $\Omega$-plugs. Assume that there exists a continuous family $(\alpha_s)_{s\in [0,1]}$ of topological embeddings of $C^0$ Hamiltonian structures $\alpha_s : ((0,1)\times \Sigma,\omega_\Sigma)\hookrightarrow (W,\Omega)$ connecting $\alpha_0$ to $\alpha_1$ such that $\alpha_s(t,p)$ is contained in the same characteristic leaf of $\Omega$ as $\alpha_0(t,p)$ for all $s\in [0,1]$ and $(t,p)\in (0,1)\times \Sigma$. Then, there exists a homeomorphism of $C^0$ Hamiltonian structures $\psi:(W,\Omega\#{\mathcal{P}_0}) \rightarrow (W,\Omega\#{\mathcal{P}_1})$ which is supported inside $U\coloneqq \bigcup_{s\in [0,1]} \operatorname{im}(\alpha_s)$ and isotopic to the identity through homeomorphisms supported in $U$.
\end{lem}

\begin{proof}
    Define $\Phi: (0,1)\times \Sigma \rightarrow (0,1)\times \Sigma$ by $\Phi(t,p)\coloneqq (t,\varphi^t(p))$. Let $K\subset (0,1)\times \Sigma$ be a compact subset such that $\Phi_*\omega_\Sigma$ agrees with $\omega_\Sigma$ on the complement of $K$. As a consequence of the assumption on the embeddings $\alpha_s$, we can find a family of homeomorphisms $\psi_s : (W,\Omega) \rightarrow (W,\Omega)$ starting at the identity obtained by moving points along the characteristic leaves of $\Omega$ such that $\psi_s$ is supported inside $U$ and such that $\psi_s\circ \alpha_0(t,p) = \alpha_s(t,p)$ for all $(t,p)\in K$. Then, $\psi\coloneqq \psi_1$ is a homeomorphism $(W,\Omega\#{\mathcal{P}_0}) \rightarrow (W,\Omega\#{\mathcal{P}_1})$ supported inside $U$. By construction, $\psi$ is isotopic to the identity through homeomorphisms supported in $U$.
\end{proof}

\section{Smoothings modulo plugs}
\label{sec:smoothings-mod-plugs}

In this section we prove Theorems \ref{thm:from_C0_to_smooth_plus_plug} \& \ref{thm:smoothing_homeomorphisms} on smoothing $C^0$ Hamiltonian structures and their homeomorphisms.

\subsection{Smoothing \texorpdfstring{$C^0$}{C0} Hamiltonian structures}
\label{sec:smoothing-c0-structure}

We prove in this section that every $C^0$ Hamiltonian structure on a closed $3$-manifold can be obtained from a smooth Hamiltonian structure  via the insertion of a $C^0$ plug.  This result, stated below, plays a crucial role in the remainder of the paper.

\begin{thm}
\label{thm:from_C0_to_smooth_plus_plug}
    Let $\Omega$ be a $C^0$ Hamiltonian structure on a closed topological $3$-manifold $Y$. Then, there exist a smooth Hamiltonian structure $\omega$ on $Y$ and a $C^0$ $\omega$-plug $\mathcal{P}$ such that $\Omega = \omega \# \mathcal{P}$.
\end{thm}

\begin{rem}
   Recall that a smooth Hamiltonian structure $\omega$ on $Y$ includes the choice of a smooth structure on $Y$, see Remark \ref{rem:smooth-structure-on-topological-mfld}. Thus Theorem \ref{thm:from_C0_to_smooth_plus_plug} in particular contains the statement that the closed topological $3$-manifold $Y$ can be smoothed. It is of course well-known that every topological $3$-manifold admits a smoothing. This follows from theorems of Moise \cite[Theorem 1 \& 3]{moi52} (see also \cite[\S 23, Theorem 1 \& \S 35, Theorem 3]{moi77}), saying that every topological $3$-manifold can be triangulated and that every triangulated $3$-manifold is piecewise linear, and the fact that every piecewise linear $3$-manifold can be smoothed, see e.g. \cite[Theorem 3.10.8]{thu97}. As we will see, the existence of a $C^0$ Hamiltonian structure $\Omega$ on $Y$ actually simplifies the problem of smoothing the underlying manifold $Y$. Essentially, it reduces the problem to smoothing a topological $2$-manifold.
\end{rem}

\begin{rem}
    Theorem \ref{thm:from_C0_to_smooth_plus_plug} is clearly false if we do not allow for the insertion of a $C^0$ plug, i.e.\ there exist $C^0$ Hamiltonian structures which cannot be globally upgraded to smooth Hamiltonian structures. Indeed, consider a closed surface $(\Sigma,\omega_\Sigma)$ and an area-preserving homeomorphism $\varphi$ of $\Sigma$ which is not $C^0$ conjugate to any area-preserving diffeomorphism.  For example, we can take $\varphi$ to have infinite topological entropy. The Hamiltonian structure $\omega_\Sigma$ on $[0,1]\times \Sigma$ descends to a $C^0$ Hamiltonian structure on the mapping torus
    \begin{equation*}
        Y_\varphi \coloneqq [0,1]\times \Sigma/\sim \qquad (1,p)\sim (0,\varphi(p)).
    \end{equation*}
    This $C^0$ Hamiltonian structure is not homeomorphic to any smooth Hamiltonian structure.
\end{rem}

Let $Y$ and $\Omega$ be as in the statement of Theorem \ref{thm:from_C0_to_smooth_plus_plug}. Let $\varphi^t$ be a fixed-point-free topological flow whose flow lines equipped with the orientation induced by $\varphi^t$ agree with the oriented characteristic leaves of $\Omega$. The existence of $\varphi^t$ is guaranteed by Proposition \ref{prop:flow_on_C0_foliations}.  

\begin{definition}
\label{defn:exhaustive-set}
   We say a subset $A \subset (Y,\Omega)$ is \emph{exhaustive} if there exists a constant $T>0$ such that, for any point $p\in Y$, there exist times $-T<t_-<0<t_+<T$ such that $\varphi^{t_\pm}(p)\in A$.
\end{definition}

We remark that whether or not  $A \subset (Y,\Omega)$ is exhaustive is independent of the choice of the flow $\varphi^t$.  

\begin{rem}
\label{rem:from_C0_to_smooth_plus_plug_stronger}
    Our proof of Theorem \ref{thm:from_C0_to_smooth_plus_plug} shows the following slightly stronger statement, which we record here for later reference: Let $(Y,\Omega)$ be a closed topological $3$-manifold equipped with a $C^0$ Hamiltonian structure $\Omega$. Let $U\subset (Y,\Omega)$ be an arbitrary exhaustive open set. Then we can find a smooth Hamiltonian structure $\omega$ on $Y$ and an $\omega$-plug $\mathcal{P}$ whose image is contained in $U$ such that $\Omega = \omega \# \mathcal{P}$.
\end{rem}

The following lemma guarantees the existence of an exhaustive flow box.

\begin{lem}
\label{lem:exhaustive-surface}
    Let $\Omega$ be a $C^0$ Hamiltonian structure on a closed topological $3$-manifold $Y$, and let $U \subset Y$ be an open exhaustive subset. Then there exist an open surface $(\Sigma, \omega_\Sigma)$ equipped with an area form, a compact subsurface with boundary $K \subset \Sigma$, and a topological embedding of $C^0$ Hamiltonian structures  
    \begin{equation*}
        \alpha : ((0,1) \times \Sigma, \omega_\Sigma) \hookrightarrow (U, \Omega)
    \end{equation*}
such that the image $\alpha((0,1) \times K)$ is exhaustive.
\end{lem}

\begin{proof}
    Since $U$ is exhaustive, we can fix $T>0$ as in Definition \ref{defn:exhaustive-set}. Consider $p \in Y$. Then there exist $-T < t_- < 0 < t_+ < T$ such that $\varphi^{t_\pm}(p) \in U$. For $a>0$ sufficiently small, we can find embeddings of $C^0$ Hamiltonian structures $\alpha_\pm : ((-1,1)\times B(2a),\omega_{\operatorname{std}})\hookrightarrow (U,\Omega)$ such that $\alpha_\pm(0,0) = \varphi^{t_\pm}(p)$. Moreover, we can find an open neighborhood $V_p$ of $p$ with the property that if we start at any point $q \in V_p$ and follow the flow $\varphi^t$ forward/backward in time, we meet $\alpha_\pm(\{0\}\times B(a))$ within time at most $T$.
    
    Since $Y$ is compact, we can cover it by finitely many neighborhoods $V_p$. This implies that there exists a finite collection of embeddings of $C^0$ Hamiltonian structures $\alpha_i : ((-1,1)\times B(2a_i),\omega_{\operatorname{std}}) \hookrightarrow (U,\Omega)$ with $i$ ranging from $1$ to $n$ such that the union $\bigcup_i \alpha_i(\{0\}\times B(a_i))$ is an exhaustive set.  We remark that the surfaces $\alpha_i(\{0\}\times B(a_i))$ are not necessarily pairwise disjoint.

    Our next step is to construct, for each $i$, finitely many closed discs of the form $\{s_i^j\} \times D_i^j \subset (-1,1) \times B(2a_i)$ such that $B(a_i) \subset \bigcup_j D_i^j$ and the images $\alpha_i(\{s_i^j\} \times D_i^j)$, ranging over all $i, j$, are pairwise disjoint.

    The construction of these discs relies on Claim \ref{cl:disjoint-discs-covering}, stated below. Let $a>0$ and consider a closed subset $S \subset ((-1,1)\times B(2a),\omega_{\operatorname{std}})$ which is a $2$-dimensional topological submanifold, possibly with boundary, and which is transverse to the characteristic foliation on $(-1,1)\times B(2a)$ in the following sense:  near every point of $S$ there exists a local $C^0$ Hamiltonian chart in which both $S$ and the characteristic foliation are smooth and in which $S$ is transverse to the characteristic foliation in the usual sense. 
    
    \begin{claim}
    \label{cl:disjoint-discs-covering}
        We can find finitely many pairwise disjoint closed discs in $(-1,1)\times B(2a)$ of the form $\{s^j\}\times D^j$ which are all disjoint from $S$ and such that $B(a) \subset \bigcup_j D^j$.
    \end{claim}
    
    \begin{proof}[Proof of Claim \ref{cl:disjoint-discs-covering}]
        Indeed, consider an arbitrary point $p \in B(2a)$. It follows from the transversality assumption on $S$ that the characteristic leaf $(-1,1)\times \{p\}$ contains a point in the complement of $S$. This means that we can find a small closed disc $D\subset B(2a)$ containing $p$ and $s\in (-1,1)$ such that $\{s\}\times D$ is disjoint from $S$. Since $\overline{B}(a)$ is compact, we can find finitely many discs $\{s^j\}\times D^j$ in the complement of $S$ such that the $D^j$ cover $B(a)$. We can in addition make these discs pairwise disjoint by slightly perturbing the heights $s^j$ so that the $s^j$ are all distinct. This proves the claim.
    \end{proof}

    Returning to the construction of the discs $\{s_i^j\}\times D_i^j$, let $1\leq i \leq n$ and assume that the discs $\{s_k^j\}\times D_k^j$ have already been construction for all $k<i$. Now simply set
    \begin{equation*}
        S\coloneqq \alpha_i^{-1}(\bigcup_{k<i} \bigcup_j \alpha_k(\{s_k^j\}\times D_k^j)) \subset (-1,1)\times B(2a_i)
    \end{equation*}
    and apply the above claim to find the discs $\{s_i^j\}\times D_i^j$. This concludes our construction of the discs $\{s_i^j\}\times D_i^j$. We observe that since $\bigcup_i \alpha_i(\{0\}\times B(a_i))$ is an exhaustive set and $B(a_i)$ is contained in the union $\bigcup_j D_i^j$, the set $\bigcup_{i,j} \alpha_i(\{s_i^j\}\times D_i^j)$ is exhaustive as well.

    For each disc $D_i^j$, pick a slightly larger open disc $B_i^j$. Let $\varepsilon >0$ be small. Consider the embedding
    \begin{equation*}
        \tilde \alpha : \bigsqcup_{i,j} (s_i^j,s_i^j+\varepsilon) \times B_i^j \hookrightarrow U
    \end{equation*}
    which on each component is given by the restriction of the corresponding embedding $\alpha_i$. Here we have to choose $\varepsilon>0$ sufficiently small and take the enlarged discs $B_i^j$ sufficiently close to the original discs $D_i^j$ to ensure that this indeed defines an embedding and has image contained in $U$.

    We define $K$ to be the disjoint union of all the discs $D_i^j$ and $\Sigma$ to be the disjoint union of all the discs $B_i^j$. Then the desired embedding $\alpha : (0,1)\times \Sigma \hookrightarrow U$ is obtained from the above embedding $\alpha$ by reparametrizing the intervals $(s_i^j,s_i^j+\varepsilon)$ to $(0,1)$.
\end{proof}

We now present a proof of Theorem \ref{thm:from_C0_to_smooth_plus_plug}.

\begin{proof}[Proof of Theorem \ref{thm:from_C0_to_smooth_plus_plug}]
    By Lemma \ref{lem:exhaustive-surface}, applied to the exhaustive set $U=Y$, we may pick an open surface $(\Sigma,\omega_\Sigma)$, a compact subsurface with boundary $K\subset\Sigma$, and an embedding of $C^0$ Hamiltonian structures $\alpha: ((0,1)\times \Sigma,\omega_\Sigma) \hookrightarrow (Y,\Omega)$ such that $\alpha((0,1)\times K)$ is exhaustive. In order to prove Remark \ref{rem:from_C0_to_smooth_plus_plug_stronger}, we note that if we are given an arbitrary exhaustive open subset $U\subset (Y,\Omega)$, we can choose $\alpha$ such that its image is contained in $U$. Let $0<t_0<t_1<1$ and define the compact interval $I \coloneqq [t_0,t_1]$. Then the set $A \coloneqq \alpha(I\times K)$ is also exhaustive. Up to possibly enlarging $K$, we can assume that the interior of $A$ is still exhaustive.
    
    Set $W\coloneqq Y\setminus A$ and equip it with the restriction of the $C^0$ Hamiltonian structure $\Omega$. Our goal is to construct a smoothing of the $C^0$ Hamiltonian structure $(W,\Omega)$. The crucial observation that allows us to achieve this is that the characteristic foliation of $(W,\Omega)$ exhibits no complicated recurrent behavior. The closure (with respect to $Y$) of any characteristic leaf of $(W,\Omega)$ is an embedded compact interval which starts at $\alpha(\{t_1\}\times K)$ and ends at $\alpha(\{t_0\}\times K)$.
    
    Let $\mathcal{L}$ denote the leaf space of the characteristic foliation of $(W,\Omega)$ and let $\operatorname{pr}:W \rightarrow \mathcal{L}$ denote the natural projection.
    
    \begin{claim}
        The leaf space $\mathcal{L}$ has the structure of a $2$-dimensional, topological, possibly non-Hausdorff manifold equipped with a measure which in local charts is homeomorphic to the standard Lebesgue measure on $\R^2$. This means that $\mathcal{L}$ admits an atlas whose transition maps are area-preserving homeomorphisms between open subsets of $\R^2$.
    \end{claim}

    \begin{rem}
        It is not surprising to encounter non-Hausdorff manifolds in the context of leaf spaces, see for example the work of Haefliger and Reeb \cite{hr57}, which studies non-Hausdorff $1$-manifolds in connection with foliations of $\R^2$.
    \end{rem}

    \begin{proof}
        Consider a point $p \in W$. Since $W$ is an open subset of $Y$, for $b>0$ sufficiently small, we can chose an embedding of $C^0$ Hamiltonian structures $\beta : ( (-1,1)\times B(b), \omega_{\operatorname{std}})\hookrightarrow (W,\Omega)$ such that $\beta(0,0) = p$. We show that after possibly shrinking $b$, we can assume that the restriction of the projection $\operatorname{pr}: W\rightarrow \mathcal{L}$ to $\beta(\{0\}\times B(b))$ is injective. Indeed, assume by contradiction that this is not the case. Then there exist a sequence of points $p_k \in \beta(\{0\}\times B(k^{-1}b))$ and a sequence of times $t_k \in \R \setminus \{0\}$ such that $\varphi^{t_k}(p_k) \in \beta(\{0\}\times B(k^{-1}b))$ and such $\varphi^t(p_k)$ remains in $W$ as $t$ ranges from $0$ to $t_k$.  Here, $\varphi^t$ is a choice of flow as in Definition \ref{defn:exhaustive-set}. Since the complement of $W$ is exhaustive, the sequence $t_k$ must be bounded. Clearly, the sequence must also be bounded away from zero. After passing to a subsequence, we can therefore assume that $t_k$ converges to $t_* \in \R\setminus\{0\}$. This implies that $p$ is a periodic orbit of $\varphi^t$ of period $t_*$. Moreover, this periodic orbit must be contained in the closure of $W$ since it is the limit of flow line segments contained in $W$. But this contradicts the fact that the complement of the closure of $W$ (i.e.\ the interior of $A$) is exhaustive.

        The above discussion yields a continuous injective map
        \begin{equation*}
            \gamma: B(b) \cong \{0\}\times B(b) \subset (-1,1)\times B(b) \overset{\beta}{\rightarrow} W \overset{\operatorname{pr}}{\rightarrow} \mathcal{L}
        \end{equation*}
        mapping $0$ to $\operatorname{pr}(p)$. We claim that this map is open. This amounts to showing that $\operatorname{pr}^{-1}(\gamma(V)) \subset W$ is open for every open subset $V\subset B(b)$. In order to see this, consider a point $q \in \operatorname{pr}^{-1}(\gamma(V))$. The point $q$ being in $\operatorname{pr}^{-1}(\gamma(V))$ means that there exists a line segment contained in a characteristic leaf of $(W,\Omega)$ which connects $q$ to a point in $\beta(\{0\}\times V)$. Since $W$ and $V$ are open, every point $q'$ in a neighborhood of $q$ can also be connected to $\beta(\{0\}\times V)$ via a line segment contained in a leaf of $(W,\Omega)$. Thus $\operatorname{pr}^{-1}(\gamma(V))$ contains a neighborhood of $q$, showing that $\operatorname{pr}^{-1}(\gamma(V))$ is open.
        
        We conclude that $\gamma$ is a homeomorphism between $B(b)$ and an open neighborhood of $\operatorname{pr}(p)$ in $\mathcal{L}$. Since $p\in W$ was chosen arbitrarily, this implies that $\mathcal{L}$ is a $2$-dimensional, topological, possibly non-Hausdorff manifold. The transverse measure on the characteristic foliation of $(W,\Omega)$ clearly descends to a measure on $\mathcal{L}$ locally homeomorphic to the standard $2$-dimensional Lebesgue measure.
    \end{proof}
    
    \begin{claim}
    \label{cl:leaf-space-smoothing}
        $\mathcal{L}$ has a smoothing, i.e.\ it admits an atlas whose transition maps are area-preserving diffeomorphisms between open subsets of $\R^2$.
    \end{claim}
    
    \begin{proof}
        It is well-known that every Hausdorff topological $2$-manifold admits a smoothing. Indeed, by a theorem of Rad\'{o} \cite{rad25} (see also \cite[\S 8, Theorem 3]{moi77}) every Hausdorff topological $2$-manifold admits a piecewise linear structure. Moreover, piecewise linear $2$-manifolds, which are automatically Hausdorff, can be smoothed, see e.g.\ \cite[Theorem 3.10.8]{thu97}. Claim \ref{cl:leaf-space-smoothing} cannot be directly deduced from these results since it involves non-Hausdorff manifolds and area-preserving homeomorphisms/diffeomorphisms. For this reason, we provide a proof of Claim \ref{cl:leaf-space-smoothing}. As we will now explain, Claim \ref{cl:leaf-space-smoothing} essentially boils down to the fact that area-preserving homeomorphisms can be approximated by area-preserving diffeomorphisms, see Proposition \ref{prop:smooth_approximation_area_preserving_homeos}.
        
        Note that $\mathcal{L}$ admits a finite open covering by charts $\varphi_i : U_i \rightarrow V_i\subset \R^2$ with $1\leq i \leq n$. In the following, it will be useful to be able to replace these charts by smaller charts $\varphi_i|_{U_i'}:U_i'\rightarrow \varphi_i(U_i')$ for relatively compact open subsets $U_i' \Subset U_i$ which still cover $\mathcal{L}$. We can choose our initial charts $\varphi_i$ such that this is possible. Note that in this situation we can shrink the charts in such a way that it is possible to repeat the chart shrinking for the resulting cover by charts.
        
        For all $i,j$, we set $V_{ji} \coloneqq \varphi_i(U_i\cap U_j)$ and let
        \begin{equation*}
            \varphi_{ji} : V_{ji} \rightarrow V_{ij} \qquad \varphi_{ji} \coloneqq \varphi_j \circ \varphi_i^{-1}
        \end{equation*}
        denote the transition map, which is an area-preserving homeomorphism. Our goal is to modify the charts $\varphi_i$ such that the transition maps $\varphi_{ji}$ become smooth for all $i,j$. Let $1 \leq j < i \leq n$. Suppose we have already modified the charts in such a way that $\varphi_{\ell k}$ is smooth for all $k,\ell < i$ and such that $\varphi_{ki}$ is smooth for all $k<j$. We explain how to modify the charts to make $\varphi_{ji}$ smooth without destroying smoothness of the transition maps we have already made smooth. Set $V\coloneqq V_{ji} \cap \bigcup_{k<j} V_{ki}$. It follows from our assumptions that the restriction of $\varphi_{ji}$ to $V$ is smooth. After slightly shrinking all chart neighborhoods $U_k$, we can assume that $\varphi_{ji}$ is smooth on an open neighborhood of the closure of $V$ inside $V_{ji}$. Pick a non-negative continuous function $\rho : V_{ji} \rightarrow \R_{\geq 0}$ such that $V$ is contained in $\rho^{-1}(\{0\})$ and $\rho^{-1}(\{0\})$ is contained in the smooth locus of $\varphi_{ji}$. Moreover, assume that $\rho$ decays to $0$ towards the boundary of $V_{ji}$. By Proposition \ref{prop:smooth_approximation_area_preserving_homeos}, we can find an area-preserving diffeomorphism $\varphi_{ji}':V_{ji}\rightarrow V_{ij}$ such that $|\varphi_{ji}'(p)-\varphi_{ji}(p)|\leq \rho(p)$ for all $p\in V_{ji}$. Since $\rho$ decays to $0$ towards the boundary of $V_{ji}$, the homeomorphism $\varphi_{ji}^{-1} \circ  \varphi_{ji}'$ of $V_{ji}$ extends to an area-preserving homeomorphism $\chi$ of $V_i$ which agrees with the identity on the complement of $V_{ji}$. Now define the chart $\varphi_i'\coloneqq \chi^{-1}\circ \varphi_i$. The transition map between $\varphi_i'$ and $\varphi_j$ is simply given by $\varphi_{ji}'$ and is therefore smooth. Moreover, note that since $\varphi_{ji}$ and $\varphi_{ji}'$ agree on $V$, the homeomorphism $\chi$ restricts to the identity on $V_{ki}$ for all $k<j$. Thus the transition map $\varphi_{ki}$ does not change and remains smooth. Now, simply replace $\varphi_i$ by $\varphi_i'$. This concludes our construction of a smoothing of $\mathcal{L}$     
    \end{proof}
    
    We will now use the smooth structure on $\mathcal{L}$ to construct a smoothing of $(W,\Omega)$, i.e.\ we will construct a $C^0$ Hamiltonian atlas for $(W,\Omega)$ whose transition maps are smooth.

    Recall that we have a fixed-point-free flow $\varphi^t$ on $Y$ whose flow lines along with their natural orientations agree with the oriented leaves of the characteristic foliation of $\Omega$. Let $a>0$ be sufficiently small and pick finitely many topological embeddings $\iota_i: B(a) \hookrightarrow W$ such that the images of $\operatorname{pr}\circ \iota_i$ cover $\mathcal{L}$ and such that each $\operatorname{pr}\circ \iota_i$ is an area-preserving embedding which is smooth with respect to the smooth structure on $\mathcal{L}$ constructed in Claim \ref{cl:leaf-space-smoothing}. Given two embeddings $\iota_i$ and $\iota_j$, we define 
    \begin{equation*}
        V_{ji} \coloneqq (\operatorname{pr}\circ \iota_i)^{-1}(\operatorname{pr}\circ\iota_i(B(a))\cap \operatorname{pr}\circ\iota_j(B(a))) \subset B(a).
    \end{equation*}
    Note that, for every point $p \in V_{ji}$, there exists a unique transfer time $t_{ji}(p)\in \R$ such that $\varphi^{t_{ji}(p)}(\iota_i(p)) \in \operatorname{im}(\iota_j)$ and $\varphi^t(\iota_i(p)) \in W$ for all $t$ ranging from $0$ to $t_{ji}(p)$. The functions $t_{ji}: V_{ji}\rightarrow \R$ are continuous. We can perturb the embeddings $\iota_i$, without changing $\operatorname{pr} \circ \iota_i$,  such that all the $t_{ji}$ become smooth. This can be done via an elementary smoothing process similar to the one described in the proof of Claim \ref{cl:leaf-space-smoothing}. In place of Proposition \ref{prop:smooth_approximation_area_preserving_homeos}, the smoothing here relies on the simpler fact that real-valued continuous functions can be approximated by smooth ones. We omit the details of this.
    
    Now, we say an embedding $\iota: B(b)\hookrightarrow W$ to be smooth if $\operatorname{pr}\circ \iota$ is a smooth area-preserving embedding and if, for every embedding $\iota_i$, the locally defined transfer time between the image of $\iota$ and the image of $\iota_i$ is smooth. Note that by construction, if the transfer time between the image of $\iota$ and the image of $\iota_i$ is smooth near some point $p\in B(b)$, then the transfer time between the image of $\iota$ and the image of any other $\iota_j$ is smooth near $p$ as well. In particular, this implies that we can find a smooth embedding $\iota$ passing through any point in $W$.

    Finally, consider all $C^0$ Hamiltonian embeddings of the form
    \begin{equation*}
        \psi : ((-\varepsilon,\varepsilon)\times B(b),\omega_{\operatorname{std}}) \hookrightarrow (W,\Omega) \qquad \psi(t,p) = \varphi^t(\iota(p))
    \end{equation*}
    for some $\varepsilon>0$ and some smooth embedding $\iota : B(b)\hookrightarrow W$. We can cover $W$ by images of embeddings of this form. The inverses of these embeddings yield the desired collection of $C^0$ Hamiltonian charts with smooth transition maps.  
    
    We have now constructed a smoothing of $(W, \Omega)$; however, it does not in general extend to a global smoothing of $(Y,\Omega)$. The problem is the following: Let $S_- \subset \alpha((0,t_0)\times \Sigma)$ and $S_+ \subset \alpha((t_1,1)\times \Sigma)$ be two surfaces which are transverse to the characteristic foliation and smoothly embedded with respect to the smoothing of $(W,\Omega)$ constructed above. Moreover, suppose that each leaf segment of the form $\alpha((0,1)\times \{p\})$ for $p\in \Sigma$ intersects each of the surfaces $S_\pm$ exactly once. Traversing $\alpha((0,1)\times \Sigma)$ along the characteristic foliation then yields a homeomorphism $S_- \rightarrow S_+$. This homeomorphism is not going to be smooth in general and this is exactly the obstruction to extending the smoothing of $(W,\Omega)$ to all of $Y$. We may, however, modify $\Omega$ by inserting an $\Omega$-plug of the form $\mathcal{Q} = (\Sigma, \omega_\Sigma, \alpha, (\psi^t)_{t\in [0,1]})$ such that the resulting map $S_-\rightarrow S_+$ becomes smooth. We may then extend our smoothing of $(W, \Omega)$ to a global smoothing of the $C^0$ Hamiltonian structure $\Omega \# \mathcal{Q}$ on $Y$.

    Now, simply set $\omega  = \Omega\#\mathcal{Q} $ and $\mathcal{P}\coloneqq \overline{\mathcal{Q}}$, where $\overline{\mathcal{Q}}$ denotes the inverse plug of $\mathcal{Q}$. Then, $\omega\#\mathcal{P} = \Omega\#\mathcal{Q}\#\mathcal{P} = \Omega$, as desired.
    
    To prove the strengthened statement in Remark \ref{rem:from_C0_to_smooth_plus_plug_stronger}, the image of $\mathcal{P}$ is contained in the given exhaustive open set $U$ because the embedding $\alpha$ was chosen to have image contained in $U$.
\end{proof}

\subsection{Smoothing \texorpdfstring{$C^0$}{C0} Hamiltonian homeomorphisms}
\label{sec:smnoothing homeos}

We prove in this section that homeomorphisms of $C^0$ Hamiltonian structures on  closed $3$-manifolds may be smoothed, up to insertion of plugs.

\begin{thm}
\label{thm:smoothing_homeomorphisms}
    For $i \in \{1,2\}$, let $Y_i$ be a closed $3$-manifold and let $\omega_i$ be a smooth Hamiltonian structure on $Y_i$. Let $\mathcal{P}_i$ be a $C^0$ $\omega_i$-plug and assume that
    \begin{equation*}
        \psi : (Y_1,\omega_1\# \mathcal{P}_1) \rightarrow (Y_2,\omega_2\#\mathcal{P}_2)
    \end{equation*}
    is a homeomorphism of $C^0$ Hamiltonian structures. Then, there exist smooth $\omega_i$-plugs $\mathcal{Q}_i$ for $i\in \{1,2\}$ and a diffeomorphism of Hamiltonian structures
    \begin{equation*}
        \varphi: (Y_1, \omega_1 \# \mathcal{Q}_1) \rightarrow (Y_2, \omega_2  \# \mathcal{Q}_2)
    \end{equation*}
    with the following properties:
    \begin{enumerate}
        \item $\operatorname{Cal}(\mathcal{Q}_1) - \overline{\operatorname{Cal}}(\mathcal{P}_1) = \operatorname{Cal}(\mathcal{Q}_2) - \overline{\operatorname{Cal}}(\mathcal{P}_2)$

        \item $\varphi$ is homotopic to $\psi$ through homeomorphisms.
    \end{enumerate}
\end{thm}

\begin{rem}
    As an immediate corollary of Theorem \ref{thm:smoothing_homeomorphisms}, one obtains that two exact smooth Hamiltonian structures $\omega_1$ and $\omega_2$ on $Y$ which are homeomorphic have the same helicity. Indeed, we can apply Theorem \ref{thm:smoothing_homeomorphisms} with empty plugs $\mathcal{P}_i$. Then the resulting smooth $\omega_i$-plugs $\mathcal{Q}_i$ have the same (smooth) Calabi invariant, i.e.\ $\operatorname{Cal}(\mathcal{Q}_1) = \operatorname{Cal}(\mathcal{Q}_2)$. Since the smooth Hamiltonian structures $\omega_1\#\mathcal{Q}_1$ and $\omega_2\#\mathcal{Q}_2$ are diffeomorphic, they have the same helicity. Therefore, we can conclude that $\mathcal{H}(\omega_1) = \mathcal{H}(\omega_2)$ by Lemma \ref{lem:helicity_after_plug_insertion}.
\end{rem}

The proof of Theorem \ref{thm:smoothing_homeomorphisms}, which takes up the rest of this section, requires some preparatory lemmas.

Let $(\Sigma, \omega)$ be a surface equipped with an area-form $\omega$. Let $C_\Sigma$ denote the cylinder $(0,3) \times \Sigma$; we equip it with the smooth Hamiltonian structure induced by $\omega$. We also introduce the notation
\begin{equation*}
    C^-_\Sigma \coloneqq (0,1)\times \Sigma, \qquad C^0_\Sigma \coloneqq (1,2)\times \Sigma, \qquad C^+_\Sigma \coloneqq (2,3)\times \Sigma.
\end{equation*}

Figure~\ref{fig:C_Sigma} illustrates \( C_\Sigma \) and may aid the reader during the proof of Lemma~\ref{lem:covering_by_flowboxes}. In that proof, it is helpful to visualize \( C^-_\Sigma \) and \( C^+_\Sigma \) as being thin relative to \( C^0_\Sigma \). This can be made precise via a reparametrization of the interval \( (0,3) \), but we omit the details to keep the notation light and the argument clear.

\begin{figure}[h!]
    \centering
    \def\svgwidth{.2\textwidth}
\begingroup%
  \makeatletter%
  \providecommand\color[2][]{%
    \errmessage{(Inkscape) Color is used for the text in Inkscape, but the package 'color.sty' is not loaded}%
    \renewcommand\color[2][]{}%
  }%
  \providecommand\transparent[1]{%
    \errmessage{(Inkscape) Transparency is used (non-zero) for the text in Inkscape, but the package 'transparent.sty' is not loaded}%
    \renewcommand\transparent[1]{}%
  }%
  \providecommand\rotatebox[2]{#2}%
  \newcommand*\fsize{\dimexpr\f@size pt\relax}%
  \newcommand*\lineheight[1]{\fontsize{\fsize}{#1\fsize}\selectfont}%
  \ifx\svgwidth\undefined%
    \setlength{\unitlength}{389.96469092bp}%
    \ifx\svgscale\undefined%
      \relax%
    \else%
      \setlength{\unitlength}{\unitlength * \real{\svgscale}}%
    \fi%
  \else%
    \setlength{\unitlength}{\svgwidth}%
  \fi%
  \global\let\svgwidth\undefined%
  \global\let\svgscale\undefined%
  \makeatother%
  \begin{picture}(1,1.66165734)%
    \lineheight{1}%
    \setlength\tabcolsep{0pt}%
    \put(0,0){\includegraphics[width=\unitlength,page=1]{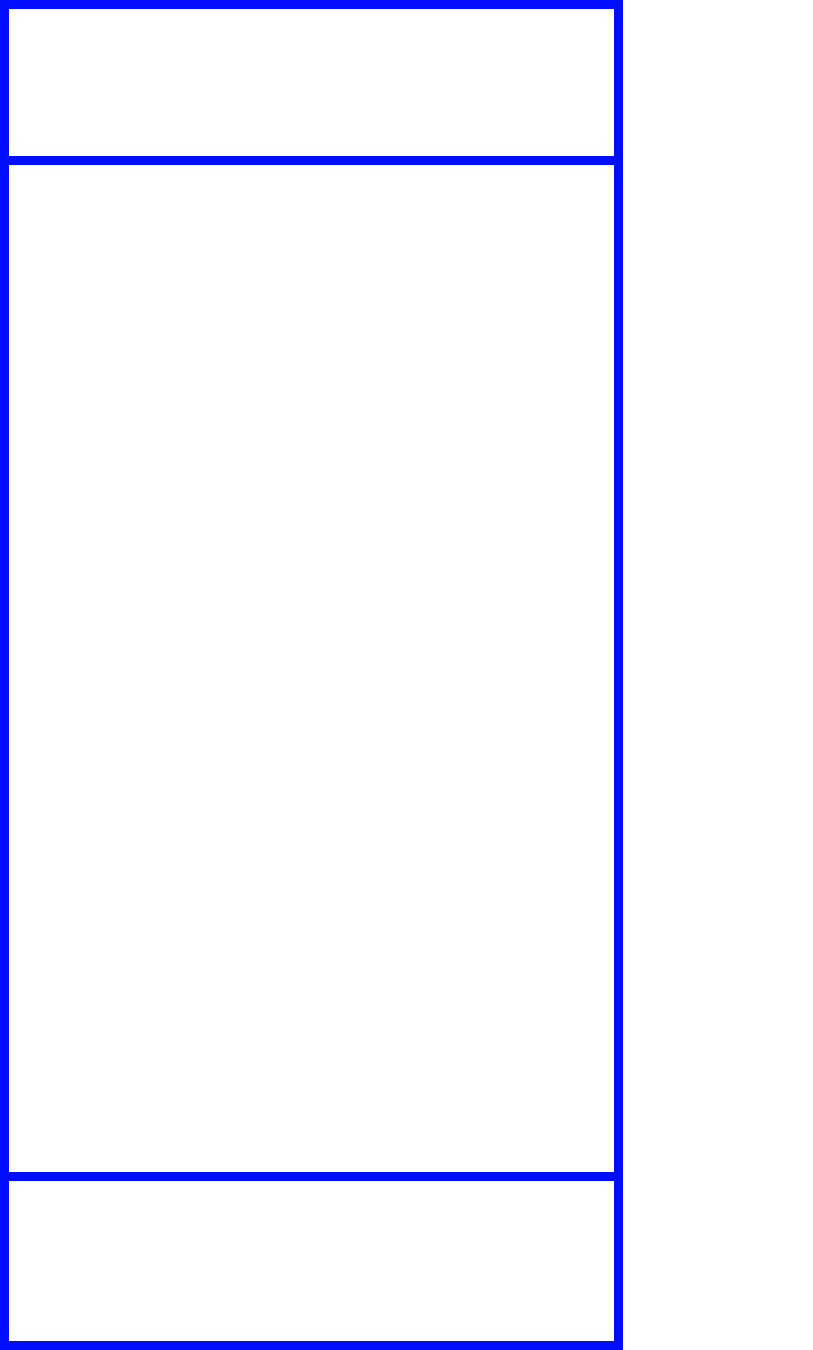}}%
    \put(0.1064068,0.08866378){\color[rgb]{0,0,0}\transparent{0.99000001}\makebox(0,0)[lt]{\lineheight{1.35000002}\smash{\begin{tabular}[t]{l}$C_{\Sigma}^-$\end{tabular}}}}%
    \put(0.09375128,1.53427005){\color[rgb]{0,0,0}\transparent{0.99000001}\makebox(0,0)[lt]{\lineheight{1.35000002}\smash{\begin{tabular}[t]{l}$C_{\Sigma}^+$\end{tabular}}}}%
    \put(0.10216084,0.81767808){\color[rgb]{0,0,0}\transparent{0.99000001}\makebox(0,0)[lt]{\lineheight{1.35000002}\smash{\begin{tabular}[t]{l}$C_{\Sigma}^0$\end{tabular}}}}%
    \put(0,0){\includegraphics[width=\unitlength,page=2]{C_Sigma.pdf}}%
  \end{picture}%
\endgroup%

    \caption{A depiction of $C_\Sigma$. The dotted vertical arrow indicates the direction of the characteristic foliation in $C_\Sigma$.  It is helpful to imagine $C^-_\Sigma , C^+_\Sigma$ as thin relative to $C^0_\Sigma$. }
    \label{fig:C_Sigma}
\end{figure}

\begin{lem}
\label{lem:covering_by_flowboxes}
    Let $Y$ be a closed topological $3$-manifold equipped with a $C^0$ Hamiltonian structure $\Omega$. Let $U,V\subset Y$ be two open sets such that $U\cup V = Y$. Assume that $U$ is exhaustive in the sense of Definition \ref{defn:exhaustive-set}.  
    
    Then there exist open surfaces with area forms $(\Sigma_1,\omega_1), \ldots, (\Sigma_n,\omega_n)$, relatively compact open subsets $S_i \Subset \Sigma_i$, and  topological embeddings of $C^0$ Hamiltonian structures
    \begin{equation*}
        \iota_i : (C_{\Sigma_i}, \omega_i) \hookrightarrow (V,\Omega)
    \end{equation*}
    satisfying the following properties:
    \begin{enumerate}
        \item \label{item:covering_by_flowboxes_covering} $Y\setminus U \subset \bigcup_i \iota_i(C_{S_i}^0)$
        \item \label{item:covering_by_flowboxes_collars} $\overline{\iota_i(C_{\Sigma_i}^\pm)} \subset U\cap V$ for all $i$.
        \item \label{item:covering_by_flowboxes_intersection_with_collars} $\iota_j(C_{\Sigma_j}^\pm) \cap \iota_i(C_{\Sigma_i}) = \emptyset$ for all $j<i$.
    \end{enumerate}
    Moreover, if the restriction of $\Omega$ to $V$ is smooth, we can take all $\iota_i$ to be smooth as well.
\end{lem}

\begin{proof}
    Let $p \in Y \setminus U$ be an arbitrary point. Let $L$ denote the characteristic leaf passing through $p$. By the exhaustiveness assumption on $U$, there exists a topologically embedded closed interval $I\subset L \cap V$ such that $p$ is contained in the interior of $I$ and the endpoints of $I$ are contained in $U\cap V$. Using sufficiently small local transverse sections of the characteristic foliation at the endpoints of $I$, one can construct an open surface with area form $(\Sigma,\omega)$ and a topological embedding $\iota: (C_\Sigma,\omega)\hookrightarrow (V,\Omega)$ containing $p$ in its image such that $\overline{\iota(C_\Sigma^\pm)} \subset U\cap V$. If the restriction of $\Omega$ to $V$ is smooth, we can construct $\iota$ to be smooth as well.
    
    Since the set $Y\setminus U$ is compact, we can cover it with the images of finitely many such cylinder embeddings. We may therefore pick $n\geq 0$ and tuples $(\Sigma_i,S_i,\iota_i)$ for $1\leq i \leq n$ such that
    \begin{equation*}
        Y\setminus U \subset \bigcup_i \iota_i(C_{S_i}^0) \qquad \text{and} \qquad \overline{\iota_i(C_{\Sigma_i}^\pm)} \subset U\cap V.
    \end{equation*}
    It remains to explain how to achieve property \ref{item:covering_by_flowboxes_intersection_with_collars} in the statement of the lemma. First, note that there is enough flexibility in the construction of $(\Sigma_i,S_i,\iota_i)$ to make sure that the intersection between $\overline{\iota_i(C_{\Sigma_i}^*)}$ and $\overline{\iota_j(C_{\Sigma_j}^\bullet)}$ for $1\leq i,j\leq n$ and $*, \bullet \in \{+,-\}$ is non-empty if and only if $i=j$ and $* = \bullet$.
    
    Moreover, we can arrange $(\Sigma_i,S_i,\iota_i)$ to have the property that, for all $j<i$ and $*\in \{+,-\}$ such that the intersection $\iota_j(C_{\Sigma_j}^*)\cap \iota_i(C_{\Sigma_i})$ is non-empty, there exist an open subset $E \subset \Sigma_i$ and an open interval $J \Subset (1,2)$ such that
    \begin{equation*}
        \iota_i^{-1}(\iota_j(C_{\Sigma_j}^*)) = J \times E.
    \end{equation*}
    In order to resolve, and remove,  the intersection between $\iota_j(C_{\Sigma_j}^*)$ and $\iota_i(C_{\Sigma_i})$, we replace $(\Sigma_i,S_i,\iota_i)$ by three tuples $(\Sigma_i^k,S_i^k,\iota_i^k)$ with $1\leq k\leq 3$ as described in detail below. At the end of this process, the cylinder $\iota_i(C_{\Sigma_i})$ will be replaced by three cylinders $\iota_i^1(C_{\Sigma_i^1}), \iota_i^2(C_{\Sigma_i^2}), \iota_i^3(C_{\Sigma_i^3})$, which will be constructed in such a way that they do not intersect    $\iota_j(C_{\Sigma_j}^*)$. This, of course, requires removing some portions of $\iota_i(C_{\Sigma_i})$, and also $\iota_i(C_{S_i}^0)$. The removed portions will all be contained in $U$ and hence property \ref{item:covering_by_flowboxes_covering} will continue to hold. We will also ensure that property \ref{item:covering_by_flowboxes_collars} continues to hold. However, this is less complicated. 
    
    \begin{figure}[h!]
        \centering
        \def\svgwidth{.9\textwidth}
        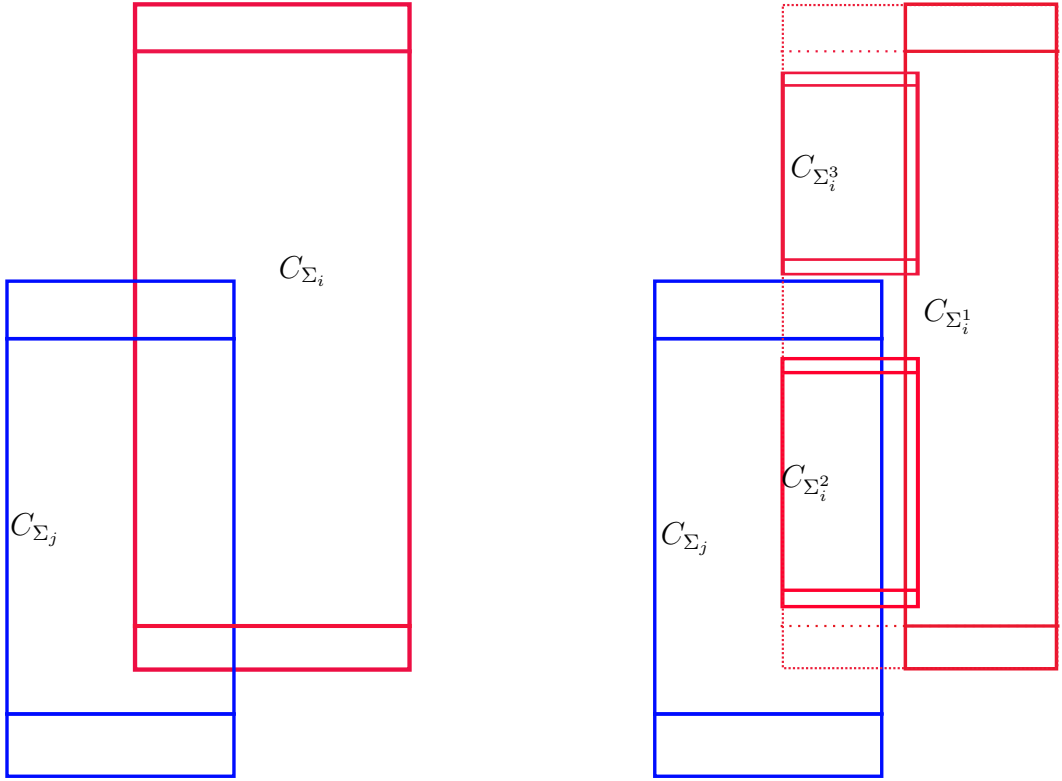
        \caption{Removing intersections: On the left, $C_{\Sigma_j}^+$ intersects $C_{\Sigma_i}$, with $i>j$. On the right, $C_{\Sigma_i}$ is replaced with with $C_{\Sigma_i^1}, C_{\Sigma_i^2},C_{\Sigma_i^3}$.}
        \label{fig:resolving-intersections}
    \end{figure}

    We now proceed with a detailed description of the above process, which is also depicted in Figure \ref{fig:resolving-intersections}. Pick open sets
    \begin{equation*}
        S_i^1\Subset \Sigma_i^1\Subset \Sigma_i \quad \text{and}\quad S_i^2 = S_i^3\Subset \Sigma_i^2 = \Sigma_i^3\Subset \Sigma_i
    \end{equation*}
    such that
    \begin{equation*}
        S_i \subset S_i^1\cup S_i^2, \quad  E\cap \Sigma_i^1 = \emptyset \quad \text{and}\quad \overline{\iota_i(J\times \Sigma_i^2)} \subset U\cap V.
    \end{equation*}
    Moreover, we can arrange $\Sigma_i^2=\Sigma_i^3$ to be contained in an arbitrarily small neighborhood of the closure of $E$ inside $\Sigma_i$. Note that $(1,2)\setminus J$ has two components $(1,t_-]$ and $[t_+,2)$. We pick orientation-preserving embeddings
    \begin{equation*}
        \tau_2 : [0,3]\rightarrow (1,t_-) \quad \text{and} \quad \tau_3 : [0,3] \rightarrow (t_+,2)
    \end{equation*}
    such that
    \begin{align*}
        \tau_2([0,1]) &\subset (1,1+\varepsilon) & \tau_2([2,3]) 
        &\subset (t_--\varepsilon,t_-) \\
        \tau_3([0,1]) &\subset (t_+,t_++\varepsilon) & \tau_3([2,3]) &\subset (2-\varepsilon,2)
    \end{align*}
    for some small $\varepsilon>0$. We define
    \begin{align*}
        \iota_i^1 : C_{\Sigma_i^1} &\rightarrow V & \iota_i^1(t,p) &\coloneqq \iota_i(t,p)\\
        \iota_i^2 : C_{\Sigma_i^2} &\rightarrow V & \iota_i^2(t,p) &\coloneqq \iota(\tau_2(t),p) \\
        \iota_i^3 : C_{\Sigma_i^3} &\rightarrow V & \iota_i^3(t,p) &\coloneqq \iota_i(\tau_3(t),p).
    \end{align*}
    If $\varepsilon$ is chosen sufficiently small and $\Sigma_i^2 = \Sigma_i^3$ is contained in a sufficiently small neighborhood of the closure of $E$ in $\Sigma_i$, then the six sets $\overline{\iota_i^k(C_{\Sigma_i^k}^\pm)}$ for $1\leq k \leq 3$ are contained in $U\cap V$ and are disjoint from $\overline{\iota_\ell(C_{\Sigma_\ell}^\pm)}$ for any $\ell \neq i$. They are also pairwise disjoint. Moreover, $\iota_i(C_{S_i}^0)\setminus U$ is contained in $\bigcup_k\iota_i^k(C_{S_i^k}^0)$, which implies that
    \begin{equation*}
        Y\setminus U \subset \bigcup_{1\leq k \leq 3} \iota_i^k(C_{S_i^i}^0) \cup \bigcup_{\ell \neq i} \iota_\ell(C_{S_\ell}^0).
    \end{equation*}
    The three sets $\iota_i^k(C_{\Sigma_i^k})$ for $1\leq k \leq 3$ are disjoint from $\iota_j(C_{\Sigma_j}^*)$. In addition, $\iota_i^k(C_{\Sigma_i^k}^\pm)$ is disjoint from $\iota_i^\ell(C_{\Sigma_i^\ell})$ for $1\leq k < \ell \leq 3$. We discard the tuple $(S_i,\Sigma_i,\iota_i)$ and instead insert the three tuples $(S_i^1,\Sigma_i^1,\iota_i^1), (S_i^2,\Sigma_i^2,\iota_i^2), (S_i^3,\Sigma_i^3,\iota_i^3)$ between $(\Sigma_{i-1},S_{i-1}, \iota_{i-1})$ and $(\Sigma_{i+1},S_{i+1},\iota_{i+1})$. Then we adjust the indexing of our new list of tuples. As observed above, properties \ref{item:covering_by_flowboxes_covering} and \ref{item:covering_by_flowboxes_collars} in Lemma \ref{lem:covering_by_flowboxes} are still satisfied. Applying the above construction finitely many times, we can get rid of all non-empty intersections of the form $\iota_j(C_{\Sigma_j}^\pm) \cap \iota_i(C_{\Sigma_i})$ for $j<i$ and therefore also achieve property \ref{item:covering_by_flowboxes_intersection_with_collars}.
\end{proof}

\begin{definition}
    Let $U, V$ be open subset of $\R^3$.  Suppose that $\psi: (U, \omega_{\operatorname{std}}) \rightarrow (V, \omega_{\operatorname{std}})$ is a homeomorphism of Hamiltonian structures.  We say that $\psi$ is {\it transversely smooth} if, for every point $p\in U$, there exist an open neighborhood of $p$ of the form $I\times B$, where $I\subset \R$ is an open interval and $B\subset \R^2$ is a ball, and a {\it  smooth} area-preserving embedding $a:B\hookrightarrow \R^2$ such that
    \begin{equation*}
      \psi(t,z) = (*, a(z)) \qquad \text{for all $(t,z)\in I\times B$.}
    \end{equation*}

    Let $\psi: (Y, \omega_0) \rightarrow (Y, \omega_1)$ be a homeomorphism of Hamiltonian structures, where $\omega_0, \omega_1$ are smooth.  We say that $\psi$ is \textit{transversely smooth} if it is transversely smooth in smooth local charts.  
\end{definition}

\begin{lem}
\label{lem:transversely_smooth_topological_conjugacy_implies_diffeomorphism}
   
For \( i \in \{1,2\} \), let \( (Y_i, \omega_i) \) be a closed \( 3 \)-manifold with a smooth Hamiltonian structure.  
Suppose that \( \varphi_0: (Y_1, \omega_1) \to (Y_2, \omega_2) \) is a transversely smooth homeomorphism of Hamiltonian structures.  Then, there exists a diffeomorphism of Hamiltonian structures $\varphi_1: (Y_1,\omega_1) \rightarrow (Y_2,\omega_2)$ which is isotopic to $\varphi_0$ through homeomorphisms of Hamiltonian structures $\varphi_t: (Y_1,\omega_1) \rightarrow (Y_2,\omega_2)$.  Moreover, the homeomorphisms $\varphi_t$ all induce the same homeomorphism between the leaf spaces of $(Y_1, \omega_1)$ and $(Y_2, \omega_2)$.
\end{lem}

\begin{proof}
    Using smooth local charts on $(Y_1, \omega_1)$ and $(Y_2, \omega_2)$, one can reduce the proof of the lemma to the following claim.
    
    \begin{claim}
        Consider open intervals $I\Subset J\subset \R$ and open balls $B\Subset C \subset \R^2$. Let $U\subset J\times C$ be an open subset. Let $\operatorname{pr}_2 : \R^3= \R\times \R^2 \rightarrow \R^2$ denote the projection onto the second factor. Consider an embedding $\psi_0: J\times C \hookrightarrow \R^3$ with the property that $\operatorname{pr}_2 \circ \psi_0 = \operatorname{pr}_2|_{J\times C}$. Assume that the restriction of $\psi_0$ to $U$ is a smooth embedding.
        
        Then there exists a family of embeddings $\psi_t: J\times C \hookrightarrow \R^3$, for $t \in [0,1]$, satisfying $\operatorname{pr}_2\circ \psi_t = \operatorname{pr}_2$ and agreeing with $\psi_0$ outside some compact subset of $J\times C$, such that the restriction of $\psi_1$ to $U\cup (I\times B)$ is a smooth embedding. 
    \end{claim}

    In order to see that this claim implies the lemma, note that since $\varphi_0$ is transversely smooth, in suitable smooth local Darboux coordinates it precisely takes the form of an embedding $\psi_0$ as in the claim. Therefore, the assertion of the claim allows us to go through a finite covering of $(Y_1,\omega_1)$ by Darboux charts and step by step make $\varphi_0$ smooth in each of the charts. 
    
    We briefly outline the proof of the claim.  The embedding $\psi_0: J\times C\hookrightarrow \R^3$ is of the form
    \begin{equation*}
        \psi_0(s,z) = (f_0(s,z), z)
    \end{equation*}
     where $(s,z)$ denotes a point in $\R \times \R^2$. For each fixed $z \in C$, the function $f_0(\cdot,z)$ is strictly increasing. The restriction of $f_0$ to $U$ is smooth and the partial derivative $\partial_sf_0$ is strictly positive in this region. We can find a function $f_1:J\times C\rightarrow \R$ agreeing with $f_0$ outside some compact subset of $J\times C$ such that $f_1(\cdot,z)$ is strictly increasing for every $z$ and such that the restriction of $f_1$ to $U \cup (I\times B)$ is smooth and has strictly positive partial derivative $\partial_sf_1$ in this region. Moreover, we can find a path of functions $f_t(s,z)$ connecting $f_0$ to $f_1$ such that each $f_t$ agrees with $f_0$ outside a compact subset of $J\times C$ and such that $f_t(\cdot,z)$ is strictly increasing for all $z$. We can then define \begin{equation*}
        \psi_t(s,z) = (f_t(s,z), z),
    \end{equation*}
     which satisfies the requirements of the claim.
\end{proof}

\begin{rem}
\label{rem:making_plug_embedding_smooth}
    Let $(Y,\omega)$ be a $3$-manifold equipped with a smooth Hamiltonian structure and let $\mathcal{P} = (\Sigma,\omega_\Sigma,\alpha, (\varphi^t)_{t\in [0,1]})$ be a $C^0$ $\omega$-plug. Using Proposition \ref{prop:smooth_approximation_area_preserving_homeos}, it is possible to find an area-preserving homeomorphism $\gamma$ of $(\Sigma,\omega_\Sigma)$, not necessarily compactly supported, such that $\alpha \circ (\operatorname{id}\times \gamma) : ((0,1)\times \Sigma,\omega_\Sigma)\hookrightarrow (Y,\omega)$ is a transversely smooth embedding of Hamiltonian structures. By an argument similar to the proof of Lemma \ref{lem:transversely_smooth_topological_conjugacy_implies_diffeomorphism}, it is possible to modify the embedding $\alpha\circ (\operatorname{id}\times \gamma)$ by slightly sliding it along the leaves of the characteristic foliation of $(Y,\omega)$ to obtain a smooth embedding of Hamiltonian structures $\alpha' : ((0,1)\times \Sigma,\omega_\Sigma)\hookrightarrow (Y,\omega)$. Define the $\omega$-plug
    \begin{equation*}
        \mathcal{P'} \coloneqq (\Sigma,\omega_\Sigma,\alpha', (\gamma^{-1}\circ \varphi^t\circ \gamma)_{t\in [0,1]}).
    \end{equation*}
    One can check, using Lemma \ref{lem:C0_plugs_with_isotopic_embeddings}, that $\omega\#\mathcal{P}'$ is homeomorphic to $\omega\#\mathcal{P}$ via a homeomorphism isotopic to the identity. Moreover, it follows from the naturality of the $\mathcal{R}$-valued extension of the Calabi homomorphism that $\overline{\operatorname{Cal}}(\mathcal{P}') = \overline{\operatorname{Cal}}(\mathcal{P})$. The upshot of this remark is that we can replace any $C^0$ $\omega$-plug $\mathcal{P}$ by an $\omega$-plug whose underlying embedding $\alpha : ((0,1)\times \Sigma,\omega_\Sigma)\hookrightarrow (Y,\omega)$ is smooth.
\end{rem}

\begin{proof}[Proof of Theorem \ref{thm:smoothing_homeomorphisms}]
    First, let us consider the special case where $\mathcal{P}_2$ is a trivial plug. In this simplified situation, we will construct a smooth plug $\mathcal{Q}_1$ satisfying $\operatorname{Cal}(\mathcal{Q}_1) = \overline{\operatorname{Cal}}(\mathcal{P}_1)$ and a diffeomorphism of Hamiltonian structures
    \begin{equation*}
        \varphi: (Y_1,\omega_1\#\mathcal{Q}_1) \rightarrow (Y_2,\omega_2)
    \end{equation*}
    which is isotopic to $\psi$ through homeomorphisms.  Note that this in particular implies that $\overline{\operatorname{Cal}}(\mathcal{P}_1) \in \R \subset \mathcal{R}$.

    Let $U\subset Y_1$ denote the image of the plug $\mathcal{P}_1$. After possibly adding trivial components to the plug $\mathcal{P}_1$, we can assume that $U$ is exhaustive in the sense of Definition \ref{defn:exhaustive-set}. This is possible by Lemma \ref{lem:exhaustive-surface}. Note that the addition of trivial components to $\mathcal{P}_1$ does not change the Calabi invariant of the plug. Pick an open neighborhood $V$ of $Y_1\setminus U$ such that $\omega_1\#\mathcal{P}_1$ agrees with $\omega_1$ on $V$ and is thus smooth in this region. We apply Lemma \ref{lem:covering_by_flowboxes}. This yields smooth open surfaces $(\Sigma_1,\sigma_1), \ldots, (\Sigma_n, \sigma_n)$, relatively compact open subsets $S_i\Subset \Sigma_i$, and smooth embeddings of Hamiltonian structures $\iota_i : (C_{\Sigma_i},\sigma_i) \hookrightarrow (V,\omega_1\# \mathcal{P}_1)$ satisfying all the properties listed in Lemma \ref{lem:covering_by_flowboxes}.

    In what follows, we will construct $C^0$ $\omega_1\#\mathcal{P}_1$-plugs $\mathcal{Z}_1^-,  \mathcal{Z}_1^+, \ldots, \mathcal{Z}_n^-,  \mathcal{Z}_n^+$ which are of the form 
    \begin{align*}
        &\mathcal{Z}_i^- = (\Sigma_i, \sigma_i, \iota_i|_{C_{\Sigma_i}^-}, \;\varphi_i^t \,) \\
        &\mathcal{Z}_i^+ = (\Sigma_i, \sigma_i, \iota_i|_{C_{\Sigma_i}^+}, (\varphi_i^t)^{-1})
    \end{align*}
    where $\varphi_i^t$ is an isotopy in $\overline{\operatorname{Ham}}(\Sigma_i)$ which will be chosen below. By the properties listed in Lemma \ref{lem:covering_by_flowboxes}, the $2n$ plugs $\mathcal{Z}_i^\pm$ are pairwise disjoint and contained in $U\cap V$. Let $\Omega_i$ denote the $C^0$ Hamiltonian structure obtained by inserting the plugs $\mathcal{Z}_j^\pm$ for $1\leq j \leq i$ into $\omega_1\#\mathcal{P}_1$. We will show that for appropriately chosen $\varphi_i^t$ there exists a homeomorphism of $C^0$ Hamiltonian structures  $$\psi_i : (Y_1,\Omega_i) \rightarrow (Y_2,\omega_2)$$  whose restriction  to $\bigcup_{j\leq i} \iota_j(C_{S_j}^0) \setminus \overline{U}$ is transversely smooth. Note that the assertion on transverse smoothness of $\psi_i$ makes sense because the restriction of $\Omega_i$ to the complement of $\overline{U}$ is a smooth Hamiltonian structure.
    
    We will now describe an inductive procedure for constructing $\varphi_i^t$ and $\psi_i$ as above.  Let $1\leq i \leq n$ and suppose that $\varphi_k$ and $\psi_k$ have already been constructed for all $k<i$. We explain how to construct $\varphi_i$ and $\psi_i$. Consider the composition
    \begin{equation*}
        \psi_{i-1} \circ \iota_i : (C_{\Sigma_i},\sigma_i) \hookrightarrow (Y_2,\omega_2),
    \end{equation*}
    which is a topological embedding of $C^0$ Hamiltonian structures. Note that the leaf space of $(\operatorname{im}(\psi_{i-1} \circ \iota_i),\omega_2)$ has the structure of a smooth $2$-dimensional manifold with area form. We abbreviate it by $(\mathcal{L},\lambda)$. The embedding $\psi_{i-1} \circ \iota_i$ induces an area-preserving homeomorphism $\gamma : (\Sigma_i,\sigma_i) \rightarrow (\mathcal{L},\lambda)$. Let $\operatorname{pr}_2 : C_{\Sigma_i} = (0,3)\times \Sigma_i \rightarrow \Sigma_i$ denote the projection onto the second factor. Note that the restriction of $\gamma$ to the set
    \begin{equation*}
        E\coloneqq      \operatorname{pr}_2(\iota_i^{-1}(\bigcup_{j<i}\iota_j(C_{S_j}^0))\setminus \overline{U})
    \end{equation*}
    is smooth by the transverse smoothness assumption on $\psi_k$ for $k<i$.
    
    Using Proposition \ref{prop:smooth_approximation_area_preserving_homeos}, one can find an area-preserving homeomorphism $\tilde{\gamma} : (\Sigma_i,\sigma_i) \rightarrow (\mathcal{L},\lambda)$ which agrees with $\gamma$ outside some compact set and whose restriction to $S_i\cup E$ is smooth. Indeed, pick a continuous function $\rho:S_i \rightarrow \R_{\geq 0}$ which decays to zero towards the boundary of $S_i$, whose zero set $\rho^{-1}(\{0\})$ is contained in $E$, and which has the property that there exists an open neighborhood $F$ of $\partial S_i \cap E$ in $E$ such that $F\cap S_i \subset \rho^{-1}(\{0\})$. Now apply Proposition \ref{prop:smooth_approximation_area_preserving_homeos} to $\gamma|_{S_i}$ with this function $\rho$. This yields an area-preserving diffeomorphism $\tilde \gamma : S_i \rightarrow \gamma(S_i)$. Since $\rho$ decays to zero near the boundary of $S_i$, we can extend $\tilde \gamma$ to an area-preserving homeomorphism $\tilde\gamma : \Sigma_i \rightarrow \mathcal{L}$ by setting it equal to $\gamma$ outside of $S_i$. Since $F\cap S_i \subset \rho^{-1}(\{0\})$, we see that $\tilde{\gamma}$ agrees with $\gamma$ on $F$ and is therefore smooth in this region. We conclude that the restriction of $\tilde\gamma$ to $S_i \cup E$ is smooth.
    
    Moreover, we can pick $\tilde{\gamma}$ such that $\tilde{\gamma}^{-1}\circ \gamma$ is in $\overline{\operatorname{Ham}}(\Sigma_i)$. Indeed, note that since $\tilde{\gamma}$ is a $C^0$ approximation of $\gamma$, the homeomorphism $\tilde{\gamma}^{-1}\circ \gamma$ is contained in the identity component $\operatorname{Homeo}_0(\Sigma_i,\sigma_i)$. If $\tilde\gamma^{-1}\circ \gamma$ is not in $\overline{\operatorname{Ham}}(\Sigma_i)$, simply pick a compactly supported area-preserving diffeomorphism $\beta \in \operatorname{Diff}_0(\mathcal{L},\lambda)$ such that the mass flow homomorphism on $\operatorname{Homeo}_0(\mathcal{L},\lambda)$ takes the same value on $\beta$ and $\gamma\circ \tilde\gamma^{-1}$. Then replace $\tilde\gamma$ by $\beta\circ \tilde\gamma$.
    
    Since $\tilde\gamma^{-1}\circ \gamma \in \overline{\operatorname{Ham}}(\Sigma_i)$, we can find  an isotopy $\varphi_i^t \in \overline{\operatorname{Ham}}(\Sigma_i)$ whose time-1 map satisfies
    \begin{equation*}
        \varphi_i^1=\tilde{\gamma}^{-1}\circ \gamma.
    \end{equation*}
    
    The map $\psi_i$ is constructed by modifying $\psi_{i-1}$ inside $\iota_i(C_{\Sigma_i})$, hence it agrees with $\psi_{i-1}$ outside $\iota_i(C_{\Sigma_i})$. We describe the construction of $\psi_i$ on each of $\iota_i(C_{\Sigma_i}^-), \iota_i(C_{\Sigma_i}^0), \iota_i(C_{\Sigma_i}^+)$. Inside $\iota_i(C_{\Sigma_i}^0)$, we define $\psi_i$ such that
    \begin{equation*}
         \psi_i \circ \iota_i(t,z) = \psi_{i-1}\circ \iota_i (t,\gamma^{-1}\circ \tilde{\gamma}(z)) \qquad \text{for $(t,z) \in C_{\Sigma_i}^0$.}
    \end{equation*}
    Inside $\iota_i(C_{\Sigma_i}^-)$, we define  $\psi_i$ such that
    \begin{equation*}
         \psi_i \circ \iota_i(t,z) = \psi_{i-1}\circ \iota_i(t, (\varphi_i^t)^{-1}(z)) \qquad \text{for $(t,z) \in C_{\Sigma_i}^-$.}
    \end{equation*}
    Inside $\iota_i(C_{\Sigma_i}^+)$, we define  $\psi_i$ such that
    \begin{equation*}
         \psi_i \circ \iota_i(t,z) = \psi_{i-1}\circ \iota_i(t, \varphi_i^{t-2}(z)) \qquad \text{for $(t,z) \in C_{\Sigma_i}^+$.}
    \end{equation*}
    One can verify that the above yield a well-defined map $\psi_i$ on $\iota_i(C_{\Sigma_i})$ which coincides with $\psi_{i-1}$ near the boundary of $\iota_i(C_{\Sigma_i})$. The plugs $\mathcal{Z}_i^\pm$ were defined precisely such that $\psi_i$ is a homeomorphism of $C^0$ Hamiltonian structures $\psi_i: (Y_1,\Omega_i)\rightarrow (Y_2,\omega_2)$. Finally, note that the map $\Sigma_i \rightarrow \mathcal{L}$ induced by $\psi_i\circ \iota_i|_{C_{\Sigma_i}^0}$ is precisely given by $\tilde{\gamma}$. Since the restriction of $\tilde\gamma$ to $S_i \cup E$ is smooth, we see that the restriction of $\psi_i$ to $\bigcup_{j\leq i} \iota_j(C_{\Sigma_j}^0) \setminus \overline{U}$ is transversely smooth. Since $\psi_{i-1}$ is isotopic to $\psi$ by assumption, so is $\psi_i$. This concludes our construction of $\psi_i$ and the isotopy $\varphi_i^t$.

    To summarize, at this point we have a homeomorphism of $C^0$ Hamiltonian structures 
    \begin{equation*}
        \psi_n: (Y_1, \Omega_n) \rightarrow (Y_2, \omega_2)
    \end{equation*}
    which, by property \ref{item:covering_by_flowboxes_covering} in Lemma \ref{lem:covering_by_flowboxes}, is transversely smooth on $Y_1 \setminus \overline U$.  Moreover, $\Omega_n$ is obtained from $\omega_1\#\mathcal{P}_1$ via the insertion of the plugs $\mathcal{Z}_1^\pm, \ldots, \mathcal{Z}_n^\pm$.

    Since the plugs $\mathcal{Z}_1^\pm, \ldots, \mathcal{Z}_n^\pm$ are pairwise disjoint and all contained in $U\cap V$, there exist a single $\omega_1$-plug $\mathcal{Q}_1$, contained in $U$, and a homeomorphism $\chi : (Y_1,\omega_1\#\mathcal{Q}_1) \rightarrow (Y_1,\Omega_n)$ which is isotopic to the identity and compactly supported inside $U$. The Calabi invariant of $\mathcal{Q}_1$ is given by
    \begin{equation*}
        \overline{\operatorname{Cal}}(\mathcal{Q}_1) = \overline{\operatorname{Cal}}(\mathcal{P}_1) + \sum_i (\overline{\operatorname{Cal}}(\mathcal{Z}_i^-) + \overline{\operatorname{Cal}}(\mathcal{Z}_i^+)) = \overline{\operatorname{Cal}}(\mathcal{P}_1),
    \end{equation*}
    where we use naturality of the tautological extension of the Calabi homomorphism and the identity
    \begin{equation*}
        \overline{\operatorname{Cal}}(\mathcal{Z}_i^-) + \overline{\operatorname{Cal}}(\mathcal{Z}_i^+) = \overline{\operatorname{Cal}}(\varphi_i) + \overline{\operatorname{Cal}}(\varphi_i^{-1}) = 0.
    \end{equation*}
    Like $\psi_n$, the homeomorphism
    \begin{equation*}
        \psi_n \circ \chi : (Y_1,\omega_1\#\mathcal{Q}_1) \rightarrow (Y_2,\omega_2)
    \end{equation*}
    is transversely smooth on the complement of $\overline{U}$. Write $\mathcal{Q}_1 := (\Sigma, \sigma, \alpha, (\theta^t)_{t\in [0,1]})$ for the data of the plug $\mathcal{Q}_1$. Note that we can assume that the embedding $\alpha$ is smooth with respect to the smooth Hamiltonian structure $\omega_1$, see Remark \ref{rem:making_plug_embedding_smooth}.
    
    \begin{claim}
    \label{cl:plug-is-smooth}
        The map $\theta^1$ is smooth.
    \end{claim}
    
    \begin{proof}
        Pick $\delta>0$ sufficiently small such that $\theta^t$ is equal to $\operatorname{id}$ for $t \in [0,2\delta)$ and equal to $\theta^1$ for $t\in (1-2\delta,1]$. Note that $\omega_1 \#\mathcal{Q}_1$ agrees with $\omega_1$ on $\alpha((0,2\delta) \times \Sigma)$ and on $\alpha((1-2\delta)\times \Sigma)$ and that $\psi_n\circ \chi$ is transversely smooth on these sets.

        Set $S_-\coloneqq \alpha(\{\delta\}\times \Sigma)$ and $S_+ \coloneqq \alpha(\{1-\delta\}\times \Sigma)$. Traversing $\operatorname{im}(\alpha)$ along the characteristic foliation of $\omega_1\# \mathcal{Q}_1$ induces a homeomorphism
        \begin{equation*}
            f: S_- \rightarrow S_+.
        \end{equation*}
        Similarly, traversing $\operatorname{im}(\psi_n\circ\chi\circ\alpha)$ along the the characteristic foliation of $\omega_2$ induces a homeomorphism
        \begin{equation*}
            f': \psi_n\circ \chi(S_-)) \rightarrow \psi_n\circ \chi(S_+).
        \end{equation*}
        These homeomorphisms fit into the following commutative diagram:
        \begin{equation*}
            \begin{tikzcd}
                \psi_n\circ \chi(S_-) \arrow[r, "f'"] & \psi_n\circ \chi(S_+) \\
                S_- \arrow[r, "f"] \arrow[u, "\psi_n\circ \chi"] & S_+ \arrow[u, "\psi_n\circ \chi"] \\
                \{\delta\} \times \Sigma \arrow[u,"\alpha"] & \{1-\delta\}\times \Sigma \arrow[u,"\alpha"] \\
                \Sigma \arrow[u,"\cong"] \arrow[r,"\theta^1"] & \Sigma \arrow[u,"\cong"]
            \end{tikzcd}
        \end{equation*}
        Since $\psi_n\circ \chi$ is transversely smooth on a neighborhood of $S_-$ and $S_+$ and $\omega_2$ is a smooth Hamiltonian structure, we see that $\theta^1$ is smooth.    
    \end{proof}

    As a consequence of Claim \ref{cl:plug-is-smooth}, we may modify the plug $\mathcal{Q}_1$ by replacing $\theta^t$ with a smooth Hamiltonian isotopy with the same time-1 map is $\theta^1$. By Lemma \ref{lem:C0_plugs_with_same_time_1_maps}, modifying the plug in this way yields a smooth Hamiltonian structure homeomorphic to $\omega_1 \#\mathcal{Q}_1$ via a homeomorphism supported in $U$ and isotopic to the identity.  Moreover, the two plugs have the same Calabi invariant. We will refer to this new smooth plug by the same symbol $\mathcal{Q}_1$. We then have a homeomorphism of Hamiltonian structures $$\varphi:(Y_1,\omega_1\#\mathcal{Q}_1) \rightarrow (Y_2,\omega_2)$$
    which is isotopic to $\psi$ and transversely smooth outside of $\overline{U}$. Since both $\omega_1\# \mathcal{Q}_1$ and $\omega_2$ are smooth and no characteristic leaf of $\omega_1\#\mathcal{Q}_1$ is trapped in $U$, we can conclude that $\varphi$ is transversely smooth everywhere. By Lemma \ref{lem:transversely_smooth_topological_conjugacy_implies_diffeomorphism}, we can replace $\varphi$ by a diffeomorphism of Hamiltonian structures, still isotopic to $\psi$. This concludes the proof of Theorem \ref{thm:smoothing_homeomorphisms} in the special case that $\mathcal{P}_2$ is a trivial plug.

    It remains to treat the general case. Let $\overline{\mathcal{P}}_2$ be the inverse plug of $\mathcal{P}_2$ and let $\psi^*\overline{\mathcal{P}}_2$ be its pull back via $\psi$; see Equations \eqref{eqn:plug-Calabi-defn} and \eqref{eqn:pullback-plug}. Then, as noted in \eqref{eqn:pullback-plug-homeos}, $\psi$ is a homeomorphism of $C^0$ Hamiltonian structures
    \begin{equation*}
        \psi : (Y_1,\omega_1\#\mathcal{P}_1\#\psi^*\overline{\mathcal{P}}_2) \rightarrow (Y_2,\omega_2\#\mathcal{P}_2\#\overline{\mathcal{P}}_2) = (Y_2,\omega_2).
    \end{equation*}
    We also have
    \begin{equation*}
        \overline{\operatorname{Cal}}(\psi^*\overline{\mathcal{P}}_2) = \overline{\operatorname{Cal}}(\overline{\mathcal{P}}_2) = -\overline{\operatorname{Cal}}(\mathcal{P}_2).
    \end{equation*}
    
   As we explain below, we can slide the plug $\psi^*\overline{\mathcal{P}}_2$ along the characteristic foliation of $\omega_1 \# \mathcal{P}_1$ so that it becomes disjoint from $\mathcal{P}_1$, as described in Lemma \ref{lem:C0_plugs_with_isotopic_embeddings}.
   
    Let $\mathcal{S}$ denote the resulting plug disjoint from $\mathcal{P}_1$. It is both an $\omega_1$- and an $\omega_1\#\mathcal{P}_1$-plug and has Calabi invariant $\overline{\operatorname{Cal}}(\mathcal{S}) = - \overline{\operatorname{Cal}}(\mathcal{P}_2)$. Let $\mathcal{P}_1 \sqcup \mathcal{S}$ be the $\omega_1$-plug given by the disjoint union of $\mathcal{P}_1$ and $\mathcal{S}$. By Lemma \ref{lem:C0_plugs_with_isotopic_embeddings}, $(Y_1,\omega_1\#\mathcal{P}_1\#\psi^*\overline{\mathcal{P}}_2)$ is homeomorphic to $(Y_1,\omega_1\#(\mathcal{P}_1\sqcup \mathcal{S}))$ via a homeomorphism isotopic to the identity. Thus $(Y_1,\omega_1\#(\mathcal{P}_1\sqcup\mathcal{S}))$ is homeomorphic to $(Y_2,\omega_2)$ via a homeomorphism isotopic to $\psi$. By the special case of Theorem \ref{thm:smoothing_homeomorphisms} already treated above, we can find a smooth $\omega_1$-plug $\mathcal{Q}_1$ such that there exists a diffeomorphism of smooth Hamiltonian structures $\varphi: (Y_1,\omega_1\#\mathcal{Q}_1) \rightarrow (Y_2,\omega_2)$, still isotopic to $\psi$, and such that
    \begin{equation*}
        \operatorname{Cal}(\mathcal{Q}_1) = \overline{\operatorname{Cal}}(\mathcal{P}_1\sqcup \mathcal{S}) = \overline{\operatorname{Cal}}(\mathcal{P}_1) - \overline{\operatorname{Cal}}(\mathcal{P}_2).
    \end{equation*}
    Now simply define $\mathcal{Q}_2$ to be the empty plug. Then the tuple $(\mathcal{Q}_1, \mathcal{Q}_2, \varphi)$ clearly satisfies all assertions of Theorem \ref{thm:smoothing_homeomorphisms}.

    It thus remains to explain why we can always reduce to a situation where the plug $\psi^*\overline{\mathcal{P}}_2$ can be made disjoint from $\mathcal{P}_1$ by sliding it along the characteristic foliation of $\omega_1 \# \mathcal{P}_1$.

   For $i\in \{1,2\}$, write $\mathcal{P}_i = (\Sigma_i,\omega_{\Sigma_i},\alpha_i, (\varphi_i^t)_{t\in [0,1]})$. Note that by Remark \ref{rem:making_plug_embedding_smooth}, we can assume that $\alpha_i$ are smooth embeddings of Hamiltonian structures. After possibly shrinking the images of $\mathcal{P}_i$ as in Remark \ref{rem:shrinking_plugs}, we can assume that there exist surfaces $(\Sigma_i',\omega_{\Sigma_i'})$ containing $(\Sigma_i,\omega_{\Sigma_i})$ as relatively compact open subsets such that $\alpha_i$ extends to a smooth embedding of Hamiltonian structures
    \begin{equation*}
        \alpha_i' : ((-1,2)\times \Sigma_i',\omega_{\Sigma_i'}) \hookrightarrow (Y_i,\omega_i).
    \end{equation*}
    
    Consider a compact surface with smooth boundary $S\subset \Sigma_1'$ containing $\overline{\Sigma}_1$ in its interior. Let $t_0\in (-1,2)$ and let $U\subset (-1,2)\times \Sigma_1'$ be an arbitrary open neighborhood of $\{t_0\}\times S$. Then there exists a compactly supported homeomorphism of $C^0$ Hamiltonian structures
    \begin{equation*}
        \Psi:((-1,2)\times \Sigma_1',\omega_{\Sigma_1'}) \rightarrow ((-1,2)\times \Sigma_1',\omega_{\Sigma_1'})
    \end{equation*}
    such that $[0,1] \times \overline{\Sigma}_1 \subset \Psi(U)$. Indeed, any compactly supported homeomorphism of $((-1,2)\times \Sigma_1',\omega_{\Sigma_1'})$ takes the form $\Psi(t,p) = (f_p(t),p)$ for a family of compactly supported homeomorphisms $f_p$ of $(-1,2)$ which agree with the identity for $p$ outside some compact subset of $\Sigma_1'$. One can arrange $f_p$ such that, for every $p\in S$, the homeomorphism $f_p$ maps an arbitrarily small neighborhood of $t_0$ to the interval $(-1/2,3/2)$. The image of $U$ under the resulting homeomorphism $\Psi$ will then contain $[0,1]\times \overline{\Sigma}_1$, as desired. Let us also point out that the group of compactly supported homeomorphisms of $((-1,2)\times \Sigma_1',\omega_{\Sigma_1'})$ is connected. In particular, $\Psi$ is isotopic to the identity through compactly supported homeomorphisms of $C^0$ Hamiltonian structures.

   We can extend the isotopy $(\varphi_1^t)_{t\in [0,1]}$ in $\overline{\operatorname{Ham}}(\Sigma_1)$ to an isotopy $(\varphi_1^t)_{t\in (-1,2)}$ in $\overline{\operatorname{Ham}}(\Sigma_1')$ which is equal to the identity outside of $\Sigma_1$, and agrees with the identity for $t<0$ and with $\varphi_1^1$ for $t>1$. We can then define the homeomorphism
    \begin{equation*}
        \Phi_1 : (-1,2)\times \Sigma_1' \rightarrow (-1,2)\times \Sigma_1' \qquad \Phi_1(t,p) \coloneqq 
            (t,\varphi_1^t(p)),
    \end{equation*}
    see Subsection \ref{subsec:C0_plugs}. Recall that by definition, $\alpha_1'$ is an embedding of $C^0$ Hamiltonian structures
    \begin{equation*}
        \alpha_i' : ((-1,2)\times \Sigma_1',(\Phi_1)_*\omega_{\Sigma_1'}) \hookrightarrow (Y_1,\omega_1\#\mathcal{P}_1).
    \end{equation*}

  The homeomorphism $\Phi_1$ leaves $\{t_0\}\times S$ invariant. The above discussion thus implies that for every open neighborhood $U$ of $\{t_0\}\times S$, there exists a compactly supported homeomorphism $\Psi$ of $((-1,2)\times \Sigma_1',(\Phi_1)_*\omega_{\Sigma_1'})$ isotopic to the identity through such homeomorphisms such that $\Psi(U)$ contains $[0,1]\times \overline{\Sigma}_1$. From this observation, we obtain the following statement: Suppose that the closure of the image of $\psi^*\overline{\mathcal{P}}_2$ is disjoint from $\alpha_1'(\{t_0\}\times S)$. Then one can slide the plug $\psi^*\overline{\mathcal{P}}_2$ along the characteristic foliation of $\omega_1\#\mathcal{P}_1$ to make it disjoint from $\mathcal{P}_1$.

    We therefore need to explain how to reduce to the case where the closure of the image of $\psi^*\overline{\mathcal{P}}_2$ is disjoint from $\alpha_1'(\{t_0\}\times S)$. We set $T\coloneqq \psi(\alpha_1'(\{t_0\}\times S))$. This defines a compact surface in $(Y_2,\omega_2\#\mathcal{P}_2)$ which is transverse to the characteristic foliation, in the sense described in the proof of Claim \ref{cl:disjoint-discs-covering}. In what follows, we identify $((-1,2)\times \Sigma_2',\omega_{\Sigma_2'})$ with its image under the embedding $\alpha_2'$ and work in these coordinates. Since $T$ is transverse to the characteristic foliation, we can find finitely many open discs $B_1,\dots,B_n \Subset \Sigma_2'$ covering $\overline{\Sigma}_2$ and numbers $-1<s_1<\dots <s_n<0$ such that the closures of the discs $\{s_i\}\times B_i$ are disjoint from $T$ (cf. proof of Claim \ref{cl:disjoint-discs-covering}). Now, choose embeddings of $C^0$ Hamiltonian structures
    \begin{equation*}
        \beta_i : ((0,1)\times B_i,\omega_{B_i}) \hookrightarrow ((-1,0)\times \Sigma_2',\omega_{\Sigma_2'}) \subset (Y,\omega_2)
    \end{equation*}
    such that the closures of their images are pairwise disjoint and also disjoint from $T$.

    Pick a smooth isotopy $(\tilde\varphi_2^t)_{t\in [0,1]}$ in $\operatorname{Ham}(\Sigma_2)$ which $C^0$ approximates the continuous isotopy $(\varphi_2^t)_{t\in [0,1]}$; the existence of $(\tilde\varphi_2^t)$ is well-known and it can be deduced from the following two facts: first, every homeomorphism in $\overline{\Ham}$ can be written as a $C^0$ limit of elements in $\Ham$; and second, $\operatorname{Ham}$ and $\overline{\Ham}$  are both locally path connected; see, for example, \cite[Cor.\ 2]{Serraille} or \cite[Lem.\ 3.2]{Sey13}.  We can take $\gamma$ to be arbitrarily $C^0$ close to the identity by choosing better approximations $\tilde{\varphi}_2^t$. By Proposition \ref{prop:fragmentation_hamiltonian_homeomorphisms}, we can therefore find a fragmentation $\gamma = \gamma_n \circ \cdots \circ \gamma_1$ into Hamiltonian homeomorphisms $\gamma_i \in \overline{\operatorname{Ham}}(B_i)$ which are $C^0$ close to the identity. For each $i$, pick a $C^0$ small isotopy $(\gamma_i^t)_{t\in [0,1]}$ connecting the identity to $\gamma_i$. We define $\omega_2$-plugs
    \begin{equation*}
        \mathcal{Q}_2 \coloneqq (\Sigma_2,\omega_{\Sigma_2},\alpha_2,(\tilde\varphi_2^t)_{t\in [0,1]}) \qquad \text{and}\qquad \mathcal{Z}_i \coloneqq (B_i,\omega_{B_i},\beta_i,(\gamma_i^t)_{t\in [0,1]}).
    \end{equation*}
    Note that these plugs are disjoint. Moreover, $\mathcal{Q}_2$ is a smooth plug because both $\alpha_2$ and $\tilde\varphi_2^t$ are smooth. We also define the $\omega_2$-plug $\mathcal{Z}$ to be the disjoint union of the plugs $\mathcal{Z}_i$. Using Lemma \ref{lem:C0_plugs_with_same_time_1_maps} and the identity $\varphi_2^1 = \tilde\varphi_2^1\circ \gamma_n^1 \circ \cdots \circ \gamma_1^1$, we see that there exists a homeomorphism of $C^0$ Hamiltonian structures
    \begin{equation*}
        \chi: (Y_2,\omega_2\#\mathcal{P}_2) \rightarrow (Y_2,\omega_2\#(\mathcal{Q}_2\sqcup \mathcal{Z}))
    \end{equation*}
    isotopic to the identity. Moreover, since $\tilde\varphi_2^t$ is $C^0$ close to $\varphi_2^t$ and all $\gamma_i^t$ are $C^0$ close to the identity, it is clear from the proof of Lemma \ref{lem:C0_plugs_with_same_time_1_maps} that the homeomorphsim $\chi$ can be taken $C^0$ close to the identity. In particular, we can assume that $\chi(T)$ is disjoint from the closure of the image of $\mathcal{Z}$. Now consider the composite
    \begin{equation*}
        \chi\circ \psi : (Y_1,\omega_1\#\mathcal{P}_1) \rightarrow (Y_2,(\omega_2\#\mathcal{Q}_2) \# \mathcal{Z})
    \end{equation*}
    and note that $(\chi\circ \psi)^*\overline{\mathcal{Z}}$ is disjoint from $\alpha_1'(\{t_0\}\times S)$. As discussed above, we can slide $(\chi\circ \psi)^*\overline{\mathcal{Z}}$ away from $\mathcal{P}_1$ along the characteristic foliation of $\omega_1\#\mathcal{P}_1$. Hence, there exist a smooth $\omega_1$-plug $\mathcal{Q}_1$ of Calabi invariant
    \begin{equation}
    \label{eq:smoothing_homeomorphisms_proof_last_step}
        \operatorname{Cal}(\mathcal{Q}_1) = \overline{\operatorname{Cal}}(\mathcal{P}_1) - \overline{\operatorname{Cal}}(\mathcal{Z})
    \end{equation}
    and a diffeomorphism of Hamiltonian structures $\varphi: (Y_1,\omega_1\#\mathcal{Q}_1) \rightarrow (Y_2,\omega_2\#\mathcal{Q}_2)$ isotopic to $\chi\circ \psi$ and hence also to $\psi$. Now observe that by naturality of our $\mathcal{R}$-valued extension of the Calabi homomorphism, we have
    \begin{equation*}
        \overline{\operatorname{Cal}}(\mathcal{P}_2) = \operatorname{Cal}(\mathcal{Q}_2) + \overline{\operatorname{Cal}}(\mathcal{Z}).
    \end{equation*}
    Together with identity \eqref{eq:smoothing_homeomorphisms_proof_last_step}, this readily yields the desired identity
    \begin{equation*}
        \operatorname{Cal}(\mathcal{Q}_1) - \overline{\operatorname{Cal}}(\mathcal{P}_1) = \operatorname{Cal}(\mathcal{Q}_2) - \overline{\operatorname{Cal}}(\mathcal{P}_2).
    \end{equation*}
    This completes the proof of 
    Theorem \ref{thm:smoothing_homeomorphisms} in the general case.
\end{proof}

\section{Flux and helicity of \texorpdfstring{$C^0$}{C0} Hamiltonian structures}
\label{sec:flux and helicity extension}

In this section, we use the results of Section \ref{sec:smoothings-mod-plugs} to extend the definitions of flux and helicity to $C^0$ Hamiltonian structures. We prove our main result, Theorem \ref{thm:coherent_helicity_extension}, on the existence and uniqueness of a universal $\mathcal{R}$-valued helicity extension. Moreover, we prove Proposition \ref{prop:universality}.

\subsection{Flux}
\label{sec:flux_C0_Ham_structures}

We generalize the notion of flux, which we introduced for smooth Hamiltonian structures in Section \ref{def:smooth_flux}, to $C^0$ Hamiltonian structures.

\medskip

We begin with the following lemma, which is a straightforward consequence of Theorem \ref{thm:smoothing_homeomorphisms}.

\begin{lem}
\label{lem:C0_flux_definition}
    For $i \in \{1,2\}$, let $Y_i$ be a closed $3$-manifold and let $\omega_i$ be a smooth Hamiltonian structure on $Y_i$. Let $\mathcal{P}_i$ be a $C^0$ $\omega_i$-plug and assume that
    \begin{equation*}
        \psi : (Y_1,\omega_1\# \mathcal{P}_1) \rightarrow (Y_2,\omega_2\#\mathcal{P}_2)
    \end{equation*}
    is a homeomorphism of $C^0$ Hamiltonian structures. Then,
    \begin{equation*}
        \operatorname{Flux}(\omega_1) = \psi^* \operatorname{Flux}(\omega_2).
    \end{equation*}
\end{lem}

\begin{proof}
    By Theorem \ref{thm:smoothing_homeomorphisms}, we can find a smooth $\omega_i$-plug $\mathcal{Q}_i$ for $i\in \{1,2\}$ and a diffeomorphism of Hamiltonian structures
    \begin{equation*}
        \varphi : (Y_1,\omega_1\# \mathcal{Q}_1) \rightarrow (Y_2,\omega_2\# \mathcal{Q}_2)
    \end{equation*}
    which is isotopic to $\psi$. Clearly, we have
    \begin{equation*}
        \operatorname{Flux}(\omega_1\# \mathcal{Q}_1) = \varphi^* \operatorname{Flux}(\omega_2 \# \mathcal{Q}_2).
    \end{equation*}
    Recall from Lemma \ref{lem:flux_invariant_under_plug_insertion_smooth_case} that $\operatorname{Flux}(\omega_i\#\mathcal{Q}_i) = \operatorname{Flux}(\omega_i)$. The assertion of the lemma now follows from the observation that the actions of $\varphi$ and $\psi$ on cohomology agree because these two maps are isotopic.
\end{proof}

\begin{definition}
\label{def:flux_C0}
    Let $\Omega$ be a $C^0$ Hamiltonian structure on a closed topological $3$-manifold $Y$. By Theorem \ref{thm:from_C0_to_smooth_plus_plug}, we can find a smooth Hamiltonian structure $\omega$ on $Y$ and a $C^0$ $\omega$-plug $\mathcal{P}$ such that $\Omega = \mathcal{\omega} \# \mathcal{P}$. We define the \textit{flux} of $\Omega$ to be the cohomology class
    \begin{equation*}
        \overline{\operatorname{Flux}}(\Omega) \coloneqq \operatorname{Flux}(\omega) \in H^2(Y;\R),
    \end{equation*}
    which is independent of choices by Lemma \ref{lem:C0_flux_definition}. We say that a $C^0$ Hamiltonian structure $\Omega$ is \textit{exact} if $\overline{\operatorname{Flux}}(\Omega) = 0$.
\end{definition}

Note that if $\omega$ is a smooth Hamiltonian structure, then $\overline{\operatorname{Flux}}(\omega) = \operatorname{Flux}(\omega)$ because we can simply take $\mathcal{P}$ to be the trivial plug in Definition \ref{def:flux_C0}. Moreover, it is immediate from Lemma \ref{lem:C0_flux_definition} that if $\psi: (Y_1,\Omega_1)\rightarrow (Y_2,\Omega_2)$ is a homeomorphism of $C^0$ Hamiltonian structures, then $\overline{\operatorname{Flux}}(\Omega_1) = \psi^* \overline{\operatorname{Flux}}(\Omega_2)$.

\begin{lem}
\label{lem:flux_invariant_under_plug_insertion_C0_case}
    Let $\Omega$ be a $C^0$ Hamiltonian structure on a closed topological manifold $Y$ and let $\mathcal{P}$ be an $\Omega$-plug. Then
    \begin{equation*}
        \overline{\operatorname{Flux}}(\Omega) = \overline{\operatorname{Flux}}(\Omega\# \mathcal{P}).
    \end{equation*}
\end{lem}

\begin{proof}
    After possibly slightly shrinking the image of the plug $\mathcal{P}$, see Remark \ref{rem:shrinking_plugs}, we can assume that the complement of the image of $\mathcal{P}$ contains an exhaustive open set $U$. By Theorem \ref{thm:from_C0_to_smooth_plus_plug} and Remark \ref{rem:from_C0_to_smooth_plus_plug_stronger}, we can find a smooth Hamiltonian structure $\omega$ on $Y$ and a $C^0$ $\omega$-plug $\mathcal{Q}$ whose image is contained in $U$ such that $\Omega = \omega \# \mathcal{Q}$. Let $\mathcal{Z}$ be the disjoint union of the plugs $\mathcal{Q}$ and $\mathcal{P}$. Now $\mathcal{Z}$ is an $\omega$-plug and $\omega\# \mathcal{Z} = \Omega\# \mathcal{P}$. Using the definition of $\overline{\operatorname{Flux}}$, we can therefore compute
    \begin{equation*}
        \overline{\operatorname{Flux}}(\Omega \#\mathcal{P}) = \overline{\operatorname{Flux}}(\omega \# \mathcal{Z}) = \operatorname{Flux}(\omega) = \overline{\operatorname{Flux}}(\omega \# \mathcal{Q}) = \overline{\operatorname{Flux}}(\Omega).
    \end{equation*}
\end{proof}

\subsection{Helicity}
\label{sec:helictiy-extended}

The goal of this section is to prove Theorem \ref{thm:coherent_helicity_extension}. 

\medskip

The uniqueness of the extension $\overline{\mathcal{H}}$ in Theorem \ref{thm:coherent_helicity_extension} is a consequence of Theorem \ref{thm:from_C0_to_smooth_plus_plug}. Indeed, consider an arbitrary exact $C^0$ Hamiltonian structure $\Omega$ on a closed topological $3$-manifold $Y$. By Theorem \ref{thm:from_C0_to_smooth_plus_plug}, there exist a smooth Hamiltonian structure $\omega$ on $Y$ and an $\omega$-plug $\mathcal{P}$ such that $\Omega = \omega\# \mathcal{P}$. By Lemma \ref{lem:flux_invariant_under_plug_insertion_C0_case}, the smooth Hamiltonian structure $\omega$ is exact. It follows from properties \ref{item:coherent_helicity_extensions_smooth}, \ref{item:coherent_helicity_extensions_invariance}, and \ref{item:coherent_helicity_extensions_plug} in Theorem \ref{thm:coherent_helicity_extension} that we must have
\begin{equation}
\label{eq:coherent_helicity_extension_proof_computation_uniqueness}
    \overline{\mathcal{H}}(\Omega) = \overline{\mathcal{H}}(\omega\# \mathcal{P}) = \overline{\mathcal{H}}(\omega) + \overline{\operatorname{Cal}}(\mathcal{P}) = \mathcal{H}(\omega) + \overline{\operatorname{Cal}}(\mathcal{P}).
\end{equation}
This proves uniqueness.

We turn to the proof of existence. We begin with the following lemma, which is a corollary of Theorem \ref{thm:smoothing_homeomorphisms}.

\begin{lem}
\label{lem:coherent_helicity_extension_proof_existence}
    For $i\in \{1,2\}$, let $Y_i$ be a closed $3$-manifold and let $\omega_i$ be an exact smooth Hamiltonian structure on $Y_i$. Let $\mathcal{P}_i$ be a $C^0$ $\omega_i$-plug and assume that
    \begin{equation*}
        \psi : (Y_1, \omega_1\# \mathcal{P}_1) \rightarrow (Y_2,\omega_2\# \mathcal{P}_2)
    \end{equation*}
    is a homeomorphism of $C^0$ Hamiltonian structures. Then
    \begin{equation*}
        \mathcal{H}(\omega_1) + \overline{\operatorname{Cal}}(\mathcal{P}_1) = \mathcal{H}(\omega_2) + \overline{\operatorname{Cal}}(\mathcal{P}_2).
    \end{equation*}
\end{lem}

\begin{proof}
    By Theorem \ref{thm:smoothing_homeomorphisms}, we can find a smooth $\omega_i$-plug $\mathcal{Q}_i$ for $i\in \{1,2\}$ and a diffeomorphism of Hamiltonian structures
    \begin{equation*}
        \varphi : (Y_1,\omega_1\#\mathcal{Q}_1) \rightarrow (Y_1,\omega_1\#\mathcal{Q}_1)
    \end{equation*}
    such that
    \begin{equation*}
        \operatorname{Cal}(\mathcal{Q}_1) - \overline{\operatorname{Cal}}(\mathcal{P}_1)  = \operatorname{Cal}(\mathcal{Q}_2) - \overline{\operatorname{Cal}}(\mathcal{P}_2)
    \end{equation*}
    Since $\varphi$ is a diffeomorphism, we clearly have $\mathcal{H}(\omega_1\#\mathcal{Q}_1) = \mathcal{H}(\omega_2\#\mathcal{Q}_2)$. It follows from Lemma \ref{lem:Helicity-Plug} that $\mathcal{H}(\omega_i\#\mathcal{Q}_i) = \mathcal{H}(\omega_i) + \operatorname{Cal}(\mathcal{Q}_i)$. The assertion of the lemma is an immediate consequence of these facts.
\end{proof}

Let $\Omega$ be an arbitrary exact $C^0$ Hamiltonian structure on a closed topological $3$-manifold $Y$. By Theorem \ref{thm:from_C0_to_smooth_plus_plug}, we may write $\Omega = \omega \# \mathcal{P}$ for a smooth Hamiltonian structure $\omega$ and a $C^0$ $\omega$-plug $\mathcal{P}$. By Lemma \ref{lem:flux_invariant_under_plug_insertion_C0_case}, the smooth Hamiltonian structure $\omega$ is exact. The idea is to use identity \eqref{eq:coherent_helicity_extension_proof_computation_uniqueness} as a definition, i.e.\ to set
\begin{equation}
\label{eq:coherent_helicity_extension_proof_existence}
    \overline{\mathcal{H}}(\Omega) \coloneqq \mathcal{H}(\omega) + \overline{\operatorname{Cal}}(\mathcal{P}).
\end{equation}
By Lemma \ref{lem:coherent_helicity_extension_proof_existence}, this definition is independent of choices. It remains to check that our extension $\overline{\mathcal{H}}$ satisfies properties \ref{item:coherent_helicity_extensions_smooth}, \ref{item:coherent_helicity_extensions_invariance}, and \ref{item:coherent_helicity_extensions_plug} in Theorem \ref{thm:coherent_helicity_extension}. For property \ref{item:coherent_helicity_extensions_smooth}, observe that if $\Omega = \omega$ is already smooth, we can simply take $\mathcal{P}$ to be the trivial plug in \eqref{eq:coherent_helicity_extension_proof_existence}. Property \ref{item:coherent_helicity_extensions_invariance} is immediate from Lemma \ref{lem:coherent_helicity_extension_proof_existence}.  Finally, suppose that $\Omega$ is an exact $C^0$ Hamiltonian structure and that $\mathcal{P}$ is an $\Omega$-plug. After possibly slightly shrinking the image of $\mathcal{P}$, we can assume that the complement of $\mathcal{P}$ contains an exhaustive open set $U$. By Theorem \ref{thm:from_C0_to_smooth_plus_plug} and Remark \ref{rem:from_C0_to_smooth_plus_plug_stronger}, we can write $\Omega = \omega \# \mathcal{Q}$ for a smooth Hamiltonian structure $\omega$ on $Y$ and a $C^0$ $\omega$-plug $\mathcal{Q}$ whose image is contained in $U$. By Lemma \ref{lem:flux_invariant_under_plug_insertion_C0_case}, the Hamiltonian structure $\omega$ is exact. The disjoint union $\mathcal{Z}$ of $\mathcal{P}$ and $\mathcal{Q}$ is an $\omega$-plug and we have $\Omega \# \mathcal{P} = \omega \# \mathcal{Z}$. Moreover, we have $\overline{\operatorname{Cal}}(\mathcal{Z}) = \overline{\operatorname{Cal}}(\mathcal{Q}) + \overline{\operatorname{Cal}}(\mathcal{P})$. We can therefore compute
\begin{equation*}
    \overline{\mathcal{H}}(\Omega\#\mathcal{P}) = \mathcal{\omega} + \overline{\operatorname{Cal}}(\mathcal{Z}) = \mathcal{\omega} + \overline{\operatorname{Cal}}(\mathcal{Q}) + \overline{\operatorname{Cal}}(\mathcal{P}) = \overline{\mathcal{H}}(\Omega) + \overline{\operatorname{Cal}}(\mathcal{P}),
\end{equation*}
verifying property \ref{item:coherent_helicity_extensions_plug}. This concludes the proof of Theorem \ref{thm:coherent_helicity_extension}.\qed

\subsection{Universality}

The goal of this subsection is to prove Proposition \ref{prop:universality}. Let $\R\subset A$ be an arbitrary extension of abelian groups and let $\overline{\mathcal{H}}^A$ be an $A$-valued extension of helicity to $C^0$ Hamiltonian structures. Assume that $\overline{\mathcal{H}}^A$ satisfies the Extension and Invariance properties in Theorem \ref{thm:coherent_helicity_extension}. Moreover, assume that $\overline{\mathcal{H}}^A$ satisfies plug homogeneity. Our task is to show that there exists a unique group homomorphism $p:\mathcal{R}\rightarrow A$ over $\R$ such that $\overline{\mathcal{H}}^A = p \circ \overline{\mathcal{H}}$.

Let $(\Sigma,\omega_\Sigma)$ be a non-empty connected open surface with area form. Let us begin by defining a map
\begin{equation*}
    \tilde{p}_\Sigma : \overline{\operatorname{Ham}}(\Sigma)\rightarrow A
\end{equation*}
as follows. Let $\varphi \in \overline{\operatorname{Ham}}(\Sigma)$ be a Hamiltonian homeomorphism. Pick an arbitrary isotopy $\varphi^t$ in $\overline{\operatorname{Ham}}(\Sigma)$ connecting the identity to $\varphi$. Moreover, pick an arbitrary $C^0$ Hamiltonian structure $(Y,\Omega)$ such that there exists an embedding of $C^0$ Hamiltonian structures $\alpha : ((0,1)\times \Sigma,\omega_\Sigma) \hookrightarrow (Y,\Omega)$. Define the $\Omega$-plug $\mathcal{P} \coloneqq (\Sigma,\omega_\Sigma,\alpha,(\varphi^t)_{t\in [0,1]})$ and set
\begin{equation*}
    \tilde{p}_\Sigma(\varphi) \coloneqq \overline{\mathcal{H}}^A(\Omega \# \mathcal{P}) - \overline{\mathcal{H}}^A(\Omega) \in A.
\end{equation*}
By plug homogeneity, this is well-defined and independent of choices.

We claim that $\tilde{p}_\Sigma$ is a group homomorphism. Indeed, consider two Hamiltonian homeomorphisms $\varphi,\psi \in \overline{\operatorname{Ham}}(\Sigma)$. Pick Hamiltonian isotopies $(\varphi^t)_{t\in [0,1]}$ and $(\psi^t)_{t\in [0,1]}$ starting at the identity and ending at $\varphi$ and $\psi$, respectively. We define the concatenation $(\chi^t)_{t\in [0,1]}$ by
\begin{equation*}
    \chi^t \coloneqq \begin{cases}
        \varphi^{2t} & t\in [0,1/2]\\
        \psi^{2t-1}\circ \varphi^1 & t \in [1/2,1].
    \end{cases}
\end{equation*}
Now pick a $C^0$ Hamiltonian structure $(Y,\Omega)$ and an embedding of $C^0$ Hamiltonian structures $\alpha : ((0,1)\times \Sigma, \omega_\Sigma) \hookrightarrow (Y,\Omega)$. Define the plug $\mathcal{Z} \coloneqq (\Sigma,\omega_\Sigma,\alpha,(\chi^t)_{t\in [0,1]})$. By construction, it is possible to split the plug $\mathcal{Z}$ into two disjoint plugs $\mathcal{P}$ and $\mathcal{Q}$ inserting the isotopies $\varphi^t$ and $\psi^t$, respectively, such that $\Omega \# \mathcal{Z} = \Omega\# (\mathcal{P} \sqcup \mathcal{Q})$. We can then compute
\begin{eqnarray*}
    \tilde{p}_\Sigma(\psi \circ \varphi) & = & \overline{\mathcal{H}}^A(\Omega \# \mathcal{Z}) - \overline{\mathcal{H}}^A(\Omega) \\
    & = & \overline{\mathcal{H}}^A((\Omega \# \mathcal{P}) \# \mathcal{Q}) - \overline{\mathcal{H}}^A(\Omega) \\
    & = & \overline{\mathcal{H}}^A((\Omega \# \mathcal{P}) \# \mathcal{Q}) - \overline{\mathcal{H}}^A(\Omega \# \mathcal{P}) + \overline{\mathcal{H}}^A(\Omega \# \mathcal{P}) - \overline{\mathcal{H}}^A(\Omega) \\
    & = & \tilde{p}_\Sigma(\psi) + \tilde{p}_\Sigma(\varphi).
\end{eqnarray*}
Here we use plug homogeneity and the definition of $\tilde{p}_\Sigma$. This shows that $\tilde{p}_\Sigma$ is a homomorphism.

Next, we claim that the restriction of $\tilde{p}_\Sigma$ to the group of Hamiltonian diffeomorphisms $\operatorname{Ham}(\Sigma)$ agrees with the Calabi homomorphism $\operatorname{Cal}_\Sigma$ via the inclusion $\R \subset A$. Suppose that $\varphi \in \operatorname{Ham}(\Sigma)$. Then we can choose the Hamiltonian isotopy, the Hamiltonian structure, and the plug embedding involved in the definition of $\tilde{p}_\Sigma$ to be smooth. Let $(Y,\omega)$ and $\mathcal{P}$ be the resulting smooth Hamiltonian structure and plug, respectively. We compute
\begin{eqnarray*}
    \tilde{p}_\Sigma(\varphi) & = & \overline{\mathcal{H}}^A(\omega \# \mathcal{P}) - \overline{\mathcal{H}}^A(\omega) \\
    & = & \mathcal{H}(\omega\# \mathcal{P}) - \mathcal{H}(\omega) \\
    & = & \operatorname{Cal}_\Sigma(\varphi).
\end{eqnarray*}
Here the second equality follows from the assumption that $\overline{\mathcal{H}}^A$ satisfies the Extension property in Theorem \ref{thm:coherent_helicity_extension}. The third equality follow from Lemma \ref{lem:helicity_after_plug_insertion}.

Suppose that $\iota : \Sigma_1 \hookrightarrow \Sigma_2$ is an area- and orientation-preserving embedding of open surfaces. Then we have
\begin{equation}
\label{eq:universality_proof_naturality}
    \tilde{p}_{\Sigma_2} \circ \overline{\operatorname{Ham}}(\iota) = \tilde{p}_{\Sigma_1},
\end{equation}
where $\overline{\operatorname{Ham}}(\iota)$ is the homomorphism between Hamiltonian homeomorphism groups obtained by pushforward via $\iota$; see Section \ref{sec:hamiltonian_homeos_and_calabi}. In order to see this, consider a $C^0$ Hamiltonian structure $(Y,\Omega)$ and an embedding $\alpha : ((0,1)\times \Sigma_2,\omega_{\Sigma_2}) \hookrightarrow (Y,\Omega)$. Given $\varphi \in \overline{\operatorname{Ham}}(\Sigma_1)$, we can form a plug $\mathcal{P}_2$ by inserting the homeomorphism $\overline{\operatorname{Ham}}(\iota)(\varphi)$ via the embedding $\alpha$. We can form a second plug $\mathcal{P}_1$ by inserting the homeomorphism $\varphi$ via the embedding $\alpha \circ (\operatorname{id}_{(0,1)}\times \iota)$. Note that the $C^0$ Hamiltonian structures obtained by inserting $\mathcal{P}_1$  and $\mathcal{P}_2$ agree, i.e.\ $\Omega \# \mathcal{P}_1 = \Omega \# \mathcal{P}_2$. Identity \eqref{eq:universality_proof_naturality} is then immediate from the definition of $\tilde{p}_{\Sigma_j}$.

Combining the above observations, we see that the homomorphisms $\tilde{p}_\Sigma$ descend to a well-defined group homomorphism
\begin{equation*}
    p : \mathcal{R} \cong \overline{\operatorname{Ham}}^{\operatorname{ab}}(\Sigma) \rightarrow A
\end{equation*}
over $\R$ which is independent of $\Sigma$. We need to check that $\overline{\mathcal{H}}^A = p \circ \overline{\mathcal{H}}$.

First, observe that since $p$ is a homomorphism over $\R$, and both $\overline{\mathcal{H}}^A$ and $\overline{\mathcal{H}}$ satisfy the Extension property in Theorem \ref{thm:coherent_helicity_extension}, we have
\begin{equation*}
    \overline{\mathcal{H}}^A(\omega) = p\circ \overline{\mathcal{H}}(\omega) = \mathcal{H}(\omega)
\end{equation*}
for every smooth Hamiltonian structure $\omega$.

It is immediate from the definition of $p$ that $\overline{\mathcal{H}}^A$ satisfies the following version of the Calabi property in Theorem \ref{thm:coherent_helicity_extension}: for every $C^0$ Hamiltonian structure $\Omega$ and every $\Omega$-plug $\mathcal{P}$, we have
\begin{equation*}
    \overline{\mathcal{H}}^A(\Omega \# \mathcal{P}) = \overline{\mathcal{H}}^A(\Omega) + p \circ \overline{\operatorname{Cal}}(\mathcal{P}).
\end{equation*}

Now consider an arbitrary $C^0$ Hamiltonian structure $(Y,\Omega)$. By Theorem \ref{thm:from_C0_to_smooth_plus_plug}, we can write $\Omega = \omega \# \mathcal{P}$ for some smooth Hamiltonian structure $\omega$ and a $C^0$ plug $\mathcal{P}$. We compute
\begin{eqnarray*}
    \overline{\mathcal{H}}^A(\Omega) & = & \overline{\mathcal{H}}^A(\omega \# \mathcal{P}) \\
    & = & \overline{\mathcal{H}}^A(\omega) + p\circ \overline{\operatorname{Cal}}(\mathcal{P}) \\
    & = & p(\mathcal{H}(\omega) + \overline{\operatorname{Cal}}(\mathcal{P})) \\
    & = & p\circ \overline{\mathcal{H}}(\omega \# \mathcal{P}) \\ 
    & = & p\circ \overline{\mathcal{H}}(\Omega).
\end{eqnarray*}
Here the second equality uses the Calabi property for $\overline{\mathcal{H}}^A$, the third equality uses the Extension property and the fact that $p$ is a homomorphism over $\R$, and the fourth equality uses the Calabi property of $\overline{\mathcal{H}}$.

This concludes the proof of the existence part of Proposition \ref{prop:universality}. For the uniqueness part, simply observe that every element of $\mathcal{R}$ is attained as the universal $\mathcal{R}$-valued helicity $\overline{\mathcal{H}}(\Omega)$ of some $C^0$ Hamiltonian structure $\Omega$.\qed

\section{The relationship between Hamiltonian structures and volume-preserving flows}
\label{sec:relationship_hamiltonian_structures_vol_pres_flows} 

As mentioned earlier, there exists a close relationship between Hamiltonian structures and volume-preserving flows. The goal of this section is to clarify this relation, particularly in the $C^0$ setting.  This is needed for deducing Theorem \ref{thm:topological_invariance_helicity_flows}, which is stated for volume-preserving flows, from our main result Theorem \ref{thm:coherent_helicity_extension}, which is stated for Hamiltonian structures.  Theorem \ref{thm:topological_invariance_helicity_flows} is proven in Section \ref{sec:proof-thm-flows}.

\subsection{The smooth case}
\label{sec:smooth-case}

Consider a closed oriented smooth $3$-manifold $Y$. Suppose $\mu$ is a volume form on $Y$ that is compatible with its orientation. Let $\omega$ denote a Hamiltonian structure, i.e., a closed maximally nondegenerate $2$-form. Finally, let $\varphi$ be a smooth, fixed-point-free flow on $Y$, generated by a nowhere vanishing vector field $X$.

Consider $\R^3 = \R\times \R^2$ equipped with coordinates $(t,x,y)$.  Let $\mu_{\operatorname{std}} = dt\wedge dx\wedge dy$ be the standard volume form and equip $\R^3$ with the orientation induced by this volume form.  Let $\varphi_{\operatorname{std}}$ be the flow generated by the vector field $X_{\operatorname{std}} = \partial_t$.  Recall that the standard Hamiltonian structure $\omega_{\operatorname{std}}$ is given by the $2$-form $\omega_{\operatorname{std}} = dx\wedge dy$.

\begin{lem}
\label{lem:compatible_triple_equivalent_conditions_smooth}
    The following statements are equivalent:
    \begin{enumerate}
        \item \label{item:compatible_triple_equivalent_conditions_smooth_one} We have $\iota_X\mu = \omega$.
        \item \label{item:compatible_triple_equivalent_conditions_smooth_two} The triple $(\varphi,\mu,\omega)$ is locally diffeomorphic to the triple $(\varphi_{\operatorname{std}},\mu_{\operatorname{std}}, \omega_{\operatorname{std}})$.
    \end{enumerate}
\end{lem}

\begin{proof}
    Note that  $\iota_{X_{\operatorname{std}}} \mu_{\operatorname{std}} = \omega_{\operatorname{std}}$ and, moreover, this condition is preserved by diffeomorphisms and can be checked locally.  Hence statement \ref{item:compatible_triple_equivalent_conditions_smooth_two} implies statement \ref{item:compatible_triple_equivalent_conditions_smooth_one}.

    Conversely, suppose that $\iota_X\mu = \omega$.  Given an arbitrary point $p\in Y$, we may pick $a>0$ sufficiently small and an embedding $\iota : B(a)\hookrightarrow Y$ such that $\iota(0) = p$ and the pull back $\iota^*\omega$ agrees with the the standard area form on $B(a)$. For $\varepsilon>0$ sufficiently small, extend $\iota$ to an embedding $\iota : (-\varepsilon,\varepsilon)\times B(a) \hookrightarrow Y$ via the flow $\varphi$. It can be checked that the pull back of $(\varphi,\mu,\omega)$ via $\iota$ is given by $(\varphi_{\operatorname{std}},\mu_{\operatorname{std}},\omega_{\operatorname{std}})$. 
\end{proof}

\begin{definition}
\label{def:compatible_triple_smooth}
    We call $(\varphi,\mu,\omega)$ a \textit{compatible triple} if the equivalent conditions in Lemma \ref{lem:compatible_triple_equivalent_conditions_smooth} are satisfied.
\end{definition}

Note that given $(\mu,\omega)$, there exists a unique vector field $X$ satisfying $\iota_X\mu= \omega$ and we can extend $(\mu,\omega)$ to a compatible triple by taking $\varphi$ to be the flow generated by $X$.  Below we formulate criteria for extension of $(\varphi, \mu)$ or $(\varphi, \omega)$ to compatible triples.

\begin{prop}
\label{prop:compatible_triple_existence_extension_smooth}
    The pair $(\varphi,\mu)$ extends to a compatible triple if and only if it is locally diffeomorphic to $(\varphi_{\operatorname{std}},\mu_{\operatorname{std}})$. The analogous statements hold for the pairs $(\varphi,\omega)$ and $(\mu,\omega)$.
\end{prop}
\begin{proof}
    The ``only if" part of the statement is clear in view of the characterization of compatibility given in Lemma \ref{lem:compatible_triple_equivalent_conditions_smooth}. For the converse direction, note that any two components of $(\varphi_{\operatorname{std}},\mu_{\operatorname{std}},\omega_{\operatorname{std}})$ clearly extend to a compatible triple. This shows that being locally diffeomorphic to a pair contained in the triple $(\varphi_{\operatorname{std}},\mu_{\operatorname{std}},\omega_{\operatorname{std}})$ allows to locally extend to a compatible triple. But in view of the uniqueness of extensions to a compatible triple proved in Lemma \ref{lem:compatible_triple_unique_extension_smooth}, stated and proven below, local extendibility implies global extendibility.
\end{proof}

\begin{lem}
\label{lem:compatible_triple_unique_extension_smooth}
    Any two components of a compatible triple $(\varphi,\mu,\omega)$ uniquely determine the third.
\end{lem}

\begin{proof}
    This is a consequence of the following fact: Let $V$ be a $3$-dimensional vector space and suppose $v\in V$, $\sigma \in \Lambda^3V^*$, and $\tau \in \Lambda^2V^*$ are non-zero and satisfy the identity $\iota_v\sigma = \tau$.  Then, any two components of the triple $(v,\sigma,\tau)$ uniquely determine the third.
\end{proof}

\begin{prop}
\label{prop:compatible_triple_existence_extension_smooth_2}
    \begin{enumerate}
        \item \label{item:compatible_triple_existence_extension_smooth_2_phi_mu} The pair $(\varphi,\mu)$ extends to a compatible triple if and only if the flow $\varphi$ preserves $\mu$.
        \item \label{item:compatible_triple_existence_extension_smooth_2_phi_omega} The pair $(\varphi,\omega)$ extends to a compatible triple if and only if the flow $\varphi$ is positively tangent to the characteristic foliation of $\omega$.
    \end{enumerate}
\end{prop}

\begin{proof}
    If $(\varphi,\mu,\omega)$ is a compatible triple, then $\varphi$ preserves $\mu$ and the flow $\varphi$ is positively tangent to the characteristic foliation of $\omega$. This shows the ``only if" direction in both statements. Conversely, if $\varphi$ preserves $\mu$, then $\iota_X\mu$ is a maximally nondegenerate closed $2$-form and we can extend $(\varphi,\mu)$ to a compatible triple by defining the Hamiltonian structure $\omega$ to consist of the $2$-form $\omega\coloneqq \iota_X\mu$.
    
    Finally, if the flow $\varphi$ is positively tangent to the characteristic foliation of $\omega$, then there exists a unique volume form $\mu$ satisfying $\iota_X\mu = \omega$.
\end{proof}

\subsection{The \texorpdfstring{$C^0$}{C0} case}
\label{sec:c0_case}

The goal of this section is to explain the relationship between Hamiltonian structures and volume-preserving flows in the $C^0$ setting.

\subsubsection*{Topological flows and $C^0$ foliations}

Let $Y$ be a closed oriented topological $3$-manifold, for now not equipped with a measure. Consider a topological flow $\varphi$ on $Y$, i.e.\ a continuous map $\varphi : \R \times Y \rightarrow Y$ such that $\varphi(t, \cdot )$ is a homeomorphism of $Y$ for each $t \in \R$ and, moreover,
\begin{equation*}
    \varphi(t,\varphi(s,p)) = \varphi(t+s,p) \qquad \text{for all $t,s\in \R$ and $p\in Y$.}
\end{equation*}
It will be convenient to use the notation $\varphi^t(p)\coloneqq \varphi(t,p)$. In the following, we always assume that $\varphi$ is fixed-point-free, i.e.\ that there does not exist a point $p\in Y$ such that $\varphi^t(p)= p$ for all $t\in \R$.

\begin{prop}
\label{prop:flow_on_C0_foliations}
    The flow lines of a fixed-point-free topological flow $\varphi$ on $Y$ form an oriented $1$-dimensional $C^0$ foliation $\mathcal{F}_\varphi$. Conversely, for every oriented $C^0$ foliation $\mathcal{F}$ on $Y$, there exists a fixed-point-free topological flow $\varphi$ such that $\mathcal{F} = \mathcal{F}_\varphi.$
\end{prop}

\begin{proof}
    We begin by proving that the flow lines of a fixed-point-free flow $\varphi$ on $Y$ form an oriented $1$-dimensional $C^0$ foliation $\mathcal{F}_\varphi$. This actually makes important use of the fact that $Y$ has dimension three; see Remark \ref{rem:flow_to_foliation_importance_dim_3} below. Let $p \in Y$ be an arbitrary point. It was shown by Whitney \cite{whi38} (see also \cite[Lemma 1 and Corollary 1]{co75}) that $\varphi$ admits a local cross section through $p$ which is a $2$-dimensional topological disc. This means that there exist $\varepsilon >0$ and a topologically embedded $2$-dimensional disc $B\subset Y$ containing $p$ such that
    \begin{equation*}
        (-\varepsilon,\varepsilon)\times B \rightarrow Y \qquad (t,q) \mapsto \varphi^t(q)
    \end{equation*}
    is an embedding. By construction, this embedding maps straight lines $(-\varepsilon,\varepsilon)\times \{*\}$ into flow lines of $\varphi$. Its inverse is therefore a topological foliation chart. Since $p\in Y$ was arbitrary, we can cover $Y$ by such foliation charts.
    
    Next, we show that every oriented $C^0$ foliation $\mathcal{F}$ admits a fixed-point-free flow $\varphi$ such that $\mathcal{F} = \mathcal{F}_\varphi$. Our strategy is to construct a special metric $d$ on $Y$ with the property that the leaves of the characteristic foliation of $\Omega$ are rectifiable curves with respect to the metric $d$. We then use the metric $d$ to obtain, for each point $p\in Y$, a unique continuous curve $\gamma_p : \R \rightarrow L \subset Y$, where $L$ denotes the characteristic leaf containing $p$, with the following properties:

    \begin{itemize}
        \item $\gamma_p$ parametrized by arc length,
        \item  $\gamma_p(0)= p$,
        \item $\gamma_p : \R \rightarrow L$ is an orientation-preserving homeomorphism if $L$ is an open leaf and an orientation-preserving covering if $L$ is a closed leaf.
    \end{itemize}
    In this situation, we can define the flow $\varphi$ by simply setting $\varphi^t(p) \coloneqq \gamma_p(t)$.
    
    \medskip

    We now describe the construction of the metric $d$. For $\delta > 0$, we consider foliation charts of the form
    \begin{equation*}
        C_\delta := (0,1) \times D_\delta,
    \end{equation*}
    where $D_\delta$ denotes the disc of radius $\delta$. The oriented leaves of the foliation within this chart are given by
    \begin{equation*}
        (0,1) \times \{p\},
    \end{equation*}
    for $p \in D_\delta$. For some $\delta>0$, we can pick a finite collection of charts 
    \begin{equation*}
        \iota_i : C_{2\delta} \hookrightarrow Y \qquad 1\leq i \leq n
    \end{equation*}
    such that the sets $\iota_i(C_\delta)$ form an open covering of $Y$. Moreover, fix a smooth map
    \begin{equation*}
        f : C_{2\delta} \rightarrow S^3 \subset \R^4
    \end{equation*}
    with the property that there exist an open set $C_{\delta} \Subset U \Subset C_{2\delta}$ and a point $* \in S^3$ such that the restriction of $f$ to $U$ is a diffeomorphism $f|_U : U \rightarrow S^3\setminus \{*\}$ and $f(C_{2\delta}\setminus U) = *$. For $1\leq i \leq n$, we define
    \begin{equation*}
        f_i : Y \rightarrow S^3 \qquad f_i(p) \coloneqq
        \begin{cases}
            f(\iota_i^{-1}(p)) & \text{if $p$ in $\operatorname{im}(\iota_i)$} \\
            * & \text{otherwise}.
        \end{cases}
    \end{equation*}
    
    Note that $f_i$ is continuous and that its restriction to $\iota_i(C_\delta)$ is an embedding. For $1\leq i \leq n$, we define the pseudo-metric
    \begin{equation*}
        d_i : Y\times Y \rightarrow \R_{\geq 0} \qquad d_i(p,q) \coloneqq |f_i(p) - f_i(q)|,
    \end{equation*}
    where $|\cdot |$ denotes the standard norm on $\R^4$. Suppose that $p \neq q \in Y$ are two distinct points. There exists $i$ such that $p \in \iota_i(C_\delta)$. For this $i$, we must have $f_i(p) \neq f_i(q)$ which shows that
    \begin{equation*}
        d \coloneqq \sum\limits_{i=1}^n d_i
    \end{equation*}
    defines a metric on $Y$.

    Next, we show that every leaf $L$ of the characteristic foliation of $\Omega$ is rectifiable with respect to the metric $d$.  We recall the necessary preliminaries.  A continuous curve 
    \begin{equation*}
        \gamma : [a,b] \to Y
    \end{equation*}
    is called \emph{rectifiable} if its \emph{length} 
    \begin{equation}
    \label{eqn:length}
        \ell(\gamma) := \sup \left\{ \sum_{i=1}^{n} d(\gamma(t_{i-1}), \gamma(t_i)) \right\}
    \end{equation}
    is finite, where the supremum is taken over all partitions $a = t_0 < t_1 < \cdots < t_n = b$ of the interval $[a,b]$.  We say $\gamma : \R \rightarrow Y$ is rectifiable if the restriction of $\gamma$ to any closed interval is rectifiable. If $\gamma$ is rectifiable, then any reparametrization of it is rectifiable and has the same length.

    \medskip

    Now, consider a leaf $L$ of the foliation $\mathcal{F}$ and pick a continuous map $\gamma : \R \rightarrow L$ which is an orientation-preserving homeomorphism if $L$ is open and an orientation-preserving covering if $L$ is closed. The claim below shows that $\gamma$ is a rectifiable curve.
    
    \begin{claim}
    \label{claim:flow_on_C0Hamstructure_claim_rectifiable}
        For each compact interval $[a,b]\subset \R$, we have $\ell(\gamma|_{[a,b]}) < \infty$. Moreover, we have
        \begin{equation}
        \label{eq:flow_on_C0Hamstructrue_claim_limits}
            \lim_{b\rightarrow +\infty} \ell(\gamma|_{[a,b]}) = + \infty \qquad \text{and} \qquad \lim_{a\rightarrow - \infty} \ell(\gamma|_{[a,b]}) = + \infty.
        \end{equation}
    \end{claim}

    \begin{proof}
    For each $i$, we can write
    \begin{equation*}
        I_i \coloneqq \gamma^{-1}(\operatorname{im}(\iota_i)) = \bigcup_j I_i^j
    \end{equation*}
    where $(I_i^j)_j$ is a countable collection of pairwise disjoint finite open intervals $I_i^j\subset \R$. Every compact subset of $\R$ intersects only finitely many of the intervals $I_i^j$. Given a point $z \in D_{2\delta}$, let $\eta_z$ be a parametrization of the line segment $(0,1)\times \{z\} \subset C_{2\delta}$. Since $f$ is a smooth map, we have $\ell_{\operatorname{std}}(f\circ \eta_z) < +\infty$, where $\ell_{\operatorname{std}}$ stands for length measured with respect to the standard metric on $\R^4$. From this we can conclude that $\ell_i(\gamma|_{I_i^j}) < + \infty$, where $\ell_i$ denotes length with respect to the pseudo-metric $d_i$, defined analogously to \eqref{eqn:length} . Observe that if $J$ is an interval disjoint from $I_i$, then $\ell_i(\gamma|_J)=0$. Using the fact that $[a,b]$  intersects only finitely many of the intervals $I_i^j$, we conclude that $\ell(\gamma|_{[a,b]}) = \sum_i \ell_i(\gamma|_{[a,b]}) < + \infty$.

    Since the restriction of $f$ to $C_\delta$ is an embedding, there exists a positive number $c>0$ such that $\ell(f\circ \eta_z) > c$ for all $z \in D_\delta$. If $I_i^j$ is an interval such that $\gamma(I_i^j)\subset \iota_i(C_\delta)$, we then have $\ell_i(\gamma|_{I_i^j}) > c$. Since the sets $\iota_i(C_\delta)$ cover $Y$, every $t\in \R$ is contained in some interval $I_i^j$ with this property. The limits \eqref{eq:flow_on_C0Hamstructrue_claim_limits} are an immediate consequence. This concludes the proof of the claim.
    \end{proof}

    It follows from Claim \ref{claim:flow_on_C0Hamstructure_claim_rectifiable} that if $p\in Y$ is a point and $L$ is the leaf containing $p$, there exists a unique curve $\gamma_p : \R \rightarrow L\subset Y$ which is an orientation-preserving homeomorphism, or an orientation-preserving covering if $L$ is closed,  such that $\gamma_p$ is parametrized by arc length and $\gamma_p(0) = p$. As already mentioned, we can then define $\varphi(t,p) \coloneqq \gamma_p(t)$. It remains to check that $\varphi : \R \times M \rightarrow M$ is continuous.  A priori, we must also show that $\varphi(t, \cdot)$ is a homeomorphism for each $t \in \R$, however, this follows immediately from continuity as $\varphi(-t, \cdot)$ is the inverse of $\varphi(t, \cdot)$.
    
    Fix a point $p \in Y$ and an arbitrary  $T>0$. Our goal is to show that if $q\in Y$ is close to $p$, then $\gamma_q|_{[-T,T]}$ is $C^0$ close to $\gamma_p|_{[-T,T]}$. First note that if $q$ is sufficiently close to $p$ and $L_q$ denotes the leaf containing $q$, we can find a continuous map $\beta : \R \rightarrow L_q\subset Y$ which is an orientation-preserving covering of $L_q$ such that $\beta(0) = q$ and such that $\beta|_{[-T,T]}$ is $C^0$ close to $\gamma_p|_{[-T,T]}$. It then suffices to show that $\beta$ is close to being parametrized by arc length, i.e.\ that for every compact interval $[a,b]\subset [-T,T]$, the length $\ell(\beta|_{[a,b]})$ is close to $|b-a|$. We provide an outline of an argument proving this: For $i\in \{1,2\}$, let $\eta_i : [0,1] \rightarrow C_{2 \delta}$ be an embedded continuous line segment with image contained in $(0,1)\times \{z_i\}$ for $z_i \in D_{2 \delta}$. If $\eta_1$  and $\eta_2$ are $C^0$ close, then $\ell(f\circ \eta_1)$ and $\ell(f\circ \eta_2)$ are close. Note that for every segment of $\gamma_p$ traversing $\iota_i(C_{2_\delta})$ there is a corresponding $C^0$-close-by segment of $\beta$ traversing $\iota_i(C_{2\delta})$.  This implies that $\ell_i(\gamma_p|_{[a,b]})$ is close to $\ell_i(\beta|_{[a,b]})$ for all $i$ and for every interval $[a,b]\subset [-T,T]$. Thus $|b-a| = \ell(\gamma_p|_{[a,b]})$ and $\ell(\beta|_{[a,b]})$ are close as well. This concludes the proof of continuity of $\varphi$.
\end{proof}

\begin{rem}
\label{rem:flow_to_foliation_importance_dim_3}
    The assumption that Y has dimension 3 plays a crucial role in our proof of the first part of Proposition \ref{prop:flow_on_C0_foliations}, namely that the flow lines of $\varphi$ define a $C^0$ foliation. However, this dimensional restriction is not required for the second part of the proposition: in higher-dimensional manifolds, every oriented 1-dimensional $C^0$ foliation still admits a fixed-point-free flow tracing out its leaves.
    
    We now describe an example of a fixed-point-free topological flow $\psi$ on a topological $5$-manifold $Z$ which does not admit a local cross section homeomorphic to a $4$-dimensional ball $B^4$ everywhere. Start with a $3$-dimensional homology sphere $M$ which is not homeomorphic to the $3$-sphere $S^3$. Consider the suspension $SM = [-1,1]\times M/\sim$, where $\{\pm 1\}\times M$ are collapsed to points $p_\pm$. While this space is a homology manifold, it is not a topological manifold. The two points $p_\pm$ do not have neighborhoods homeomorphic to the $4$-dimensional ball. However, by the double suspension theorem \cite{can79}, the double suspension $S^2M$ is homeomorphic to $S^5$. This also implies that the product $Z \coloneqq S^1 \times SM$ is a topological $5$-manifold. Let $\psi$ be the fixed-point-free flow on $Z$ which rotates the $S^1$ factor and restricts to the identity on the $SM$ factor. Now observe that $\psi$ cannot have a local cross section homeomorphic to $B^4$ at any point contained in $S^1 \times \{p_\pm\}$. If it did, one would obtain a neighborhood of $p_\pm$ in $SM$ homeomorphic to $B^4$.
\end{rem}

\subsubsection*{The Oxtoby--Ulam theorem}

Let $X$ be a compact metric space and let $\mu$ be a \textit{Borel measure} on $X$ of finite mass, i.e.\ a measure $\mu$ defined on the Borel $\sigma$-algebra of $X$ satisfying $\mu(X)<+\infty$. The \textit{support} $\operatorname{supp}(\mu)$ of $\mu$ is the set of all points $x\in X$ such that every open neighborhood of $x$ has strictly positive measure. The measure $\mu$ is said to have \textit{full support} if $\operatorname{supp}(\mu) = X$. This is equivalent to requiring that the measure of any non-empty open subset is non-negative. Following \cite[Def. 2.15]{dgs76}, we say that $\mu$ is \textit{non-atomic} if $\mu(\{x\}) = 0$ for every point $x \in X$.

\begin{rem}
    An \textit{atom} of a general measure space $(X,\Sigma,\mu)$ is often defined to be a measurable set $A\in \Sigma$ such that $\mu(A)>0$ and such that every measurable set $B\subset A$ satisfies $\mu(B) = 0$ or $\mu(B) = \mu(A)$. A measure is then called \textit{non-atomic} if it does not have any atoms. A simple argument (see \cite[2.IV]{kno67}) shows that if $\mu$ is an inner and outer regular finite Borel measure, then $\mu$ is non-atomic in this sense if and only if it is non-atomic in the sense defined above, i.e. if and only if $\mu(\{x\}) = 0$ for all points $x$. Note that any finite Borel measure on a compact metric space is inner and outer regular \cite[Prop. 2.3]{dgs76}. This means that in all situations of interest to us in this paper, the two notions of being non-atomic agree.
\end{rem}

The following theorem was proved in \cite{ou41} by Oxtoby and Ulam, who also give credit to von Neumann for an independent and unpublished proof. Our formulation of the theorem follows \cite{fat80}.

\begin{thm}
\label{thm:oxtoby_ulam}
    Let $M^n$ be a compact connected topological manifold, possibly with boundary. For $i\in \{1,2\}$, let $\mu_i$ be a finite, non-atomic Borel measure on $M$ of full support. Moreover, assume that $\mu_i(\partial M) = 0$. Then, there exists a homeomorphism $f : M\rightarrow M$ such that $f_*\mu_1 = \mu_2$ if and only if $\mu_1(M) = \mu_2(M)$. If this is the case, the homeomorphism $f$ can be chosen such that $f|_{\partial M} = \operatorname{id}_{\partial M}$.
\end{thm}

\subsubsection*{Volume-preserving topological flows and measured $C^0$ foliations}

As above, let $\varphi$ be a fixed-point-free topological flow on a closed oriented topological $3$-manifold $Y$. Let $\mathcal{F}_\varphi$ be the $C^0$ foliation induced by $\varphi$; see Proposition \ref{prop:flow_on_C0_foliations}.

\begin{lem}
\label{lem:induced_transverse_measure}
    Suppose that $\mu$ is a finite Borel measure on $Y$ preserved by $\varphi$. Then $\varphi$ and $\mu$ induce a transverse measure $\Lambda_{\varphi,\mu}$ on $\mathcal{F}_\varphi$ characterized by the condition that, for every transversal $\Sigma \subset Y$, every Borel subset $A\subset \Sigma$, and every $\varepsilon>0$ sufficiently small, we have
    \begin{equation}
    \label{eq:induced_transverse_measure}
        \Lambda_{\varphi,\mu}(A) = \varepsilon^{-1}\mu (\{\varphi^t(p)\mid p\in A, 0<t<\varepsilon\}).
    \end{equation}
\end{lem}

\begin{proof}
    Pick $\varepsilon_0>0$ such that the restriction of $\varphi$ to $(-\varepsilon_0,\varepsilon_0)\times\Sigma$ is an embedding. Since $\mu$ is preserved by $\varphi$, the expression on the right hand side of \eqref{eq:induced_transverse_measure} is independent of $\varepsilon \in (0,\varepsilon_0)$. Therefore, \eqref{eq:induced_transverse_measure} yields a well-defined finite Borel measure on $\Sigma$. We need to show that maps between subsets of transversals obtained by sliding along leaves of $\mathcal{F}_\varphi$ are measure preserving. This can be reduced to showing that if $\Sigma\subset Y$ is a transversal, $A\subset \Sigma$ is a Borel subset, $\tau: A \rightarrow \R$ is a continuous function with small $L^\infty$ norm, and $\varepsilon>0$ is small, then
    \begin{equation*}
        \mu(\{\varphi^t(p)\mid p\in A, 0<t<\varepsilon\}) = \mu(\{\varphi^t(p) \mid p \in A, \tau(p) < t < \tau(p) + \varepsilon\}).
    \end{equation*}
    This identity follows from the assumption that $\varphi$ preserves $\mu$.
\end{proof}

Let us say that a transverse measure $\Lambda$ on some foliation $\mathcal{F}$ has \emph{full support} if the induced measure on any transversal has full support. Similarly, we say that $\Lambda$ is \emph{non-atomic} if the induced measure on any transversal is non-atomic.

Now suppose that $\varphi$ is a fixed-point-free topological flow preserving a finite Borel measure $\mu$. Let $(\mathcal{F}_\varphi,\Lambda_{\varphi,\mu})$ be the induced measured foliation; see Lemma \ref{lem:induced_transverse_measure}. It is straightforward to see that $\mu$ has full support if and only if the induced transverse measure $\Lambda_{\varphi,\mu}$ has full support. Since $\mu$ has finite mass and $\varphi$ does not have fixed points, the measure $\mu$ is automatically non-atomic. However, it is possible that the transverse measure $\Lambda_{\varphi,\mu}$ has atoms.

\begin{example}
\label{ex:atomic_transverse_measure}
    Consider $Y= S^3$ with the Hopf flow $\varphi$, i.e.\ the flow tracing out the fibers of the Hopf fibration at unit speed. Let $\mu_3$ be the standard $3$-dimensional Lebesgue measure on $S^3$, scaled to have total volume $1$. The Hopf flow $\varphi$ preserves $\mu_3$. Moreover, pick a Hopf fiber $F$ and let $\mu_1$ be the $1$-dimensional Lebesgue measure supported on this fiber $F$, again scaled to have total volume $1$. Then $\mu \coloneqq \frac{1}{2}(\mu_1 + \mu_3)$ is a finite Borel measure on $Y$ invariant under $\varphi$. The induced transverse measure $\Lambda_{\varphi,\mu}$ on the Hopf fibration $\mathcal{F}_\varphi$ is not non-atomic: any intersection point of a transversal with the special fiber $F$ yields an atom.

    In this example, the flow $\varphi$ is smooth, but the measure $\mu$ is not. We can turn the situation around using the Oxtoby-Ulam theorem: Since both $\mu$ and $\mu_3$ have full support, are non-atomic, and have total volume $1$, there exists a homeomorphism $f$ of $Y$ such that $f_*\mu = \mu_3$. Let $f_*\varphi$ be the pushforward of $\varphi$ under $f$. Then $f_*\varphi$ is a fixed-point-free topological flow on $Y$ which preserves the standard Lebesgue measure $\mu_3$. However, $f_*\varphi$ has a periodic orbit $f(F)$ of positive measure, causing the induced transverse measure $\Lambda_{f_*\varphi,\mu_3}$ on $\mathcal{F}_{f_*\varphi}$ to be not non-atomic.

    Note that we can produce even more pathological examples. Let $(F_n)_{n\geq 1}$ be a dense sequence of Hopf fibers. For each $n$, let $\mu_{F_n}$ denote the $1$-dimensional Lebesgue measure on $F_n$ of total volume $1$. Define the $\varphi$-invariant measure $\mu \coloneqq \sum_{n\geq 1} 2^{-n}\mu_{F_n}$. This measure has full support, is non-atomic, and has total volume $1$. Again by Oxtoby-Ulam, the measure $\mu$ is homeomorphic to the standard $3$-dimensional Lebesgue measure $\mu_3$, i.e.\ there exists a homeomorphism $g$ of Y such that $g_*\mu = \mu_3$. Then $g_*\varphi$ is a fixed-point-free topological flow on $Y$ which preserves $\mu_3$ and has the property that the complement of some countable sequence of periodic orbits has measure zero.
\end{example}

\begin{lem}
\label{lem:non-atomic}
    Suppose that $\mu$ is a finite Borel measure preserved by a fixed-point-free topological flow $\varphi$. Then the following statements are equivalent:
    \begin{enumerate}
        \item $\Lambda_{\varphi,\mu}$ is non-atomic.
        \item All flow lines of $\varphi$ have vanishing measure.
        \item All periodic orbits of $\varphi$ have vanishing measure.
    \end{enumerate}
\end{lem}

\begin{proof}
    If $\Sigma\subset Y$ is a transversal and $x\in \Sigma$ is a point, then $\Lambda_{\varphi,\mu}(\{x\})>0$ if and only if the flow line of $\varphi$ through $x$ has positive measure. This shows the equivalence of the first two statements.

    We argue that any flow line $L$ of $\varphi$ of positive measure is necessarily periodic. Indeed, if $L$ is a non-periodic flow line through a point $x$, then $L$ is the disjoint union of the countable collection of flow line segments $L_i\coloneqq \varphi([i,i+1)\times \{x\})$ for $i\in \Z$. All of these flow line segments have the same volume. Since $\mu$ is finite, the total volume of $L$ must be finite as well. But this implies that all flow line segments $L_i$ and therefore $L$ itself have vanishing volume.
\end{proof}

\subsubsection*{Compatible triples}

Consider a closed oriented topological $3$-manifold $Y$. Fix a finite Borel measure $\mu$ and a $C^0$ Hamiltonian structure $\Omega$ on $Y$. Let $\varphi$ denote a fixed-point-free topological flow on $Y$.  

The $C^0$ Hamiltonian structure $\Omega$ determines a $1$-dimensional foliation $\mathcal{F}_\Omega$ on $Y$ equipped with a transverse measure $\Lambda_\Omega$.  Moreover, $\Omega$ determines a coorientation of $\mathcal{F}_\Omega$, which in combination with the orientation on $Y$ determines an orientation $\mathcal{F}$.

\begin{lem}
\label{lem:compatible_triple_equivalent_conditions_C0}
    The following statements are equivalent:
    \begin{enumerate}
        \item \label{item:compatible_triple_equivalent_conditions_C0_one} The flow $\varphi$ preserves $\mu$ and  $(\mathcal{F}_\varphi,\Lambda_{\varphi,\mu})$ coincides with $ (\mathcal{F}_\Omega, \Lambda_\Omega)$, as measured, cooriented (hence, oriented) foliations. 
        \item \label{item:compatible_triple_equivalent_conditions_C0_two} The triple $(\varphi,\mu,\Omega)$ is locally homeomorphic to the triple $(\varphi_{\operatorname{std}},\mu_{\operatorname{std}},\omega_{\operatorname{std}}).$
    \end{enumerate}
\end{lem}

\begin{proof}
    Note that the triple $(\varphi_{\operatorname{std}},\mu_{\operatorname{std}},\omega_{\operatorname{std}})$ satisfies all conditions in statement \ref{item:compatible_triple_equivalent_conditions_C0_one}. Since all of these conditions can be checked locally, statement \ref{item:compatible_triple_equivalent_conditions_C0_two} implies statement \ref{item:compatible_triple_equivalent_conditions_C0_one}. Conversely, suppose that $(\varphi,\mu,\Omega)$ satisfies statement \ref{item:compatible_triple_equivalent_conditions_C0_one}. Every point in $Y$ has a neighborhood which can be parametrized via an orientation-preserving topological embedding of $C^0$ Hamiltonian structures
    \begin{equation*}
        \iota : ((-1,1)\times B(a),\omega_{\operatorname{std}}) \hookrightarrow (Y,\Omega).
    \end{equation*}
    Note that $\iota$ maps line segments $(-1,1)\times \{*\}$ into flow lines of $\varphi$ respecting orientation. For $\varepsilon>0$ sufficiently small, define the embedding
    \begin{equation*}
        \alpha : (-\varepsilon,\varepsilon)\times B(a) \rightarrow Y \qquad \alpha(t,p) \coloneqq \varphi^t(\iota(0,p)).
    \end{equation*}
    This embedding continues to pull back $\Omega$ to $\omega_{\operatorname{std}}$ and in addition pulls back $\varphi$ to $\varphi_{\operatorname{std}}$. We conclude that the transverse measure $\alpha^*\Lambda_{\varphi,\mu} = \Lambda_{\varphi_{\operatorname{std}},\alpha^*\mu}$ agrees with $\alpha^*\omega = \omega_{\operatorname{std}}$. Thus $\alpha^*\mu$ agrees with $\mu_{\operatorname{std}}$ on all sets of the form $I \times A$ where $I\subset (-\varepsilon,\varepsilon)$ is an interval and $A\subset B(a)$ is a Borel set. This determines $\alpha^*\mu = \mu_{\operatorname{std}}$ and we conclude that $\alpha$ is a local homeomorphism between $(\varphi_{\operatorname{std}},\mu_{\operatorname{std}},\omega_{\operatorname{std}})$ and $(\varphi,\mu,\omega)$.
\end{proof}

\begin{definition}
\label{def:compatible-triple-C0}
    We call $(\varphi,\mu,\Omega)$ a \textit{compatible triple} if the equivalent conditions in Lemma \ref{lem:compatible_triple_equivalent_conditions_C0} are satisfied.
\end{definition}

\medskip
We now formulate criteria guaranteeing that any of the pairs $(\varphi, \mu)$, $(\varphi, \Omega)$ or $(\mu, \Omega)$ extend to compatible triples.

\begin{prop}
\label{prop:from_pair_to_tiple_iff_loc_triv_C0}
    The pair $(\varphi,\mu)$ extends to a compatible triple if and only if it is locally homeomorphic to $(\varphi_{\operatorname{std}},\mu_{\operatorname{std}})$. The analogous statements hold for the pairs $(\varphi,\Omega)$ and $(\mu,\Omega)$.
\end{prop}
\begin{proof}
   This can be proven through an argument analogous to the smooth case presented in Proposition \ref{prop:compatible_triple_existence_extension_smooth}; we omit the details here. The proof of Proposition \ref{prop:compatible_triple_existence_extension_smooth} makes use of Lemma \ref{lem:compatible_triple_unique_extension_smooth}, whose $C^0$ counterpart is Lemma \ref{lem:unique_extension_to_triple_C0}, stated and proven below.
\end{proof}

\begin{lem}
\label{lem:unique_extension_to_triple_C0}
    Any two components of a compatible triple $(\varphi,\mu,\Omega)$ uniquely determine the third.
\end{lem}

\begin{proof}
    It suffices to check the above uniqueness locally. This means we need to verify that if $(\varphi,\mu,\Omega)$ is a compatible triple on (an open subset of) $\R^3$ two of whose components coincide with the corresponding components of $(\varphi_{\operatorname{std}},\mu_{\operatorname{std}},\omega_{\operatorname{std}})$, then  $(\varphi,\mu,\Omega) = (\varphi_{\operatorname{std}},\mu_{\operatorname{std}},\omega_{\operatorname{std}})$. In view of the characterization of compatibility given by statement \ref{item:compatible_triple_equivalent_conditions_C0_two}, it suffices to check that if $\psi$ is a homeomorphism of $\R^3$ preserving two members of the triple $(\varphi_{\operatorname{std}},\mu_{\operatorname{std}},\omega_{\operatorname{std}})$, then it preserves the third. This is elementary and we omit the details.
\end{proof}

\begin{prop}
\label{prop:from_pair_to_tiple_conditions_C0}
    \begin{enumerate}
        \item The pair $(\varphi,\mu)$ extends to a compatible triple if and only if the flow of $\varphi$ preserves $\mu$ and the flow lines of $\varphi$ have zero measure.
        \item The pair $(\varphi,\mu)$ extends to a compatible triple if and only if the flow of $\varphi$ preserves $\mu$ and the induced transverse measure $\Lambda_{\varphi,\mu}$ on $\mathcal{F}_\varphi$ is of full support and non-atomic.
        \item The pair $(\varphi,\Omega)$ extends to a compatible triple if and only if the oriented flow lines of $\varphi$ agree with the oriented characteristic leaves of $\Omega$.
    \end{enumerate}
\end{prop}

\begin{proof}
    The first and the second items are equivalent by Lemma \ref{lem:non-atomic}.  We will prove the second item.

    If the pair $(\varphi,\mu)$ extends to a compatible triple, then $(\mathcal{F}_\varphi,\Lambda_{\varphi,\mu})$ is locally homeomorphic to (the measured foliation of) $\omega_{\operatorname{std}}$, which clearly implies that $\Lambda_{\varphi,\mu}$ is of full support and non-atomic. Conversely, suppose that $\varphi$ preserves $\mu$ and $\Lambda_{\varphi,\mu}$ is of full support and non-atomic. Given an arbitrary point in $Y$, pick a local cross section $B\subset Y$ through that point which is homeomorphic to a disc; as mentioned earlier, the existence of the cross section $B$ was proven by Whitney \cite{whi38}. Since the measure on $B$ induced by $\Lambda_{\varphi,\mu}$ is of full support and non-atomic, it follows from the Oxtoby-Ulam theorem that there exists an area-preserving parametrization $\iota : (B(a),\omega_{\operatorname{std}}) \rightarrow (B,\Lambda_{\varphi,\mu})$. In addition, we can pick $\iota$ such that it is orientation preserving with respect to the orientation of $B$ induced by the flow $\varphi$ and the orientation of $Y$. Extend $\iota$ to an embedding $\alpha : (-\varepsilon,\varepsilon)\times B(a) \hookrightarrow Y$ via the flow $\varphi$. Clearly, $\alpha$ pulls back $\varphi$ to $\varphi_{\operatorname{std}}$. Moreover, it pulls back $(\mathcal{F}_\varphi,\Lambda_{\varphi,\mu})$ to $\omega_{\operatorname{std}}$. Hence $\Omega \coloneqq (\mathcal{F}_\varphi,\Lambda_{\varphi,\mu})$ equipped with the orientation and coorientation induced by $\varphi$ and the ambient orientation of $Y$ is a $C^0$ Hamiltonian structure extending $(\varphi,\mu)$ to a compatible triple. We have proven the second statement.

    We turn now  to the third statement. If $(\varphi,\Omega)$ extends to a compatible triple, then clearly the oriented flow lines of $\varphi$ agree with the oriented characteristic leaves of $\Omega$ by the characterization of compatibility given in statement \ref{item:compatible_triple_equivalent_conditions_C0_one} in Lemma \ref{lem:compatible_triple_equivalent_conditions_C0}. Conversely, suppose that the oriented flow lines of $\varphi$ agree with the oriented characteristic leaves of $\Omega$. As we saw in the proof of Lemma \ref{lem:compatible_triple_equivalent_conditions_C0}, we can parametrize a neighborhood of any point in $Y$ via an embedding of $C^0$ Hamiltonian structures $\alpha : ((-\varepsilon,\varepsilon)\times B(a),\omega_{\operatorname{std}}) \hookrightarrow (Y,\Omega)$ that also pulls back the flow $\varphi$ to $\varphi_{\operatorname{std}}$. By Proposition \ref{prop:from_pair_to_tiple_iff_loc_triv_C0}, this means that $(\varphi,\Omega)$ extends to a compatible triple.
    \end{proof}

\begin{rem}\label{rem:every-Ham-sturcture-extend-to-triple}
We proved in Proposition \ref{prop:flow_on_C0_foliations} that every oriented $C^0$ foliation admits a fixed-point free flow tracing out its leaves.  Hence, as a consequence of the second item in Proposition \ref{prop:from_pair_to_tiple_conditions_C0}, every $C^0$ Hamiltonian structure $\Omega$ is a component of a compatible tripe $(\varphi, \mu, \Omega)$.
\end{rem}

\subsection{Mass flow and flux}
\label{sec:mass flow flux}

In this section, we briefly review the mass flow homomorphism and show that for a volume-preserving flow $\varphi$ which is a component of a compatible triple $(\varphi, \mu, \Omega)$, the flux $\overline{\operatorname{Flux}}(\Omega)$ admits a description in terms of mass flow.

\medskip
 
Consider a closed and oriented topological manifold $M$ of dimension $n$. Suppose that $\mu$ is a finite non-atomic Borel measure on $M$ of full support. Recall that there is a \textit{mass flow homomorphism}, introduced in \cite{fat80,sch57},
\begin{equation*}
    \widetilde{\theta} : \widetilde{\operatorname{Homeo}}_0(M,\mu) \rightarrow H_1(M;\R) \cong H^{n-1}(M;\R),
\end{equation*}
where $\widetilde{\operatorname{Homeo}}_0(M,\mu)$ denotes the universal cover of the identity component of the group of volume-preserving homeomorphisms of $(M,\mu)$ and the identification $ H_1(M;\R) \cong H^{n-1}(M;\R)$ is given by Poincar\'{e} duality. 

We briefly recall the definition of $\widetilde{\theta}$. Let $\T\coloneqq \R/\Z$ denote the circle. Then we have an isomorphism $H^1(M;\Z) \cong [M,\T]$ and therefore $H_1(M;\R) \cong \operatorname{Hom}([M,\T],\R)$. Given an element $\widetilde{\varphi} \in \widetilde{\operatorname{Homeo}}_0(M,\mu)$ represented by a volume-preserving isotopy $(\varphi_t)_{t\in [0,1]}$ starting at the identity, defining $\widetilde{\theta}(\widetilde{\varphi})$ therefore amounts to defining a group homomorphism $$\widetilde{\theta}(\widetilde{\varphi}) : [M,\T]\rightarrow \R.$$ Let $[f] \in [M,\T]$ be a homotopy class represented by a map $f: M \rightarrow \T$ and consider the homotopy of maps $f\circ \varphi^t - f : Y \rightarrow \T$, which at $t=0$ coincides with the constant map zero. There exists a unique lift to a homotopy of maps $\overline{f\circ \varphi^t - f} : Y \rightarrow \R$, which at $t=0$ coincides with the constant map zero. Then,
\begin{equation}
\label{eq:construction_mass_flow}
    \widetilde{\theta}(\widetilde{\varphi}) ([f]) \coloneqq \int_Y \overline{f\circ \varphi_1 - f} d\mu,
\end{equation}
which turns out to be independent of the choice of representative of $\tilde{\varphi}$ and $[f]$.  This defines  $\widetilde{\theta} : \widetilde{\operatorname{Homeo}}_0(M,\mu) \rightarrow H_1(M;\R)$, which we may view as taking values in $H^{n-1}(M;\R)$, via Poincar\'{e} duality.

\begin{definition}\label{def:exact_vol_pres}
   Let ${\varphi^t}_{t \in \R} \subset \operatorname{Homeo}_0(M,\mu)$ be a flow.  We say $\varphi^t$ is exact if the element of $\widetilde{\operatorname{Homeo}}_0(M,\mu)$ determined by ${\varphi^t}_{t \in [0,1]}$  has vanishing mass flow.
\end{definition}

We now restrict our attention to the case where $M$ is a closed, oriented 3-manifold, which we will refer to from this point onward as $Y$.
\begin{prop}
\label{prop:equivalence_mass_flow_flux}
    Suppose that $(\varphi,\mu,\Omega)$ is a compatible triple. Then, for every $T\in \R\setminus \{0\}$, we have
    \begin{equation}
    \label{eq:equivalence_mass_flow_flux}
       \frac{1}{T} \widetilde{\theta}([(\varphi^t)_{t\in [0,T]}]) = \overline{\operatorname{Flux}}(\Omega).
    \end{equation}
\end{prop}
\begin{proof}
   We begin by proving the identity \eqref{eq:equivalence_mass_flow_flux} for smooth compatible triples $(\varphi,\mu,\omega)$. Let $f:Y\rightarrow \T$ be a smooth map.  Throughout this proof, we will view  $\widetilde{\theta}$ as mapping into $H_1(Y;\R)$ which we identify with 
   $\operatorname{Hom}([Y,\Z],\R)$.   Viewing the right hand side of identity \eqref{eq:equivalence_mass_flow_flux} as a homomorphism $[Y,\T]\rightarrow \R$, we compute
    \begin{align*}
        T^{-1}\widetilde{\theta}([(\varphi^t)_{t\in [0,T]}]) ([f])
        &= T^{-1}\int_Y \overline{f\circ \varphi^T - f} \, \mu \\
        &= T^{-1}\int_Y \left( \int_0^T (\iota_Xdf)\circ \varphi^t \enspace dt \right) \mu \\
        &= T^{-1}\int_0^T \left(\int_Y (\iota_X df) \circ \varphi^t \enspace \mu\right) dt \\
        &= \int_Y (\iota_Xdf) \mu \\
        &= \int_Y df \wedge \iota_X\mu  \\
        &= \int_Y df \wedge \omega \\
        &= \langle \, [df] \cup \operatorname{Flux}(\omega), [Y] \, \rangle,
    \end{align*}
    where $\langle \cdot, \cdot \rangle$ stands for the pairing of cohomology and homology.
    Here, the first identity uses the definition of the mass flow homomorphism \eqref{eq:construction_mass_flow}. The fourth equality uses that $\varphi^t$ preserves $\mu$, which implies that $\int_Y (\iota_Xdf)\circ \varphi^t \, \mu$ is independent of $t$. The fifth equality holds because $df\wedge \mu$ being a $4$-form necessarily vanishes and hence we have
    \begin{equation*}
        0 = \iota_X (df \wedge \mu) = (\iota_Xdf)\mu - df \wedge \iota_X\mu.
    \end{equation*}
    The sixth equality uses that $(\varphi,\mu,\omega)$ is a compatible triple. In the final equality, $[df] \in H^1(Y;\R)$ is the de Rham cohomology class represented by the $1$-form $df$. This class agrees with the image of the cohomology class represented by the circle map $f$ under the natural map $H^1(Y;\Z) \rightarrow H^1(Y;\R)$. Using the identifications
    \begin{equation*}
        \operatorname{Hom}([Y,\Z],\R) \cong \operatorname{Hom}(H^1(Y;\Z),\R) \cong H_1(Y;R) \cong H^2(Y;\R),
    \end{equation*}
    this shows identity \eqref{eq:equivalence_mass_flow_flux} in the smooth case.

    In the general case, where $(\varphi,\mu,\Omega)$ is not assumed to be smooth, we make use of the following claim. Its proof closely parallels that of Theorem \ref{thm:from_C0_to_smooth_plus_plug} and is therefore omitted. 

\begin{claim}
\label{cl:from_C0_to_smooth_plus_plug_flow_version}
     There exist 
     \begin{enumerate}
         \item a smooth compatible triple $(\varphi_0,\mu_0,\omega_0)$ on a smooth, closed and oriented $3$-manifold $Y_0$,
         \item an open surface with area form $(\Sigma,\omega_\Sigma)$ consisting of finitely many disc components, and a smooth embedding $\iota : ((0,\varepsilon) \times \Sigma,\omega_\Sigma) \hookrightarrow (Y_0,\omega_0)$ satisfying $\varphi_0^s(\iota(t,p)) = \iota(t+s,p)$ for all $-t<s<\varepsilon -t$,
         \item  an isotopy $(\psi_t)_{t\in [0,\varepsilon]}$ in $\overline{\operatorname{Ham}}(\Sigma)$,
     \end{enumerate} 
     such that the compatible triple $(\varphi_1,\mu_0,\Omega_1)$ obtained from $(\varphi_0,\mu_0,\omega_0)$ by inserting\footnote{Here, by inserting the isotopy $(\psi_t)_t$ into $(\varphi_0,\mu_0,\omega_0)$ we mean modifying the compatible triple $(\varphi_0,\mu_0,\omega_0)$ in a manner analogous to the plug insertion operation, introduced in Section \ref{subsec:C0_plugs}.  We alluded to this in relation to Equation \eqref{eq:helicity_and_calabi_flows}.} the isotopy $(\psi_t)_t$ via the embedding $\iota$ is homeomorphic to $(\varphi,\mu,\Omega)$. 
\end{claim}

Given an arbitrary homotopy class in $[Y_0,\T]$, we may represent it by a smooth map $f: Y_0 \rightarrow \T$ with the property that $f\circ \iota (t,p) = t$. This makes use of the fact $\Sigma$ consists of finitely many disc components. For $T>0$ sufficiently small, we then have
\begin{equation*}
    \overline{f\circ \varphi_0^T - f} = \overline{f\circ \varphi_1^T - f}.
\end{equation*}
This implies that
\begin{equation*}
    \widetilde{\theta}([(\varphi_0^t)_{t\in [0,T]}]) = \widetilde{\theta}([(\varphi_1^t)_{t\in [0,T]}])
\end{equation*}
for all $T>0$ sufficiently small. Because mass flow is a homomorphism, we conclude that this identity holds for all $T>0$. Since $(\varphi_0,\mu_0,\omega_0)$ is a smooth compatible triple, we have
\begin{equation*}
    T^{-1}\widetilde{\theta}([(\varphi_0^t)_{t\in [0,T]}]) = \operatorname{Flux}(\omega_0)
\end{equation*}
by the smooth case treated above. It follows from Definition \ref{def:flux_C0}, which defines $\overline{\operatorname{Flux}}$, that 
\begin{equation*}
    \overline{\operatorname{Flux}}(\Omega_1) = \operatorname{Flux}(\omega_0).
\end{equation*}
Combining the above identities, we conclude that
\begin{equation*}
    T^{-1}\widetilde{\theta}([(\varphi_1^t)_{t\in [0,T]}]) = \overline{\operatorname{Flux}}(\Omega_1).
\end{equation*}
Since $(\varphi,\mu,\Omega)$ is homeomorphic to $(\varphi_1,\mu_1,\Omega_1)$, this implies that
\begin{equation*}
    T^{-1}\widetilde{\theta}([(\varphi^t)_{t\in [0,T]}]) = \overline{\operatorname{Flux}}(\Omega).
\end{equation*}
\end{proof}

\subsection{Proof of Theorem \ref{thm:topological_invariance_helicity_flows}}
\label{sec:proof-thm-flows}

We will now deduce Theorem \ref{thm:topological_invariance_helicity_flows} from  Theorem \ref{thm:coherent_helicity_extension} and the content of Section \ref{sec:c0_case}.

\medskip

Let $(Y, \mu)$ be a closed $3$-manifold equipped with a volume form $\mu$.  We will assume throughout this section that $Y$ is equipped with the orientation induced by $\mu$.

We denote by $\mathcal{S}$ the set consisting of all volume-preserving topological flows $\varphi$ which are fixed-point-free and whose flow lines have vanishing measure. Let \( \mathcal{S}_0 \subset \mathcal{S} \) be the subset consisting of exact flows (see Definition~\ref{def:exact_vol_pres}); in other words, \( \mathcal{S}_0 \) consists of flows satisfying the assumptions of Theorem~\ref{thm:topological_invariance_helicity_flows}.  
Similarly, let \( \mathcal{T} \) denote the set of all \( C^0 \) Hamiltonian structures on \( Y \), and let \( \mathcal{T}_0 \subset \mathcal{T} \) denote the subset of exact  \( C^0 \) Hamiltonian structures (see Definition~\ref{def:flux_C0}).

Take any $\varphi \in \mathcal{S}$. Then, by Proposition \ref{prop:from_pair_to_tiple_conditions_C0} and Lemma \ref{lem:compatible_triple_equivalent_conditions_C0}, there exists a unique $C^0$ Hamiltonian structure $\Omega_\varphi$ such that the triple $(\varphi, \mu, \Omega_\varphi)$ is a compatible triple in the sense of Definition \ref{def:compatible-triple-C0}.  This yields a mapping from $\mathcal{S}$ to $\mathcal{T}$, which preserves mass flow/flux by Proposition \ref{prop:equivalence_mass_flow_flux}. Hence it induces a mapping from $\mathcal{S}_0$ to $\mathcal{T}_0$ 

The above yields a $\mathcal{R}$-valued extension of helicity, which we continue to denote by $\overline{\mathcal{H}}$: for $\varphi \in \mathcal{S}_0$, we simply define
\begin{equation*}
    \overline{\mathcal{H}}(\varphi) := \overline{\mathcal{H}}(\Omega_\varphi),
\end{equation*}
where the right-hand-side in the above is the universal $\mathcal{R}$-valued extension of helicity given by Theorem \ref{thm:coherent_helicity_extension}. The three properties listed in Theorem \ref{thm:coherent_helicity_extension} translate as follows for this $\mathcal{R}$-valued extension of helicity $\overline{\mathcal{H}}: \mathcal{S}_0 \rightarrow \mathcal{R}$.

\begin{enumerate}
        \item Extension:  If $\varphi$ is smooth, then 
         \begin{equation*}
             \overline{\mathcal{H}}(\varphi) = \mathcal{H}(\varphi) \in \R,
        \end{equation*}
        where we view $\R$ as a subgroup of $\mathcal{R}$ via the natural inclusion \eqref{eqn:subgroup}.
        \item Conjugation invariance: we have 
        \begin{equation*}
            \overline{\mathcal{H}}(f \varphi f^{-1}) = \overline{\mathcal{H}}(\varphi),
        \end{equation*}
        for any orientation- and volume-preserving homeomorphism $f$.
        \item  Calabi Compatibility: For every plug $\mathcal{P} = (\Sigma,\omega_\Sigma,\alpha,\psi^t)$,  we have
        \begin{equation*}
            \overline{\mathcal{H}}(\varphi \# \mathcal{P}) = \overline{\mathcal{H}}(\varphi) + \overline{\mathrm{Cal}}_\Sigma(\psi^1),
        \end{equation*}
        where  $\varphi \# \mathcal{P}$ refers to the analogue of the plug-insertion operation for flows, which we alluded to in relation to Equation \eqref{eq:helicity_and_calabi_flows}.
    \end{enumerate}
    
Moreover, as in Theorem \ref{thm:coherent_helicity_extension}, $\overline{\mathcal{H}}: \mathcal{S}_0 \rightarrow \mathcal{R}$ is uniquely determined by the above properties. Indeed, by Claim \ref{cl:from_C0_to_smooth_plus_plug_flow_version} the flow $\varphi$ is topologically conjugate to a flow of the form $\psi \# \mathcal{P}$ where $\psi$ is a smooth exact volume-preserving flow. Applying Theorem \ref{thm:from_C0_to_smooth_plus_plug} and arguing as we did in the proof of Theorem \ref{thm:coherent_helicity_extension}, one can deduce that any other $\mathcal{R}$-valued extension of helicity to $\mathcal{S}_0$, satisfying the above three properties, must coincide with $\overline{\mathcal{H}}$. 

Now, to obtain a real-valued extension of helicity, as stated in Theorem \ref{thm:topological_invariance_helicity_flows}, simply pick $\operatorname{pr}: \mathcal{R}\rightarrow \R$ to be a choice of projection.   Then, the composition $\operatorname{pr}\circ \overline{\mathcal{H}}$ is an $\R$-valued extension of helicity which satisfies the three properties stated in Theorem \ref{thm:topological_invariance_helicity_flows}. The Calabi compatibility property holds with respect to the real-valued extension of Calabi $\operatorname{pr}\circ \overline{\mathrm{Cal}}_\Sigma : \overline{\Ham}(\Sigma) \rightarrow \R$. The uniqueness part of the statement may be proven using Claim \ref{cl:from_C0_to_smooth_plus_plug_flow_version}, as in the previous paragraph.
   
This completes the proof of Theorem \ref{thm:topological_invariance_helicity_flows}.\qed

\bibliographystyle{plain}
\bibliography{refs}

\medskip
\noindent Oliver Edtmair\\
\noindent ETH-ITS, ETH Z\"{u}rich, Scheuchzerstrasse 70, 8006 Z\"{u}rich, Switzerland.\\
{\it e-mail:} oliver.edtmair@eth-its.ethz.ch

\medskip
\noindent Sobhan Seyfaddini\\
\noindent D-MATH, ETH Z\"{u}rich, R\"{a}mistrasse 101, 8092 Z\"{u}rich, Switzerland.\\
{\it e-mail:} sobhan.seyfaddini@math.ethz.ch

\end{document}